\documentclass[11pt]{amsart}

\usepackage{
  amsmath,
  amsfonts,
  amssymb,
  amsthm,
  amscd,
  gensymb,      
  graphicx,
  comment,      
  etoolbox,     
  mathtools,    
  enumitem,
  mathdots,     
  stmaryrd,     
  txfonts,      
  booktabs,     
  stackrel,     
}
\usepackage[usenames,dvipsnames]{xcolor}

\makeatletter
\@namedef{subjclassname@2020}{%
  \textup{2020} Mathematics Subject Classification}
\makeatother

\usepackage[T1]{fontenc}
\usepackage{bbm}                     

\usepackage[colorlinks=true, pdfstartview=FitV, linkcolor=blue, citecolor=blue, urlcolor=blue, breaklinks=true]{hyperref}

\DeclareFontFamily{OT1}{pzc}{}
\DeclareFontShape{OT1}{pzc}{m}{it}{<-> s * [1.10] pzcmi7t}{}
\DeclareMathAlphabet{\mathpzc}{OT1}{pzc}{m}{it}

\usepackage[capitalize]{cleveref}   

\crefname{conj}{Conjecture}{Conjectures}
\crefname{cor}{Corollary}{Corollaries}
\crefname{defin}{Definition}{Definitions}
\crefname{eg}{Example}{Examples}
\crefname{lem}{Lemma}{Lemmas}
\crefname{theo}{Theorem}{Theorems}
\crefname{equation}{}{}
\crefname{enumi}{}{}

\renewcommand{\descriptionlabel}[1]{\hspace{\labelsep}(#1)}

\makeatletter
\let\orgdescriptionlabel\descriptionlabel
\renewcommand*{\descriptionlabel}[1]{%
  \let\orglabel\label
  \let\label\@gobble
  \phantomsection
  \edef\@currentlabel{#1}%
  \let\label\orglabel
  \orgdescriptionlabel{#1}%
}
\makeatother

\usepackage{tikz}
\usetikzlibrary{arrows.meta}
\usetikzlibrary{decorations.markings}

\tikzset{anchorbase/.style={>=To,baseline={([yshift=-0.5ex]current bounding box.center)}}}
\tikzset{centerbase/.style={>=To,baseline={(0,-0.1)}}}
\tikzset{wipe/.style={white,line width=4pt}}

\tikzset{->-/.style={decoration={
  markings,
  mark=at position #1 with {\arrow{>}}},postaction={decorate}}}
\tikzset{-<-/.style={decoration={
  markings,
  mark=at position #1 with {\arrow{<}}},postaction={decorate}}}

\newcommand{\braidto}{to[out=up,in=down]}

\newcommand\bubright[3]{
  \draw[->] (#1)++(0.2,0) arc(0:-360:0.2);
  \filldraw[blue] (#1)++(-0.18,.1) circle (1.5pt) node[anchor=east] {\dotlabel{#2}};
  \filldraw[fill=white, draw=red] (#1)++(-0.18,-0.1) circle (1.5pt) node[anchor=east] {{\color{red} \dotlabel{#3}}}
}

\newcommand\bubleft[3]{
  \draw[->] (#1)++(-0.2,0) arc(180:540:0.2);
  \filldraw[blue] (#1)++(0.18,-0.1) circle (1.5pt) node[anchor=west] {\dotlabel{#2}};
  \filldraw[fill=white, draw=red] (#1)++(0.18,.1) circle (1.5pt) node[anchor=west] {{\color{red} \dotlabel{#3}}}
}

\newcommand\cbubble[2]{
  \begin{tikzpicture}[anchorbase]
    \bubright{0,0}{#1}{#2};
  \end{tikzpicture}
}

\newcommand\ccbubble[2]{
  \begin{tikzpicture}[anchorbase]
    \bubleft{0,0}{#1}{#2};
  \end{tikzpicture}
}

\newcommand\dotlabel[1]{$\scriptstyle{#1}$}

\newcommand\singdot[1]{
    \filldraw[fill=white, draw=red] (#1) circle (1.5pt)
}
\newcommand\multdot[3]{
    \singdot{#2} node[anchor=#1] {\color{red} \dotlabel{#3}}
}
\newcommand\token[3]{
  \filldraw[blue] (#2) circle (1.5pt) node[anchor=#1] {\dotlabel{#3}}
}
\newcommand\teleport[2]{
  \draw[blue] (#1) to (#2);
  \filldraw[blue] (#1) circle (1.5pt);
  \filldraw[blue] (#2) circle (1.5pt)
}

\newcommand\Z{\mathbb{Z}}

\newcommand\N{\mathbb{N}}
\newcommand\kk{\Bbbk}

\newcommand\cC{\mathcal{C}}
\newcommand\cL{\mathcal{L}}
\newcommand\hcL{\hat{\mathcal{L}}}
\newcommand\cR{\mathcal{R}}
\newcommand\AWC[1][A]{\mathpzc{Wr}^\textup{aff}(#1)}    
\newcommand\AWA{\mathrm{Wr}^\textup{aff}}               
\newcommand\cW{\mathcal{W}}

\newcommand\Heis{\mathpzc{Heis}}                        

\newcommand\ba{\mathbf{a}}

\newcommand\bB{\mathbf{B}}

\newcommand\one{\mathbbm{1}}

\newcommand\rW{\mathrm{W}}              

\newcommand\diag{\textup{diag}}
\newcommand\op{\textup{op}}             
\newcommand\rev{\textup{rev}}
\newcommand\teacup{\text{cup}}
\newcommand\teacap{\text{cap}}

\newcommand\fD{\mathfrak{D}}
\newcommand\fW{\mathfrak{W}}            

\newcommand\cocenter[1]{\left\langle #1 \right\rangle}     

\DeclareMathOperator{\bub}{bub}
\DeclareMathOperator{\End}{End}
\DeclareMathOperator{\Hom}{Hom}
\DeclareMathOperator{\id}{id}
\DeclareMathOperator{\im}{im}       
\DeclareMathOperator{\Kar}{Kar}     
\DeclareMathOperator{\ord}{ord}
\DeclareMathOperator{\rank}{rank}
\DeclareMathOperator{\Span}{Span}
\DeclareMathOperator{\Sym}{Sym}
\DeclareMathOperator{\tr}{tr}
\DeclareMathOperator{\Tr}{Tr}

\newtheorem{theo}{Theorem}[section]
\newtheorem{prop}[theo]{Proposition}
\newtheorem{lem}[theo]{Lemma}
\newtheorem*{lem*}{Lemma}
\newtheorem{cor}[theo]{Corollary}
\newtheorem{conj}[theo]{Conjecture}

\theoremstyle{definition}
\newtheorem{defin}[theo]{Definition}
\newtheorem{rem}[theo]{Remark}
\newtheorem{eg}[theo]{Example}

\numberwithin{equation}{section}
\allowdisplaybreaks

\setenumerate[1]{label=(\alph*)}  
\setenumerate[2]{label=(\roman*)}

\newtoggle{comments}
\newtoggle{details}
\newtoggle{detailsnote}

\toggletrue{detailsnote}   

\iftoggle{comments}{%
  \newcommand{\acomments}[1]{
    \ \\
    {\color{red}
      \textbf{AS:} #1
    }
    \ \\
    }
  \newcommand{\mcomments}[1]{
    \ \\
    {\color{blue}
      \textbf{MR:} #1
    }
    \ \\
    }
}{%
  \newcommand{\acomments}[1]{}
  \newcommand{\mcomments}[1]{}
}

\iftoggle{details}{%
  \newcommand{\details}[1]{
      \ \\
      {\color{OliveGreen}
        \textbf{Details:} #1
      }
      \\
  }
}{%
  \newcommand{\details}[1]{}
}

\begin{document}

\title{Frobenius W-algebras and traces of Frobenius Heisenberg categories}

\author{Michael Reeks}
\address[M.R.]{
  Department of Mathematics, Bucknell Univeristy\\
  Lewisburg, PA, U.S.A.
}
\email{m.reeks@bucknell.edu}

\author{Alistair Savage}
\address[A.S.]{
  Department of Mathematics and Statistics \\
  University of Ottawa \\
  Ottawa, ON, Canada
}
\urladdr{\href{https://alistairsavage.ca}{alistairsavage.ca}, \textrm{\textit{ORCiD}:} \href{https://orcid.org/0000-0002-2859-0239}{orcid.org/0000-0002-2859-0239}}
\email{alistair.savage@uottawa.ca}


\begin{abstract}
  To each symmetric graded Frobenius superalgebra we associate a W-algebra.  We then define a linear isomorphism between the trace of the Frobenius Heisenberg category and a central reduction of this W-algebra.  We conjecture that this is an isomorphism of graded superalgebras.
\end{abstract}


\ifboolexpr{togl{comments} or togl{details}}{%
  {\color{magenta}DETAILS OR COMMENTS ON}
}{%
}

\maketitle
\thispagestyle{empty}

\tableofcontents

\section{Introduction}

The Heisenberg category (of central charge $-1$) was first introduced by Khovanov \cite{Kho11} as a powerful tool for studying the representation theory of the symmetric group.  It was conjectured by Khovanov, and then proved in \cite{BSW-K0}, that the Grothendieck ring of Khovanov's Heisenberg category is isomorphic to the Heisenberg algebra (of central charge $-1$).  Since then, the definition of the Heisenberg category has been generalized to arbitrary central charge \cite{MS18,Bru18} and with the incorporation of a Frobenius algebra \cite{RS17,Sav19,BSW20}.  Namely, to each graded Frobenius superalgebra $A$ and central charge $k \in \Z$, one can define a \emph{Frobenius Heisenberg category} $\Heis_k(A)$.  The Grothendieck ring of $\Heis_k(A)$ was shown in \cite{BSW20} to be isomorphic to a lattice Heisenberg algebra, with lattice coming from the Grothendieck group of the category of projective modules of the Frobenius algebra.

The Grothendieck ring provides one method of decategorification of a linear monoidal category; another is the \emph{trace}, or \emph{zeroth Hochschild homology}.  The trace of Khovanov's original Heisenberg category was computed in \cite{CLLS18}, and shown to be isomorphic to a central reduction of the universal enveloping algebra of the W-algebra $\fW_{1+\infty}$, which is the unique central extension of the Lie algebra of regular differential operators on the circle.   When $A$ is the two-dimensional Clifford superalgebra, the trace of the even part of $\Heis_{-1}(A)$ was computed in \cite{OR17}, where it is shown to be isomorphic to a subalgebra of $\fW_{1+\infty}$.

The goal of the current paper is to explore the traces $\Tr(\Heis_k(A))$ of general Frobenius Heisenberg categories.  To each symmetric graded Frobenius superalgebra $A$ we associate a W-algebra $\fW(A)$.  This can be viewed as a generalization of $\fW_{1+\infty}$ in the sense that when $A$ is the ground field $\kk$, we have $\fW(\kk) \cong \fW_{1+\infty}$ (\cref{rocks}).  The general definition of $\fW(A)$ involves $A$ in a highly nontrivial way, and does not seem to have appeared in the literature before.  However, if $A$ is either semisimple or has a nontrivial positive grading, then the presentation of $\fW(A)$ becomes much simpler; see \cref{Wangmatch,kapzero}.  The algebra $\fW(A)$ contains a Heisenberg algebra depending on the cocenter of $A$ (as opposed to the lattice Heisenberg algebra appearing in the Grothendieck ring of $\Heis_k(A)$) and a ``Frobenius Virasoro algebra''.

The main result of the current paper (\cref{linisom}) is that we have a linear isomorphism
\begin{equation} \label{cougar}
    \rW(A)/(C-k) \xrightarrow{\cong} \Tr(\Heis_k(A)),
\end{equation}
where $\rW(A)$ is the universal enveloping algebra of $\fW(A)$, and $C$ is a certain central element.  The key ingredient in our proof is to use the action of $\Tr(\Heis_k(A))$ on the center of the category $\Heis_k(A)$, together with the basis theorem \cite[Th.~7.2]{BSW20} for $\Heis_k(A)$.  In particular, this action gives a new way of computing the cocenter of the degenerate affine Hecke algebra, in contrast to the methods of Solleveld \cite{Sol10}, whose work was used in \cite{CLLS18}.  See \cref{Solleveld} for details.

We conjecture that the map \cref{cougar} is an isomorphism of graded superalgebras (\cref{hope}).  When $k=-1$ and $A = \kk$, this essentially corresponds to the main result of \cite{CLLS18}.  The difficulty in proving the general conjecture is that the necessary diagrammatic computations become quite involved.  In the special case handled in \cite{CLLS18}, the authors are able to avoid most of these direct computations by utilizing the fact that $\rW(\kk)$ has a fairly small set of generators, and then computing the action of these generators on the faithful Fock space representation.  Unfortunately, these methods do not work in the general case; see \cref{special} for further discussion.  Nevertheless, we provide some evidence for the conjecture in \cref{compute}, where we compute some commutation relations in $\Tr(\Heis_k(A))$ and see that they match with those of $\fW(A)$.

When the graded Frobenius superalgebra $A$ is a \emph{zigzag algebra}, the category $\Heis_{-1}(A)$ was shown in \cite{CL12} to act on derived categories of coherent sheaves on Hilbert schemes of points on ALE spaces.  It follows that $\Tr(\Heis_k(A))$ acts on the zeroth Hochschild homology of this derived category.  Thus, it would follow from \cref{hope} that $W(A)$ acts here.  For more general values of $k$, this action should be related to the AGT correspondence for moduli spaces of instantons on resolutions of Kleinian singularities.  This potential application is one of the motivations for the current work.

Quantum versions of Heisenberg categories have been introduced in \cite{LS13,BSW-qheis}.  The trace of the category defined in \cite{LS13} was computed in \cite{CLLSS18}, where it was shown to be isomorphic to half of a central extension of a certain specialization of the elliptic Hall algebra.  We expect that the trace of the larger quantum Heisenberg category of \cite{BSW-qheis} should yield the full central extension.  (This is work in progress.)  More generally, it would be interesting to explore the trace of the quantum Frobenius Heisenberg categories introduced in \cite{BSW-qFrobHeis}, which are built from the quantum affine wreath product algebras of \cite{RS19}.  These should suggest a definition of elliptic Hall algebras depending on a Frobenius algebra.

\subsection*{Conventions}

Throughout the document, we fix a field $\kk$ of characteristic zero.  Unadorned tensor products should always be interpreted as tensor products over $\kk$, and \emph{algebras} are associative $\kk$-algebras unless otherwise specified.  The term \emph{graded} means $\Z$-graded.  We will generally consider graded superalgebras.  These are $(\Z \times \Z_2)$-graded algebras with super structure (e.g.\ in tensor products of algebras) coming from the $\Z_2$-grading.  For a homogeneous element $a$ of a graded superalgebra, we let $\deg(a) \in \Z$ denote its degree and let $\bar{a} \in \Z_2$ denote its parity.  If $\bar{a} = 0$, we say $a$ is \emph{even}; if $\bar{a}=1$, we say $a$ is \emph{odd}.  We say that the grading is \emph{positive} if the negative degree pieces of $A$ are all zero.

Recall that a superalgebra is \emph{supercommutative} if $ab = (-1)^{\bar{a} \bar{b}} ba$ for all homogeneous $a,b \in A$.  Whenever we write equations involving degrees or parities of elements, we implicitly assume that these elements are homogeneous, and we extend by linearity.  We let $\N$ denote the set of non-negative integers and let $\N_+$ denote the set of positive integers.

We will work freely in the language of string diagrams for $\kk$-linear graded monoidal supercategories.  We refer the reader to \cite{Sav-exp} for a brief overview of these concepts, to \cite[Ch.~1,2]{TV17} for a more in-depth treatment, and to \cite{BE17} for a detailed discussion of signs in the analogous super setting.

\iftoggle{detailsnote}{
\subsection*{Hidden details} For the interested reader, the tex file of the \href{https://arxiv.org/abs/2007.02732}{arXiv version} of this paper includes hidden details of some straightforward computations and arguments that are omitted in the pdf file.  These details can be displayed by switching the \texttt{details} toggle to true in the tex file and recompiling.
}{}

\subsection*{Acknowledgements}

The authors would like to thank A.~Licata, Y.~Mousaaid, R.~Muth, P.~Samuelson, and W.~Wang for helpful conversations.  A.S.\ was supported by Discovery Grant RGPIN-2017-03854 from the Natural Sciences and Engineering Research Council of Canada.

\section{Cocenters and traces}

In this section we review some basic facts about cocenters of superalgebras and traces of monoidal supercategories.  We also show how one can endow the cocenter of a \emph{Frobenius} superalgebra with the structure of a superalgebra.

\subsection{Cocenters}

For a graded superalgebra $A$, define
\[
    [A,A] = \Span_\kk \{ab - (-1)^{\bar{a} \bar{b}} ba : a,b \in A\}.
\]
The \emph{cocenter} of $A$ is
\[
    C(A) = A / [A,A].
\]
For $a \in A$, we let $\cocenter{a}$ denote its image in $C(A)$.  Note that, in general, $C(A)$ is only a graded $\kk$-supermodule, and does not have any natural multiplication.  If $f \colon A \to B$ is a homomorphism of graded algebras, we have an induced $\kk$-linear map
\[
    \cocenter{f} \colon C(A) \to C(B),\quad
    \cocenter{f}(\cocenter{a}) = \cocenter{f(a)},\quad a \in A.
\]

The \emph{opposite algebra} $A^\op$ of $A$ is the same underlying $\kk$-module as $A$, with multiplication given by
\[
  a^\op b^\op = (-1)^{\bar{a} \bar{b}} (ba)^\op,
\]
where $a^\op$ denotes the element of $A^\op$ corresponding to $a \in A$.  We then have
\[
  \left( ab - (-1)^{\bar{a} \bar{b}} ba \right)^\op
  = - \left( a^\op b^\op - (-1)^{\bar{a} \bar{b}} b^\op a^\op \right).
\]
Hence we have an isomorphism of graded $\kk$-supermodules
\begin{equation} \label{beach}
  C(A) \xrightarrow{\cong} C(A^\op),\quad
  \cocenter{a} \mapsto \cocenter{a^\op},\ a \in A.
\end{equation}

\subsection{Cocenters of Frobenius algebras}

As noted above, a priori, $C(A)$ is only a graded $\kk$-supermodule, and does not have any natural multiplication.  However, as we explain here, the cocenter of a symmetric Frobenius algebra does have a natural associative linear binary operation.  Suppose $A$ is a symmetric graded Frobenius superalgebra with trace map of degree $-2d$.  In other words, we have an even $\kk$-linear map $\tr \colon A \to \kk$ of degree $-2d$ satisfying
\[
    \tr(ab) = (-1)^{\bar{a} \bar{b}} \tr(ba),
\]
such that $\ker \tr$ contains no nonzero left ideals.  It follows immediately that we have an induced $\kk$-linear map
\[
    \tr \colon C(A) \to \kk,\quad \tr(\cocenter{a}) = \cocenter{\tr(a)},\quad a \in A.
\]

The \emph{center} of $A$ is
\[
    Z(A) := \{ a \in A : ab = (-1)^{\bar{a}\bar{b}} ba \text{ for all } b \in A\}.
\]
For $a,b \in A$ and $c \in Z(A)$, we have
\[
    c \left( ab - (-1)^{\bar{a} \bar{b})} ba \right)
    = (ca)b - (-1)^{\bar{b}(\bar{a} + \bar{c})} b(ca).
\]
Thus, we have an action of $Z(A)$ on $C(A)$ defined by
\begin{equation} \label{senate}
    c \cocenter{a} = \cocenter{a} c := \cocenter{ca} = \cocenter{ac},\quad a \in A,\ c \in Z(A).
\end{equation}

Fix a homogeneous basis $\bB_A$ of $A$.  We define the dual basis elements $a^\vee$, $a \in \bB_A$, by
\[
    \tr(a^\vee b) = \delta_{a,b},\quad a,b \in \bB_A.
\]
It is straightforward to verify that $\sum_{b \in \bB_A} b \otimes b^\vee$ is independent of the choice of basis $\bB_A$.  Furthermore, the dual of the basis $\bB_A^\vee = \{b^\vee : b \in \bB_A\}$ is given by
\begin{equation} \label{doubledual}
    (b^\vee)^\vee = (-1)^{\bar{b}} b.
\end{equation}
For all $a \in A$, we have
\begin{equation} \label{contract}
    \sum_{b \in \bB_A} \tr(b^\vee a) b = a = \sum_{b \in \bB_A} \tr(ab) b^\vee.
\end{equation}
It follows that we have the important ``teleportation'' properties
\begin{equation} \label{teleport}
    \sum_{b \in \bB_A} ab \otimes b^\vee
    = \sum_{b \in \bB_A} b \otimes b^\vee a
    \quad \text{and} \quad
    \sum_{b \in \bB_A} ba \otimes b^\vee
    = \sum_{b \in \bB_A} b \otimes ab^\vee.
\end{equation}

For $a,b \in A$ and $r \in \N$, define a binary operation $\diamond$ on $A$ by
\begin{equation} \label{diamond}
    a \diamond b
    := \sum_{c \in \bB_A} (-1)^{\bar{b} \bar{c}} a c b c^\vee.
\end{equation}
Note that this is an operation of degree $2d$.  We define
\[
    \kappa = \kappa_A
    := \sum_{a \in \bB_A} a \diamond a^\vee
    = \sum_{a,b \in \bB_A} (-1)^{\bar{a} \bar{b}} a b a^\vee b^\vee.
\]
It is straightforward to verify that $\kappa$ is independent of the chosen basis $\bB_A$, and it follows from \cref{teleport} that $\kappa \in Z(A)$.  Note that $\deg (\kappa) = 4d$.  Hence, when $A$ is nontrivially positively graded, we have $\kappa=0$ (since the top degree of $A$ is $2d$).

\begin{prop} \label{trooper}
    The binary operation
    \[
        C(A) \times C(A) \to C(A),\quad
        (\cocenter{a},\cocenter{b}) \mapsto \cocenter{a} \diamond \cocenter{b} := \cocenter{a \diamond b},
    \]
    is well defined, bilinear, associative, and commutative.
\end{prop}

\begin{proof}
    For $a,a',b,b' \in A$, we have
    \begin{multline*}
        \cocenter{(aa') \diamond b}
        = \sum_{c \in \bB_A} (-1)^{\bar{b}\bar{c}} \cocenter{aa' c b c^\vee}
        \stackrel{\cref{contract}}{=} \sum_{c,e \in \bB_A} (-1)^{\bar{b}\bar{c}} \tr(e^\vee a'c) \cocenter{aebc^\vee}
        = \sum_{c,e \in \bB_A} (-1)^{\bar{b}\bar{e} + \bar{b}\bar{a'}} \tr(e^\vee a'c) \cocenter{aebc^\vee}
        \\
        \stackrel{\cref{contract}}{=} \sum_{e \in \bB_A} (-1)^{\bar{b}\bar{e} + \bar{b}\bar{a'}}  \cocenter{aebe^\vee a'}
        = \sum_{e \in \bB_A} (-1)^{\bar{b}\bar{e} + \bar{a}\bar{a'}} \cocenter{a'aebe^\vee}
        = (-1)^{\bar{a} \bar{a'}} \cocenter{(a'a) \diamond b},
    \end{multline*}
    where, in the third equality, we used the fact that $\tr(aebc^\vee)=0$ unless $\bar{c}=\bar{e}+\bar{a'}$.  Similarly,
    \[
        \cocenter{a \diamond (bb')}
        = \sum_{c \in \bB_A} (-1)^{(\bar{b} + \bar{b'})\bar{c}} \cocenter{a c b b' c^\vee}
        = \sum_{c \in \bB_A} (-1)^{(\bar{b} + \bar{b'})\bar{c} + \bar{b}\bar{b'}} \cocenter{a c b' b c^\vee}
        = (-1)^{\bar{b} \bar{b'}} \cocenter{a \diamond (b'b)}.
    \]
    Thus the operation is well defined.  It is clear that it is bilinear.  For $a,b,c \in A$, we have
    \[
        \cocenter{(a \diamond b) \diamond c}
        = \sum_{e,f \in \bB_A} (-1)^{\bar{b}\bar{e} + \bar{c} \bar{f}} \cocenter{a e b e^\vee f c f^\vee}
        = \sum_{e,f \in \bB_A} (-1)^{\bar{c}\bar{f} + \bar{e}(\bar{b} + \bar{c})} \cocenter{a e b f c f^\vee e^\vee}
        = \cocenter{a \diamond (b \diamond c)}.
    \]
    Hence the operation is associative.  Finally,
    \[
        \cocenter{a \diamond b}
        = \sum_{c \in \bB_A} (-1)^{\bar{b}\bar{c}} \cocenter{acbc^\vee}
        = \sum_{c \in \bB_A} (-1)^{\bar{a} (\bar{b} + \bar{c}) + \bar{c}} \cocenter{bc^\vee ac}
        \stackrel{\cref{doubledual}}{=} \sum_{c \in \bB_A} (-1)^{\bar{a} (\bar{b} +\bar{c})} \cocenter{bcac^\vee}
        = (-1)^{\bar{a} \bar{b}} \cocenter{b \diamond a}.
    \]
    Thus the operation is supercommutative.
\end{proof}

\begin{rem} \label{teeth}
    In general, $A$ is not a unital algebra under the operation $\diamond$ since it may not possess a unit element for the multiplication.  For example, if $d > 0$, then $\diamond$ cannot possess a unit element since it is a binary operation of degree $2d$.
\end{rem}

\begin{rem}[Change of trace map] \label{ifislip}
    If $(A,\tr_1)$ and $(A,\tr_2)$ are symmetric graded Frobenius superalgebras with the same underlying graded superalgebra $A$, then there exists an even, degree zero, invertible element $e \in Z(A)$ such that
    \begin{equation} \label{slip}
        \tr_1(a) = \tr_2(ea) \quad \text{for all } a \in A.
    \end{equation}
    If $\{b^{\vee,1} : b \in \bB_A\}$ and $\{b^{\vee,2} : b \in \bB_A\}$ denote the dual bases of $\bB_A$ with respect to $\tr_1$ and $\tr_2$, respectively, then $b^{\vee,2} = e b^{\vee,1}$ for $b \in \bB_A$.  Using subscripts to indicate the trace map in question, we then have
    \begin{equation} \label{slippin}
        a \diamond_2 b = e a \diamond_1 b,\quad a, b \in A,\qquad
        \kappa_2 = e^2 \kappa_1.
    \end{equation}
    Hence we have an isomorphism of graded associative (not necessarily unital) superalgebras
    \[
        (C(A), \diamond_1) \to (C(A), \diamond_2),\quad
        a \mapsto e^{-1} a,\ a \in A.
    \]
\end{rem}

\subsection{Examples}

We now consider some examples of symmetric graded Frobenius superalgebras that are of particular interest.  First note that if $A$ and $B$ are two symmetric graded Frobenius superalgebras with trace maps $\tr_A$ and $\tr_B$, then their direct product $A \times B$ is also a symmetric graded Frobenius superalgebra with trace map $(a,b) \mapsto \tr_A(a) + \tr_B(b)$, $a \in A$, $b \in B$.  Then $C(A \times B) \cong C(A) \times C(B)$ and $\kappa_{A \times B} = (\kappa_A,\kappa_B)$.

\begin{eg}[Matrix algebras] \label{matrix}
    For $n \in \N_+$, the algebra $M_n(\kk)$ of $n \times n$ matrices with entries in $\kk$ (with trivial grading) is a symmetric Frobenius algebra with the usual trace map.  It is straightforward to verify that $C(M_n(\kk)) \cong \kk$ as algebras and $\kappa = 1$.
\end{eg}

\begin{eg}[Semisimple algebras] \label{semisimple}
    If $\kk$ is algebraically closed, then any semisimple algebra is isomorphic to a product of matrix algebras by Wedderburn's Structure Theorem.  Hence if $A$ is a symmetric Frobenius algebra whose underlying algebra is semisimple, we have $C(A) \cong \kk^n$ for some $n$.  Again $\kappa$ is the identity element of $\kk^n$.
\end{eg}

\begin{eg}[Group algebras] \label{groupalg}
    Suppose $\Gamma$ is a finite group, and let $A = \kk \Gamma$ be its group algebra (with trivial grading).  This is a symmetric Frobenius algebra with trace map given by $\tr(\sum_{g \in \Gamma} \alpha_g g) = \alpha_{1_\Gamma}$.  Furthermore, $\frac{1}{|\Gamma|} \cocenter{1_\Gamma}$ is a multiplicative unit for the operation $\diamond$, and so $C(A)$ is a commutative unital algebra.  Let $R(\Gamma)$ denote the space of class functions on $\Gamma$.  The standard bilinear form on $R(\Gamma)$ is given by
    \[
        \langle f, g \rangle
        = \frac{1}{|\Gamma|} \sum_{y \in \Gamma} f(y) g(y^{-1}),\quad f,g \in R(\Gamma).
    \]
    The convolution product on $R(\Gamma)$ is given by
    \[
        (f \circ g)(z) = \sum_{y \in \Gamma} f(zy^{-1}) g(y),\quad  f,g \in R(\Gamma),\ z \in \Gamma.
    \]
    The map
    \[
        \varphi \colon C(A) \to R(\Gamma),\quad
        \cocenter{a} \mapsto \big(z \mapsto \tr(a \diamond z) \big),
    \]
    is an isomorphism of algebras and
    \[
        \langle \varphi(\alpha), \varphi(\alpha) \rangle
        = \tr(\alpha \diamond \beta),\quad
        \alpha, \beta \in C(A).
    \]
    \details{
        It is straightforward to verify that $\varphi$ is an isomorphism of vector spaces.  Now, for $a,b \in A$ and $z \in \Gamma$, we have
        \begin{align*}
            (\varphi(a) \circ \varphi(b))(z)
            &= \sum_{y \in \Gamma} \varphi(a)(zy^{-1}) \varphi(b)(y) \\
            &= \sum_{y \in \Gamma} \tr \left( a \diamond (zy^{-1}) \right) \tr(b \diamond y) \\
            &= \sum_{y,z \in \Gamma} \tr(awzy^{-1}w^{-1}) \tr(b \diamond y) \\
            &= \sum_{y,z \in \Gamma} \tr(w^{-1}awzy^{-1}) \tr(b \diamond y) \\
            &\stackrel{\mathclap{\cref{contract}}}{=}\ \sum_{z \in \Gamma} \tr \left( b \diamond (w^{-1} awz) \right) \\
            &= \sum_{z \in \Gamma} \tr \left( b \diamond (z w^{-1} a w) \right) \\
            &= \tr(b \diamond z \diamond w) \\
            &= \tr(z \diamond a \diamond b) \\
            &= \varphi(a \diamond b)(z).
        \end{align*}
        So $\varphi$ is an isomorphism of algebras.  We also have, using the above,
        \[
            \langle \varphi(\alpha), \varphi(\beta) \rangle
            = \frac{1}{|\Gamma|} (\varphi(\alpha) \circ \varphi(\beta))(1_\Gamma)
            = \frac{1}{|\Gamma|} \tr(1_\Gamma \diamond a \diamond b)
            = \tr(a \diamond b).
        \]
    }
    We have
    \[
        \kappa = \sum_{g,h \in \Gamma} ghg^{-1}h^{-1}.
    \]
    On the other hand, by Maschke's Theorem, $\kk[\Gamma]$ is semisimple.  Thus, if $\kk$ is algebraically closed, we can choose a trace map as in \cref{matrix,semisimple} so that $\kappa=1$.
\end{eg}

\begin{eg}[Zigzag algebras] \label{zigzag}
    Let $\Gamma = (\Gamma_0,\Gamma_1)$ be a connected graph without loops or multiple edges, and let $D \Gamma$ be the doubled directed graph, which has the same set of vertices as $\Gamma$, and, for each edge $\{i,j\} \in \Gamma_1$, directed edges $a_{i,j}$ from $j$ to $i$ and $a_{j,i}$ from $i$ to $j$.  We fix a collection of \emph{orientation coefficients} $\epsilon_{i,j} \in \kk$, $\{i,j\} \in \Gamma_0$, satisfying $\epsilon_{i,j} = 0$ iff $\{i,j\} \notin \Gamma_1$ and $\epsilon_{i,j} = -\epsilon_{j,i}$ for all $i,j \in \Gamma_0$.

    If $\Gamma$ has at least $2$ vertices, let $A(\Gamma)$ denote the quotient of the path algebra of $D \Gamma$ modulo the following relations:
    \begin{itemize}
        \item all paths of length three or greater are zero;
        \item all paths of length two that are not cycles are zero;
        \item $\epsilon_{i,j} a_{i,j} a_{j,i} = \epsilon_{i,l} a_{i,l} a_{l,i}$ for all $\{i,j\}, \{i,l\} \in \Gamma_1$.
    \end{itemize}
    The algebra $A(\Gamma)$ is a \emph{skew-zigzag algebra}; see \cite[\S4.6]{HK01} and \cite[\S5]{Cou16}.  (In certain cases, e.g.\ if $\Gamma$ is bipartite, then $A(\Gamma)$ is isomorphic to a \emph{zigzag algebra}.)

    The algebra $A(\Gamma)$ is generated by the length zero paths $e_i$, $i \in \Gamma_0$, and the length one paths $a_{i,j}$, $\{i,j\} \in \Gamma_1$.  For $i \in \Gamma_i$, define
    \[
        c_i = \epsilon_{i,j} a_{i,j} a_{j,i}
        \quad \text{for some $j$ such that } \{i,j\} \in \Gamma_1.
    \]
    The defining relations of $A(\Gamma)$ imply that $c_i$ is independent of the choice $j$.  It is easily verified (see, for example, \cite[Prop.~3.7]{Cou16}) that $A(\Gamma)$ has basis
    \[
        \{a_{i,j} : \{i,j\} \in \Gamma_1\}
        \cup \{e_i : \ i\in \Gamma_0\}
        \cup \{c_i : i\in \Gamma_0\}.
    \]
    The algebra $A(\Gamma)$ is graded by path length, so that $\deg(e_i) = 0$, $\deg(a_{i,j}) = 1$, and $\deg(c_i) = 2$.  We give it the $\Z_2$-grading induced by this $\Z$-grading.  Then $A(\Gamma)$ is a symmetric Frobenius superalgebra, with even trace map $\tr$ of degree $-2$ satisfying
    \[
        \tr(e_i) = \tr(a_{j,l}) = 0,\
        \tr(c_i) = 1, \quad
        i \in \Gamma_0,\ \{j,l\} \in \Gamma_1.
    \]
    It follows that the duals of the basis elements are given by
    \[
        e_i^\vee = c_i, \quad
        a_{i,j}^\vee = \epsilon_{i,j} a_{j,i}.
    \]
    As in \cref{teeth}, $C(A(\Gamma))$ does not have a unit with respect to the $\diamond$ operation, and $\kappa = 0$.

    We can explicitly compute the commutators of the basis elements:
    \begin{gather*}
        a_{i,j} a_{l,m} + a_{l,m} a_{i,j}
        = \delta_{j,l} \delta_{m,i} \epsilon_{i,j} (c_i - c_j),\quad
        a_{i,j} e_l - e_l a_{i,j}
        = \delta_{j,l} a_{i,j}  - \delta_{i,l} a_{i,j},
        \\
        e_i e_j - e_j e_i = c_i c_j - c_j c_i = e_i c_j - c_j e_i = a_{i,j}, c_l - c_l a_{i,j} = 0.
    \end{gather*}
    It follows that $C(A(\Gamma))$ has basis $\{\cocenter{c},\cocenter{e_i} : i \in \Gamma_0\}$, where $c=c_i$ for any $i \in \Gamma_0$.  (The class $\cocenter{c_i}$ is independent of $i$.)  For all $i,j \in \Gamma_0$, we have
    \[
        \cocenter{c} \diamond \cocenter{e_i} = 0,
        \qquad
        \cocenter{e_i} \diamond \cocenter{e_j} =
        \begin{cases}
            \cocenter{c} & \{i,j\} \in \Gamma_1, \\
            2\cocenter{c} & i=j, \\
            0 & \text{otherwise}.
        \end{cases}
    \]
    Indeed, the first relation follows from the fact that $\diamond $ is homogeneous of degree 2, and any element of degree 3 or greater is 0 in $A(\Gamma)$, while the second relation follows from the computation
    \begin{align*}
        \cocenter{e_i \diamond e_j}
        &= \sum_{\{l,m\} \in \Gamma_1} \cocenter{ \epsilon_{l,m} e_i a_{l,m} e_j a_{m,l} + \epsilon_{m,l} e_i a_{m,l} e_j a_{l,m} }
        + \sum_{l \in \Gamma_0} \cocenter{ e_i e_l e_j c_l + e_i c_l e_j e_l }
        \\
        &= \sum_{\{l,m\} \in \Gamma_1} \cocenter{ \delta_{i,l} \delta_{j,m} c + \delta_{i,m} \delta_{j,l} c }
        + 2 \sum_{l \in \Gamma_0} \cocenter{\delta_{i,l} \delta_{j,l} c}.
    \end{align*}

    If $\Gamma$ has a single vertex, we define $A(\Gamma) = \kk[c]/(c^2)$ with $c$ even of degree $2$ and $\tr(c) = 1$, $\tr(1) = 0$.  Then $A(\Gamma)$ is commutative, hence $C(A(\Gamma)) = A(\Gamma) = \kk[c]/(c^2)$ as $\kk$-modules; we have $1 \diamond 1 = 2c$ and $c \diamond c = c \diamond 1 = 0$.
\end{eg}

\subsection{Cocenters of locally unital algebras\label{coextro}}

A \emph{locally unital graded superalgebra} is a graded associative (but not necessarily unital) $\kk$-superalgebra $R$ equipped with a system of mutually orthogonal idempotents $\{1_X : X \in \mathbf{R}\}$ such that
\[
    R = \bigoplus_{X,Y \in \mathbf{R}} 1_Y R 1_X.
\]
We define
\[
    R_\diag := \bigoplus_{X \in \mathbf{R}} 1_X R 1_X,\qquad
    [R,R]_\diag = \sum_{X,Y \in \mathbf{R}} [1_X R 1_Y, 1_Y R 1_X].
\]

\begin{lem} \label{sled}
    If $R$ is a locally unital graded superalgebra, then the inclusion $R_\diag \hookrightarrow R$ induces an isomorphism of graded $\kk$-supermodules
    \[
        R_\diag / [R,R]_\diag \cong C(R).
    \]
\end{lem}

\begin{proof}
    If $X,Y \in \mathbf{R}$ with $X \ne Y$, then $1_X R 1_Y  = [1_X R 1_Y, 1_Y] \subseteq [1_X R 1_Y, 1_Y R 1_Y]$.  Thus $[R,R] =  [R,R]_\diag \oplus \bigoplus_{X \ne Y} 1_X R 1_Y$, and the result follows.
\end{proof}

For the remainder of this section, we assume that $R$ is a locally unital graded superalgebra with a system of mutually orthogonal idempotents $\{1_n : n \in \N\}$.  We also assume that we have subsets $D_{m,n} \subseteq 1_m R 1_n$, $m,n \in \N$, and $D_n \subseteq 1_n R 1_n$, $n \in \N$, whose elements are homogeneous, with the following properties:
\begin{description}
    \item[B1\label{B1}] We have $D_{n,n} = \{1_n\}$ for all $n \in \N$, and $\bB_{m,n} := \bigsqcup_{l=0}^{\min(m,n)} D_{m,l} D_l D_{l,n}$ is a basis of $1_m R 1_n$ for each $m,n \in \N$.  (Part of our assumption here is that the sets $D_{m,l} D_l D_{l,n}$ are disjoint.)
    \item[B2\label{B2}] For all $n \in \N$, $R_n := \Span_\kk D_n$ is a subalgebra of $1_n R 1_n$.
\end{description}
For $m,n \in \N$, define $\bB_{m,n}^< := \bigsqcup_{l=0}^{\min(m,n)-1} D_{m,l} D_l D_{l,n}$ and let
\[
    \bB_{m,n}'
    := \bB_{m,n} \setminus \bB_{m,n}^<
    =
    \begin{cases}
        D_m D_{m,n} & \text{if } m < n, \\
        D_{m,n} D_n & \text{if } m > n, \\
        D_n & \text{if } m=n.
    \end{cases}
\]
In other words, $\bB_{m,n}^<$ is the subset of $\bB_{m,n}$ consisting of elements that factor through $1_l$ for some $l < \min(m,n)$.  It follows from the above that we have a decomposition of $\kk$-modules
\begin{equation} \label{mouse}
    1_n R 1_n = R_n \oplus I_n,\quad
    I_n := \sum_{m < n} 1_n R 1_m R 1_n = \Span_\kk \bB_{n,n}^<.
\end{equation}
Note that $I_n$ is an ideal of $1_n R 1_n$.

\begin{lem} \label{ginger}
    With the above assumptions and notation, we have
    \begin{equation} \label{nigiri}
        [R,R]_\diag
        = \bigoplus_{n \in \N} [R_n, R_n] + \sum_{m < n} \Span_\kk [D_m D_{m,n}, D_{n,m}].
    \end{equation}
\end{lem}

\begin{proof}
    For $q \in \N$, let $V_q := \sum_{m \le n \le q} \Span_\kk [\bB_{m,n}', \bB_{n,m}']$.  We first show that, for all $m \le n$, we have
    \begin{equation} \label{sushi}
        [1_m R 1_n, 1_n R 1_m] \subseteq V_n
        \quad \text{or, equivalently,} \quad
        [\bB_{m,n}, \bB_{n,m}^<] \subseteq V_n
        \text{ and } [\bB_{m,n}^<, \bB_{n,m}] \subseteq V_n.
    \end{equation}
    We prove \cref{sushi} by induction on $n$.  The case $n=0$ is immediate, since $\bB_{0,0}^< = \varnothing$.  Fix $N \ge 1$ and suppose \cref{sushi} holds for $m \le n \le N-1$.  We will show it holds for $m \le n=N$ by induction on $m$.  The case $m=0$ is straightforward since $\bB_{N,0}^< = \bB_{0,N}^< = \varnothing$.  Now suppose that $1 < M \le N$ and that \cref{sushi} holds for all $m < M$ and $n=N$.  By definition, an arbitrary element of $\bB_{N,M}^<$ can be written in the form $gh$ for $g \in 1_N R 1_l$, $h \in 1_l R 1_M$, $l < M$.  Then, for $f \in \bB_{M,N}$, we have
    \[
        [f,gh]
        = [fg,h] + [hf,g]
        \in [1_M R 1_l, 1_l R 1_M] + [1_l R 1_N, 1_N R 1_l]
        \subseteq V_N,
    \]
    where the final inclusion follows from the induction hypotheses.  This proves that $[\bB_{M,N}, \bB_{N,M}^<] \subseteq V_N$.  The proof that $[\bB_{M,N}^<, \bB_{N,M}] \subseteq V_N$ is analogous.  This completes the proof of \cref{sushi}.

    Now let $W$ be the right-hand side of \cref{nigiri}.  For all $m \le n$, we must show that $[1_m R 1_n, 1_n R 1_m] \subseteq W$.  By \cref{sushi}, it suffices to show that $V_n \subseteq W$ for all $n \in \N$.  We prove this by induction on $n$.  The case $n=0$ is immediate, since $V_0 = [R_0,R_0]$.  Suppose that $N \ge 1$ and that $V_n \subseteq W$ for all $n < N$.  Since $V_N = V_{N-1} + \sum_{m \le N} \Span_\kk [\bB_{m,N}', \bB_{N,m}']$, we must show that $[\bB_{m,N}', \bB_{N,m}'] \subseteq W$ for all $m \le N$.

    First suppose that $m < N$.  Then $[\bB_{m,N}', \bB_{N,m}'] = [D_m D_{m,N}, D_{N,m} D_m]$.  Now, for $f \in D_m D_{m,N}$, $g \in D_{N,m}$, $h \in D_m$, we have
    \[
        [f,gh]
        = [fg,h] + [hf,g]
        \in [1_m R 1_m, R_m] + \Span_\kk [D_m D_{m,N}, D_{N,m}]
        \overset{\cref{sushi}}{\subseteq} V_m + W
        \subseteq W,
    \]
    where we have used \eqref{B2} to conclude that $hf \in \Span_\kk D_m D_{m,n}$, together with the induction hypothesis.  Thus $[\bB_{m,N}', \bB_{m,N}'] \subseteq W$ when $m < N$.  On the other hand, we have $[\bB_{N,N}', \bB_{N,N}'] = [D_N, D_N] \subseteq W$.  This completes the proof of the inductive step.
\end{proof}

\begin{lem} \label{crab}
    We have
    \begin{equation} \label{wasabi}
        \left( \bigoplus_{n \in \N} R_n \right)
        \cap \left( \sum_{m < n} \Span_\kk [D_m D_{m,n}, D_{n,m}] \right)
        = 0.
    \end{equation}
\end{lem}

\begin{proof}
    Consider an arbitrary nonzero element
    \begin{equation} \label{seaweed}
        \sum_{m<n} \sum_{f \in D_m,\ g \in D_{m,n},\, h \in D_{n,m}} c_{f,g,h} [fg,h],\quad
        c_{f,g,h} \in \kk,
    \end{equation}
    of the second sum in \cref{wasabi}.  Choose the largest $N \in \N$ such that $c_{f,g,h} \ne 0$ for some $f \in D_m$, $g \in D_{m,N}$, $h \in D_{N,m}$, $m < N$, and consider the following decomposition of \cref{seaweed}:
    \begin{equation} \label{shrimp}
        \sum_{m<n<N} \sum_{f \in D_m,\ g \in D_{m,n},\, h \in D_{n,m}} c_{f,g,h} [fg,h]
        + \sum_{m<N} \sum_{f \in D_m,\ g \in D_{m,N},\, h \in D_{N,m}} c_{f,g,h} (\underbrace{fgh}_{\in 1_m R 1_m} - \underbrace{hfg}_{\in I_N}).
    \end{equation}
    By \eqref{B1}, the elements $hfg \in I_N$ for $f \in D_m$, $g \in D_{m,N}$, $h \in D_{N,m}$, are linearly independent.  Since all the other terms appearing in \cref{shrimp} lie in $\bigoplus_{m < N} 1_m R 1_m$, the element \cref{seaweed} cannot lie in $\bigoplus_{n \in \N} R_n$.
\end{proof}

\begin{prop} \label{snowshoe}
    Under the above assumptions on $R$, the inclusion $\bigoplus_{n \in \N} R_n \hookrightarrow R$ induces an isomorphism of graded $\kk$-supermodules
    \[
        \bigoplus_{n \in \N} C(R_n) \cong C(R).
    \]
\end{prop}

\begin{proof}
    In light of \cref{sled}, we consider the composition
    \[
        \varphi \colon \bigoplus_{n \in \N} R_n \hookrightarrow R_\diag \twoheadrightarrow R_\diag/[R,R]_\diag.
    \]
    By \cref{ginger,crab}, we have $\ker(\varphi) = [R,R]_\diag \cap \bigoplus_{n \in \N} R_n = \bigoplus_{n \in \N} [R_n,R_n]$.

    It remains to be shown that $\varphi$ is surjective.  We do this by showing that $\bigoplus_{n \le N} 1_n R 1_n + [R,R]_\diag$ is contained in $\im(\varphi)$ by induction on $N$.  The result is clearly true for $N=0$ since $I_0 = 0$, implying that $1_0 R 1_0 = R_0$.  Now assume $N>0$.  We have
    \[
        I_N + [R,R]_\diag
        = \sum_{n < N} 1_N R 1_n R 1_N + [R,R]_\diag
        = \sum_{n < N} 1_n R 1_N R 1_n + [R,R]_\diag.
    \]
    Thus
    \begin{multline*}
        1_N R 1_N + [R,R]_\diag
        = \left( R_N + [R,R]_\diag \right) + \left( I_N + [R,R]_\diag \right) \\
        = \left( R_N + [R,R]_\diag \right) + \left( \sum_{n < N} 1_n R 1_N R 1_n + [R,R]_\diag \right)
        \subseteq (R_N + [R,R]_\diag) + \left( \bigoplus_{n < N} 1_n R 1_n + [R,R]_\diag \right).
    \end{multline*}
    It then follows from the induction hypothesis that $\bigoplus_{n \le N} 1_n R 1_n + [R,R]_\diag \subseteq \im(\varphi)$.  Hence $\varphi$ is surjective.
\end{proof}

\subsection{Traces of linear categories\label{traces}}

Suppose $\cC$ is a $\kk$-linear graded supercategory.  The \emph{trace}, or \emph{zeroth Hochschild homology}, of $\cC$ is the graded $\kk$-module
\[
    \Tr(\cC) := \left( \bigoplus_X \End_\cC(X) \right) / \Span_\kk \{f \circ g - (-1)^{\bar{f} \bar{g}} g \circ f\},
\]
where the sum is over all objects $X$ of $\cC$, and $f$ and $g$ run through all pairs of morphisms $f \colon X \to Y$ and $g \colon Y \to X$ in $\cC$.  We let $\cocenter{f} \in \Tr(\cC)$ denote the class of an endomorphism $f \in \End_\cC(X)$.  If $\cC$ is a $\kk$-linear graded \emph{monoidal} supercategory, then $\Tr(\cC)$ is a graded $\kk$-superalgebra with multiplication given by
\begin{equation} \label{tracemult}
    \cocenter{f} \cocenter{g} = \cocenter{f \otimes g}.
\end{equation}
As shown in \cite[Prop.~3.2]{BGHL14}, the trace is invariant under passage to the additive Karoubi envelope:
\[
    \Tr(\Kar(\cC)) \cong \Tr(\cC).
\]

The data of a $\kk$-linear graded supercategory $\cR$ is the same as the data of a locally unital graded superalgebra $R$.  Under this identification, $\mathbf{R}$ is the object set of $\cR$, the idempotent $1_X$ is the identity endomorphism of the object $X$, $1_Y R 1_X := \Hom_\cR(X,Y)$, and multiplication in $R$ is induced by composition in $\cR$.  It then follows from \cref{sled} that
\begin{equation} \label{convert}
    \Tr(\cR) = C(R).
\end{equation}
If $\cR^\op$ denotes the opposite supercategory of $\cR$, then \cref{beach} implies that we have an isomorphism of graded superalgebras
\begin{equation} \label{sand}
  \Tr(\cR) \cong \Tr(\cR^\op).
\end{equation}

Let $\cC^\rev$ denote the reverse category of $\cC$.  This is the $\kk$-linear graded monoidal supercategory that is equal to $\cC$ as a $\kk$-linear graded supercategory, but with tensor product given by $X^\rev \otimes Y^\rev := (Y \otimes X)^\rev$ and $f^\rev \otimes g^\rev = (-1)^{\bar{f} \bar{g}} (g \otimes f)^\rev$ for objects $X,Y$ and morphisms $f,g$.  It follows that we have an isomorphism of graded superalgebras
\begin{equation} \label{flipflops}
  \Tr(\cC^\rev) \xrightarrow{\cong} \Tr(\cC)^\op,\quad
  \cocenter{f^\rev} \mapsto \cocenter{f}.
\end{equation}

\section{Frobenius W-algebras}

In this section we fix a symmetric graded Frobenius superalgebra $A$.  Our goal in this section is to associate to $A$ a W-algebra.

\subsection{Definition}

\begin{lem} \label{dentist}
    For $f(X),g(X) \in \kk[X]$ and $m,n \in \Z$, we have
    \[
        f \left( X + \tfrac{n}{2}Y \right) g \left( X -\tfrac{m}{2}Y \right) - f \left( X - \tfrac{n}{2}Y \right) g \left( X + \tfrac{m}{2}Y \right)
        \in Y \kk[X,Y^2] \subseteq \kk[X,Y].
    \]
\end{lem}

\begin{proof}
    Since the given expression is $\kk$-linear in both $f(X)$ and $g(X)$, it suffices to consider the case where $f(X) = X^r$ and $g(X) = X^s$ for some $r,s \in \N$.  Then we have
    \[
        \left( X + \tfrac{n}{2}Y \right)^r \left( X -\tfrac{m}{2}Y \right)^s - \left( X - \tfrac{n}{2}Y \right)^r \left( X + \tfrac{m}{2}Y \right)^s
        = \sum_{i=0}^r \sum_{j=0}^s \left( (-1)^j - (-1)^i \right) \binom{r}{i} \binom{s}{j} X^{r+s-i-j} \left( \tfrac{n}{2} \right)^i \left( \tfrac{m}{2} \right)^j Y^{i+j}.
    \]
    Since $(-1)^j = (-1)^i$ when $i+j$ is even, the result follows.
\end{proof}

For $f(X),g(X) \in \kk[X]$ and $m,n \in \Z$, define
\begin{equation} \label{star}
    f(X) \star_{m,n} g(X)
    := \left. \frac{f \left( X + \tfrac{n}{2} Y \right) g \left( X -\tfrac{m}{2} Y \right) - f \left( X - \tfrac{n}{2}Y \right) g \left( X + \tfrac{m}{2} Y \right)}{Y} \right|_{Y^2 = \kappa}
    \in Z(A)[X].
\end{equation}

Now let $\fD(A)$ denote the Lie algebra with underlying vector space $C(A)[X,z]$, and with Lie bracket
\begin{equation} \label{mars}
    [a f(X) z^m, b g(X) z^n]
    =
    (a \diamond b) \left( f(X) \star_{m,n} g(X) \right) z^{m+n},
\end{equation}
where we use the action \cref{senate} of the center on the cocenter.  If $u$ and $v$ are indeterminates, we can write the relation \cref{mars} in terms of generating functions as
\begin{equation} \label{martian}
    \left[ a \exp(uX) z^m, b \exp(vX) z^n \right]
    = 2a \diamond b \left. \frac{\sinh \left( \left(\tfrac{n}{2} u - \tfrac{m}{2}v \right) Y \right)}{Y} \right|_{Y^2=\kappa} \exp(uX + vX) z^{m+n}.
\end{equation}
It is a routine verification to check that this bracket is bilinear, alternating, and satisfies the Jacobi identity.
\details{
    It is clear that the bracket is bilinear and alternating.  Using the form \cref{martian}, to verify that the bracket satisfies the Jacobi identity, we must show that
    \begin{multline*}
        \left[ a\exp(uX)z^l, \left[ b\exp(vX)z^m, c\exp(wX)z^n \right] \right]
        + \left[ c\exp(wX)z^n, \left[ a\exp(uX)z^l, b\exp(vX)z^m \right] \right]
        \\
        + \left[ b\exp(vX)z^m, c\exp(wX)z^n, \left[ a\exp(uX)z^l \right] \right]
        = 0.
    \end{multline*}
    If we let
    \[
        x = \left( \tfrac{n}{2}v - \tfrac{m}{2}w \right) Y,\quad
        y = \left( \tfrac{l}{2}w - \tfrac{n}{2}u \right) Y,\quad
        z = \left( \tfrac{m}{2}u - \tfrac{l}{2}v \right) Y,
    \]
    then it suffices (using the fact that $\diamond$ is commutative) to show that
    \[
        \sinh(z-y) \sinh(x) + \sinh(x-z) \sinh(y) + \sinh(y-x) \sinh(z) = 0.
    \]
    This identity follows from expanding each $\sinh$ in terms of exponential functions, and seeing that all terms cancel.
}
We can extend the grading on $A$ to a Lie superalgebra grading on $\fD(A)$ by declaring $X$ to be even of degree $2d$.

Now, for $m \in \N$ and $f(X) \in \kk[X]$, define
\[
    f_{(m)}(X^2)
    := \sum_{j=1}^m f \left( \tfrac{m+1-2j}{2} X \right)
    \in \kk[X^2].
\]
By convention, $f_{(0)}(X^2) = 0$.  We then define $f_{(m)}(\kappa)$ to be the element of $Z(A)$ obtained from $f_{(m)}(X^2)$ by evaluating at $X^2 = \kappa$.  Define $\Psi \colon \fD(A) \times \fD(A) \to \kk$ by
\[
    \Psi \left( a f(X) z^m, b g(X) z^n \right)
    =
    \begin{cases}
        \tr \left( (fg)_{(m)}(\kappa) a \diamond b \right) & \text{if } m = -n \ge 0, \\
        - \tr \left( (fg)_{(n)}(\kappa) a \diamond b\right) & \text{if } m = -n < 0, \\
        0 & \text{if } m \ne -n.
    \end{cases}
\]
In terms of generating functions, we have
\[
    \Psi \left( a\exp(uX)z^m, b\exp(vX)z^n \right)
    = \delta_{m,-n} \tr \left( a \diamond b \left. \frac{\sinh \left( \tfrac{m}{2}(u+v) X \right)}{\sinh \left( \tfrac{1}{2} (u+v) X \right)} \right|_{X^2=\kappa} \right).
\]
(Note that, after cancelling a factor of $X$ from both the numerator and denominator of the ratio of hyperbolic sines, the denominator has nonzero constant term, hence is invertible.)
\details{
    For $m = -n \ge 0$, we have
    \[
        \Psi \left( a\exp(uX)z^m, b\exp(vX)z^n \right)
        = \tr \left( \exp_{(m)}((u+v)X)(\kappa) a \diamond b \right).
    \]
    Let $Y = \frac{1}{2}(u+v)X$.  Then we have
    \begin{align*}
        \frac{\sinh \left( \tfrac{m}{2}(u+v) X \right)}{\sinh \left( \tfrac{1}{2} (u+v) X \right)}
        &= \frac{\sinh(mY)}{\sinh(Y)} \\
        &= \frac{\exp(mY) - \exp(-mY)}{\exp(Y) - \exp(-Y)} \\
        &= \frac{\exp((m-1)Y) - \exp((-m-1)Y)}{1 - \exp(-2Y)} \\
        &= \left( \exp((m-1)Y) - \exp((-m-1)Y) \right) \left( \sum_{i=0}^\infty \exp(-2iY) \right) \\
        &= \exp((m-1)Y) + \exp((m-3)Y) + \dotsb + \exp((1-m)Y) \\
        &= \exp \left( \frac{m-1}{2} (u+v) X \right) + \exp \left( \frac{m-3}{2} (u+v) X \right) + \dotsb + \exp \left( \frac{1-m}{2} (u+v) X \right) \\
        &= \exp_{(m)}((u+v)X).
    \end{align*}
    The result follows.
}
It is a straightforward computation to verify that $\Psi$ is a 2-cocycle on $\fD(A)$.
\details{
    The map $\Psi$ is alternating by definition.  It is also clearly bilinear.  It remains to verify that it satisfies the Jacobi identity for 2-cycles.  We have
    \begin{multline*}
        \Psi \left( a\exp(uX)z^l, [b\exp(uX)z^m, c\exp(wV)z^n] \right)
        \\
        = \delta_{l+m+n,0} \tr \left( 2 a \diamond b \diamond c \left. \frac{\sinh \left( \left( \tfrac{n}{2}v - \tfrac{m}{2}w \right) X \right) \sinh \left( \tfrac{l}{2}(u+v+w)X \right)}{X \sinh \left( \tfrac{1}{2} (u+v+w) X \right)} \right|_{X^2=\kappa} \right).
    \end{multline*}
    Taking the cyclic sum with respect to $a,b,c$; $l,m,n$; and $u,v,w$, it suffices to show that
    \begin{multline*}
        \sinh \left( \left( \tfrac{n}{2}v - \tfrac{m}{2}w \right) X \right) \sinh \left( \tfrac{l}{2}(u+v+w)X \right)
        + \sinh \left( \left( \tfrac{l}{2}w - \tfrac{n}{2}u \right) X \right) \sinh \left( \tfrac{m}{2}(u+v+w)X \right)
        \\
        + \sinh \left( \left( \tfrac{m}{2}u - \tfrac{l}{2}v \right) X \right) \sinh \left( \tfrac{n}{2}(u+v+w)X \right)
        = 0
    \end{multline*}
    under the condition $l+m+n=0$.  This follows from expanding all $\sinh$ in terms of exponentials and seeing that all terms cancel.
}

We then define $\fW(A)$ to be the corresponding central extension of $\fD(A)$ by a one-dimensional vector space with generator $C$:
\[
    0 \to \kk C \to \fW(A) \to \fD(A) \to 0.
\]
Thus, the Lie bracket on $\fW(A)$ is given by
\begin{equation} \label{jupiter}
    [a f(X) z^m, b g(X) z^n]
    =
    (a \diamond b) \left( f(X) \star_{m,n} g(X) \right) z^{m+n} + \Psi \left( az^mf(X), bz^ng(X) \right) C,
\end{equation}
with the element $C$ being central.  In terms of generating functions, we have
\begin{multline} \label{mercury}
    [a\exp(uX)z^m, b\exp(vX)z^n]
    \\
    = 2 a \diamond b \left. \frac{\sinh \left( \left(\tfrac{n}{2} u - \tfrac{m}{2}v \right) Y \right)}{Y} \right|_{Y^2=\kappa} \exp(uX + vX) z^{m+n}
    + \delta_{m,-n} \tr \left( a \diamond b \left. \frac{\sinh \left( \tfrac{m}{2}(u+v) X \right)}{\sinh \left( \tfrac{1}{2} (u+v) X \right)} \right|_{X^2=\kappa} \right) C.
\end{multline}
We extended the grading on $\fD(A)$ to one on $\fW(A)$ by putting $C$ in degree zero.

Let $\rW(A)$ denote the universal enveloping algebra of $\fW(A)$.  For $m \in \Z$, $r \in \N$, and $a \in C(A)$, let
\[
    L_{m,r}(a)
    = a X^r z^m
    \in \fW(A).
\]
In addition to the degree grading, we define the \emph{rank} grading and \emph{order} filtration on $\fW(A)$ by setting
\begin{equation} \label{rankorder}
    \rank L_{m,r}(a) = m,\quad
    \ord L_{m,r}(a) = r,\quad
    \rank C = \ord C = 0.
\end{equation}
When $A=\kk$, these reduce to the usual rank grading and order filtration on $\fW_{1+\infty}$ (see \cite[F.2]{SV13}) under the isomorphism of \cref{rocks} below.

Note that, for $m,n \in \Z$, $r,s \in \N$, $a,b \in C(A)$,
\begin{align*}
    \Psi \left( L_{m,0}(a), L_{n,0}(b) \right)
    &= \delta_{m,-n} m \tr(a \diamond b), \\
    \Psi \left( L_{m,1}(a), L_{n,1}(b) \right)
    &= \delta_{m,-n} \frac{m^3-m}{12} \tr(\kappa a \diamond b), \\
    \Psi \left( L_{1,r}(a), L_{-1,s}(b) \right)
    &= \delta_{r+s,0} \tr(a \diamond b), \\
    \Psi \left( L_{0,r}(a), L_{n,s}(b) \right)
    &= 0, \\
    \Psi \left( L_{m,r}(a), L_{n,s}(b) \right)
    &= 0 \quad \text{when $r+s$ is odd}, \\
    \Psi \left( L_{m,r}(a), L_{n,s}(b) \right)
    &= 0 \quad \text{when $\kappa=0$ and $r+s > 0$}.
\end{align*}
Thus, for all $m,n \in \Z$, $r,s \in \N$, and $a,b \in C(A)$, we have
\begin{align}
    [L_{m,0}(a), L_{n,0}(b)] &= \delta_{m,-n} m \tr(a \diamond b) C, \label{comm1} \\
    [L_{m,1}(a), L_{n,1}(b)] &= (n-m) L_{m+n,1}(a \diamond b) + \delta_{m,-n} \frac{m^3-m}{12} \tr(\kappa a \diamond b) C, \label{comm2} \\
    [L_{m,1}(a), L_{n,0}(b)] &= n L_{m+n,0}(a \diamond b), \label{comm3} \\
    [L_{0,2}(a), L_{n,r}(b)] &= 2n L_{n,r+1}(a \diamond b),  \label{comm4} \\
    [L_{1,r}(a), L_{-1,s}(b)] &= -\sum_{i=0}^{\left\lfloor \frac{r+s-1}{2} \right\rfloor} \binom{r+s}{2i+1} \left( \frac{\kappa}{4} \right)^i L_{0,r+s-2i-1}(a \diamond b) + \delta_{r+s,0} \tr(a \diamond b) C, \label{comm5} \\
    [L_{1,0}(a), L_{n,r}(b)] &= - \sum_{i=0}^{\left\lfloor \frac{r-1}{2} \right\rfloor} \binom{r}{2i+1} \left( \frac{\kappa}{4} \right)^i L_{n+1,r-2i-1}(a \diamond b) + \delta_{n,-1} \delta_{r,0} \tr(a \diamond b) C. \label{comm6}
\end{align}

We see from \cref{comm1} that algebra generated by the $L_{m,0}(a)$, $m \in \Z$, $a \in C(A)$, is an analogue of an oscillator algebra.  In particular, the $L_{m,0}(a)$, $m \in \Z \setminus \{0\}$, $a \in C(A)$, generate a lattice Heisenberg algebra based on the lattice $C(A)$ with bilinear form $(a,b) \mapsto \tr(a \diamond b)$.  (See \cite[\S2.1]{LRS18}.)  In addition, we see from \cref{comm2} that the $L_{m,1}(a)$, $m \in \Z$, $a \in C(A)$, generate an analogue of a Virasoro algebra.

\begin{rem}[Change of trace map] \label{changeup}
    Suppose $(A,\tr_1)$ and $(A,\tr_2)$ are symmetric graded Frobenius superalgebras with the same underlying graded superalgebra $A$.  As in \cref{ifislip}, we use subscripts to distinguish between the two, and define $e \in A$ as in \cref{slip}.  It is a straightforward computation to verify that we have an isomorphism of Lie algebras
    \begin{equation}
        \fW_2(A) \xrightarrow{\cong} \fW_1(A),\quad
        a f(X) z^m \mapsto a f(eX) z^m,\quad
        C \mapsto C.
    \end{equation}
    \details{
        Consider the map
        \[
            \psi \colon \fD_2(A) \to \fD_1(A),\quad
            a f(X) z^m \mapsto a f(eX) z^m.
        \]
        Then we have
        \begin{align*}
            [ \psi(af(X)z^m),{}& \psi(bg(X)z^n) ]_1
            \\
            &= [af(eX)z^m, bg(eX)z^n]_1
            \\
            &= a \diamond_1 b \left. \frac{f \left( eX + \tfrac{n}{2} eY \right) g \left( eX -\tfrac{m}{2} eY \right) - f \left( eX - \tfrac{n}{2} eY \right) g \left( eX + \tfrac{m}{2} eY \right)}{Y} \right|_{Y^2 = \kappa_1} z^{m+n}
            \\
            &= a \diamond_2 b \left. \frac{f \left( eX + \tfrac{n}{2} Y \right) g \left( eX -\tfrac{m}{2} Y \right) - f \left( eX - \tfrac{n}{2} Y \right) g \left( eX + \tfrac{m}{2} Y \right)}{Y} \right|_{Y^2 = \kappa_2} z^{m+n}
            \\
            &= \psi \left( [af(X)z^m, bg(X)z^n] \right).
        \end{align*}
        Hence $\psi$ is an isomorphism of Lie algebras.  Furthermore, if $m=-n \ge 0$, we have
        \begin{align*}
            \Psi_1 \left( \psi(af(X)z^m), \psi(bg(X)z^n) \right)
            &= \Psi_1 \left( af(eX)z^m, bg(eX)z^n \right)
            \\
            &= \tr_1 \left( (fg)_{(m)}(e^2 \kappa_1) a \diamond_1 b \right)
            \\
            &\overset{\mathclap{\cref{slip}}}{\underset{\mathclap{\cref{slippin}}}{=}}\ \tr_2 \left( (fg)_{(m)}(\kappa_2) a \diamond_2 b \right)
            \\
            &= \Psi_2 \left( af(X)z^m, bg(X)z^n \right).
        \end{align*}
        The claim follows.
    }
    Hence $\fW(A)$ and $\rW(A)$ are independent of the trace map on $A$, up to isomorphism.  However, different choices of trace map give rise to different basis elements $L_{m,r}(a)$ of $\fW(A)$, and hence different presentations of $\rW(A)$.
\end{rem}

\subsection{Special cases}

We now discuss some special cases of the Frobenius W-algebra $\fW(A)$ that are of particular importance.  First, as we now explain, $\fW(A)$ is a Frobenius algebra generalization of the W-algebra $\fW_{1+\infty}$, which is the unique, up to isomorphism, central extension of the Lie algebra of regular differential operators on the circle.  The Lie algebra $\fW_{1+\infty}$ has basis $\{C, w_{n,r} : n \in \Z,\ r \in \Z\}$ and, given formal variables $u$ and $v$, the relations are given by $w_{n,r} = D^r z^n$ with
\[
    [ \exp(uD)z^m, \exp(vD) z^n]
    = \left( \exp(nu) - \exp(mv) \right) \exp(uD + vD) z^{m+n} + \delta_{m,-n} \frac{\exp(-mu) - \exp(-nv)}{1-\exp(u+v)} C.
\]
(See \cite[(2.2.2)]{KR93} and \cite[F.2]{SV13}.)

\begin{prop} \label{rocks}
    The linear map $\fW(\kk) \to \fW_{1+\infty}$ determined by
    \[
        X^r z^n \mapsto \left( D + \tfrac{n+1}{2} \right)^r z^n,\quad
        C \mapsto C,
    \]
    is an isomorphism of Lie algebras.
\end{prop}

\begin{proof}
    We have
    \begin{align*}
        [\exp(uX) z^m, \exp(vX) z^n]
        &= \exp \left( \tfrac{m+1}{2}u + \tfrac{n+1}{2}v \right) [\exp(uD) z^m, \exp(vD) z^n] \\
        &= \exp \left( \tfrac{m+1}{2}u + \tfrac{n+1}{2}v \right) \left( \exp(nu) - \exp(mv) \right) \exp(uD + vD) z^{m+n} \\
        &\qquad \qquad \qquad \qquad
        + \delta_{m,-n} \exp \left( \tfrac{m+1}{2}u + \tfrac{n+1}{2}v \right) \frac{\exp(-mu) - \exp(-nv)}{1-\exp(u+v)} C \\
        &= \left( \exp \left( \tfrac{n}{2} u - \tfrac{m}{2} v \right) - \exp \left( \tfrac{m}{2} v - \tfrac{n}{2} u \right) \right) \exp(uX + vX) z^{m+n} \\
        &\qquad \qquad \qquad \qquad
        + \delta_{m,-n} \frac{\exp \left( \tfrac{u+v}{2} \right) \left( \exp \left( \tfrac{n}{2} v - \tfrac{m}{2} u \right) - \exp \left( \tfrac{m}{2} u - \tfrac{n}{2} v \right) \right)}{1 - \exp(u+v)} C \\
        &= 2 \sinh \left( \tfrac{n}{2} u - \tfrac{m}{2} v \right) \exp(uX + vX) z^{m+n} + \delta_{m,-n} \frac{\sinh \left( \tfrac{m}{2} (u+v) \right)}{\sinh \left( \tfrac{u+v}{2} \right)} C.
    \end{align*}
    Comparing to \cref{mercury} and noting that $\kappa=1$ completes the proof.
\end{proof}

\begin{rem}
    The linear map defined in \cref{rocks} is related to the family of automorphisms $\Theta_s$, $s \in \kk$, of $\fW_{1+\infty}$ defined in \cite[(2.21)]{KWY98}, which fixes $z$ and sends $D \mapsto D+s$.  In light of \cref{rocks}, the map $\Theta_{\frac{n+1}{2}}$ can be viewed as an automorphism of $\fW_{1+\infty}$ sending $D$ to $X$.
\end{rem}

In many cases of interest, the presentation of $\fW(A)$ simplifies considerably, as we now explain.

\begin{prop} \label{Wangmatch}
    If $A$ is semisimple and $\kk$ is algebraically closed, then
    \[
        \fW(A) \cong  \left( C(A) \otimes \fW_{1+\infty} \right)/(aC - \tr(a)C : a \in C(A))
    \]
    as graded superalgebras, where $C(A) \otimes \fW_{1+\infty}$ is endowed with the Lie bracket
    \[
        [a \otimes x, b \otimes y] = (a \diamond b) [x,y],\quad a,b \in C(A),\ x,y \in \fW_{1+\infty}.
    \]
\end{prop}

\begin{proof}
    By \cref{changeup,semisimple}, we may assume $\kappa=1$.  The result then follows from \cref{jupiter}, as in the proof of \cref{rocks}.
\end{proof}

\begin{rem}
    \Cref{Wangmatch} implies that $\fW(\kk \Gamma)$ is isomorphic to the W-algebra introduced by Wang \cite{Wan04} in his study of wreath products.  However, if one chooses the trace map that is projection onto the identity element, as in \cref{groupalg}, then $\kappa = \sum_{g,h \in \Gamma} ghg^{-1}h^{-1}$, and so the generating elements $L_{m,r}(a)$ yield a very different presentation of $\fW(\kk \Gamma)$.
\end{rem}

\begin{eg}[Case $\kappa=0$] \label{kapzero}
    Suppose $\kappa=0$.  For instance, this is the case whenever $A$ is nontrivially positively graded (in particular, when $A$ is a skew-zigzag algebra).  Then we have
    \[
        X^r \star_{m,n} X^s = (rn-sm) X^{r+s-1},
    \]
    and \cref{mercury} becomes
    \begin{equation} \label{pluto}
        [a\exp(uX)z^m, b\exp(vX)z^n]
        = (a \diamond b) \left( nu - mv \right) \exp(uX + vX) z^{m+n} + m \delta_{m,-n} \tr(a \diamond b) C.
    \end{equation}
    In particular,
    \begin{equation} \label{dragon}
        [L_{m,r}(a), L_{n,s}(b)]
        = (rn-sm) L_{m+n,r+s-1}(a \diamond b) + m \delta_{m,-n} \delta_{r+s,0} \tr(a \diamond b) C
        \quad \text{when } \kappa=0.
    \end{equation}
\end{eg}

\subsection{Basis}

Fix an ordered homogeneous basis $\bB_{C(A)}$ of $C(A)$. Note that $\fW(A)$ has a linear basis given by $\{L_{n,r}(a) : n \in \Z, r \in \N, a \in \bB_{C(A)}$\}.

\begin{defin} \label{order}
    Let $\vartriangleright$ be the order on $\Z$ given by
    \[
        \dotsb \vartriangleright -2 \vartriangleright -1 \vartriangleright \dotsb 3 \vartriangleright 2 \vartriangleright 1 \vartriangleright 0.
    \]
    In other words,
    \begin{itemize}
        \item $n \vartriangleright m$ if $n < 0$ and $m \ge 0$,
        \item $n \vartriangleright m$ if $n < m < 0$, and
        \item $n \vartriangleright m$ if $n > m \ge 0$.
    \end{itemize}
    Then let $\vartriangleright$ also denote the lexicographic order on $\Z \times \N \times \bB_{C(A)}$, using the order $\vartriangleright$ on $\Z$, the usual order $>$ on $\N$, and the given order $>$ on the basis $\bB_{C(A)}$.  Finally, define the relation $\gtrdot$ on $\Z \times \N \times \bB_{C(A)}$ by declaring that $(n_1,r_1,a_1) \gtrdot (n_2,r_2,a_2)$ if
    \begin{itemize}
        \item $(n_1,r_1,a_1) \vartriangleright (n_2,r_2,a_2)$, or
        \item $\bar{a_1} = 0$ and $(n_1,r_1,a_1) = (n_2,r_2,a_2)$.
    \end{itemize}
\end{defin}

Then the PBW Theorem for Lie superalgebras implies the following proposition.

\begin{prop} \label{Wbasis}
    Fix $k \in \Z$.  As a $\kk$-module, the algebra $W(A)/(C-k)$ has basis
    \begin{multline*}
        \{L_{n_1,r_1}(a_1) L_{n_2,r_2}(a_2) \dotsm L_{n_t,r_t}(a_t) :
        t \in \N,\ (n_i,r_i,a_i) \in \Z \times \N \times \bB_{C(A)},\\
        (n_1,r_1,a_1) \gtrdot (n_2,r_2,a_2) \gtrdot \dotsb \gtrdot (n_t,r_t,a_t)\}.
    \end{multline*}
\end{prop}

\subsection{Symmetries}

Note that, for $f(X), g(X) \in \kk[X]$ and $m,n \in \Z$, we have
\[
  f(X) \star_{-m,-n} g(X) = - g(X) \star_{m,n} f(X)
  \quad \text{and} \quad
  \Psi( a z^{-m} f(X), b z^{-n} g(X) ) = - \Psi( az^mf(X), bz^ng(X) ).
\]
It follows that we have an automorphism of graded algebras
\begin{equation} \label{omega-alg}
  \omega \colon W(A) \xrightarrow{\cong} W(A),\quad
  a f(X) z^m \mapsto (-1)^{m-1} a f(X) z^{-m},\quad
  C \mapsto -C.
\end{equation}
\details{
  We have
  \begin{align*}
    [ \omega(af(X))z^m&, \omega(bg(X)z^n) ] \\
    &= (-1)^{n+m-2} [ a f(X) z^{-m}, b g(X) z^{-n} ] \\
    &= (-1)^{n+m-2} (a \diamond b) \left( f(X) \star_{-m,-n} g(X) \right) z^{-m-n} + (-1)^{n+m-2} \Psi( a z^{-m} f(X), b z^{-n} g(X) ) C \\
    &= (-1)^{n+m-1} (a \diamond b) \left( f(X) \star_{m,n} g(X) \right) z^{-m-n} - \Psi( a z^m f(X), b z^n g(X) ) C \\
    &= \omega \left( [af(X)z^m, bg(X)z^n] \right),
  \end{align*}
  where, in the third equality, we used the fact that $\Psi( a f(X) z^{-m}, b g(X) z^{-n} ) = 0$ unless $n=-m$.
}
We also have an isomorphism of graded algebras
\begin{equation} \label{phi-alg}
  \varphi \colon W(A) \xrightarrow{\cong} W(A)^\op,\quad
  a f(X) z^n \mapsto a f(X) z^{-n},\quad
  C \mapsto C.
\end{equation}
\details{
  We have
  \begin{align*}
    [ \varphi(af(X)z^m), \varphi(bg(X)z^n) ]
    &= [ a f(X) z^{-m}, b g(X) z^{-n} ] \\
    &= (a \diamond b) \left( f(X) \star_{-m,-n} g(X) \right) z^{-m-n} + \Psi( a z^{-m} f(X), b z^{-n} g(X) ) C \\
    &= - (a \diamond b) \left( f(X) \star_{m,n} g(X) \right) z^{-m-n} - \Psi( a z^m f(X), b z^n g(X) ) C \\
    &= - \varphi \left( [af(X)z^m, bg(X)z^n] \right) \\
    &= (-1)^{\bar{a} + \bar{b}} \varphi \left( [ bg(X)z^n, af(X)z^m ] \right).
  \end{align*}
}
In particular, we have
\begin{equation} \label{noodles}
  \omega(L_{n,r}(a)) = (-1)^{n-1} L_{-n,r}(a)
  \quad \text{and} \quad
  \varphi(L_{n,r}(a)) = L_{-n,r}(a).
\end{equation}

\section{The affine wreath product category}

We recall here the affine wreath product category, which is the main building block of the Frobenius Heisenberg category to be discussed in \cref{FHC}.  We continue to fix a graded Frobenius superalgebra $A$ with trace map of degree $-2d$.

\subsection{Affine wreath product algebras}

View $\kk[x_1,\dotsc,x_n]$ as a graded algebra with each $x_i$ even of degree $2d$.  Let $P_n(A)$ denote the graded algebra $A^{\otimes n} \otimes \kk[x_1,\dotsc,x_n]$ (tensor product of graded superalgebras), i.e.\ the algebra of polynomials in the commuting variables $x_1,\dotsc,x_n$ with coefficients in $A^{\otimes n}$.  Let $s_{i,j} \in S_n$ be the transposition of $i$ and $j$, and define $s_i := s_{i,i+1}$.  The symmetric group $S_n$ acts on $A^{\otimes n}$ by superpermutations; so that
\[
  s_i (a_n \otimes \dotsb \otimes a_1)
  = (-1)^{\bar{a_i} \bar{a}_{i+1}} a_n \otimes \dotsb \otimes a_{i+2} \otimes a_i \otimes a_{i+1} \otimes a_{i-1} \otimes \dotsb \otimes a_1.
\]
Note that we number factors from right to left.

For $a \in A$ and $1 \le i \le n$, define
\begin{equation} \label{aidef}
    a^{(i)} := 1^{\otimes (n-i)} \otimes a \otimes 1^{\otimes (i-1)} \in A^{\otimes n}.
\end{equation}
Then, for $1 \le i < j \le n$, define
\begin{equation} \label{tau}
  \tau_{i,j}
  := \sum_{b \in \bB_A} b^{(j)} \left( b^\vee \right)^{(i)},\quad
  \tau_i := \tau_{i,i+1}.
\end{equation}
Note that the $\tau_{i,j}$ do not depend on the choice of basis $\bB_A$.  It follows from \cref{contract} that
\begin{equation} \label{travel}
    \tau_{i,j} \ba = s_{i,j}(\ba) \tau_{i,j},
    \quad \ba \in A^{\otimes n}.
\end{equation}
In particular, this gives that $\tau_i \ba = s_i(\ba) \tau_i$.

For $i=1,\dots,n-1$, define the \emph{Demazure operator} $\partial_i \colon P_n(A) \rightarrow P_n(A)$ to be the even degree zero homogeneous linear map defined by
\begin{equation} \label{Demazure2}
    \partial_i (\ba p)
    := \tau_i \ba \frac{p-s_i(p)}{x_{i+1}-x_i} = s_i(\ba) \tau_i
    \frac{p-s_i(p)}{x_{i+1}-x_i},
    \quad \ba \in A^{\otimes n},\ p \in \kk[x_1,\dots,x_n].
\end{equation}
Equivalently, we have that
\begin{equation} \label{Demazure}
    \partial_i(f) = \frac{\tau_i f - s_i(f) \tau_i}{x_{i+1}-x_i},
\quad
f \in P_n(A).
\end{equation}
Note also that the Demazure operators satisfy the twisted Leibniz identity
\begin{equation}
  \partial_i(fg) = \partial_i(f) g + s_i(f) \partial_i(g),\quad
  f,g \in P_n(A).
\end{equation}
(See also \cite[\S4.1]{Sav20}.)

The \emph{affine wreath product algebra} $\AWA_n(A)$ is the graded vector superspace $P_n(A) \otimes \kk S_n$ viewed as a graded superalgebra with multiplication defined so that $P_n(A)$ and $\kk S_n$ are subsuperalgebras and
\begin{equation} \label{sif}
  s_i f = s_i(f) s_i + \partial_i(f),\quad
  f \in P_n(A),\ i \in \{1,2,\dotsc,n-1\}.
\end{equation}
Up to isomorphism, $\AWA_n(A)$ depends only on the underlying graded superalgebra $A$, and not on the trace map; see \cite[Lem.~3.2]{Sav20}.

We have a filtration of graded superalgebras of $\AWA_n(A)$ given by
\begin{equation} \label{tart}
    \AWA_n(A)_{0} \subseteq \AWA_n(A)_{\le 1} \subseteq \AWA_n(A)_{\le 2} \subseteq \dotsb,
\end{equation}
where $\AWA_n(A)_{\le r}$ is the subsuperalgebra of $\AWA_n(A)$ spanned by elements $f \otimes \pi$, where $\pi \in \kk S_n$ and $f \in P_n(A)$ has polynomial degree (i.e.\ degree in the $x_i$) less than or equal to $r$.

\subsection{The affine wreath product category}

The \emph{affine wreath product category} $\AWC$ is a strict graded monoidal category whose nonzero morphism spaces are the algebras $\AWA_n(A)$ for all $n \geq 0$. Formally, it is the strict $\kk$-linear graded monoidal category generated by one object $\uparrow$, morphisms
\[
    \begin{tikzpicture}[anchorbase]
      \draw[->] (0,0) -- (0,0.6);
      \singdot{0,0.3};
    \end{tikzpicture}
    \colon \uparrow \to \uparrow
    \ ,\quad
    \begin{tikzpicture}[anchorbase]
      \draw [->](0,0) -- (0.6,0.6);
      \draw [->](0.6,0) -- (0,0.6);
    \end{tikzpicture}
    \colon \uparrow \otimes \uparrow\ \to \uparrow \otimes \uparrow
    \ ,\quad
    \begin{tikzpicture}[anchorbase]
      \draw[->] (0,0) -- (0,0.6);
      \token{west}{0,0.3}{a};
    \end{tikzpicture}
    \colon \uparrow \to \uparrow
    \ ,\ a \in A,
\]
of degrees
\begin{equation} \label{updegs}
    \deg \left(
        \begin{tikzpicture}[anchorbase]
            \draw[->] (0,0) -- (0,0.6);
            \singdot{0,0.3};
        \end{tikzpicture}
    \right) = 2d,
    \quad
    \deg \left(
        \begin{tikzpicture}[anchorbase]
          \draw [->](0,0) -- (0.6,0.6);
          \draw [->](0.6,0) -- (0,0.6);
        \end{tikzpicture}
    \right) = 0,
    \quad
    \deg \left(
        \begin{tikzpicture}[anchorbase]
          \draw[->] (0,0) -- (0,0.6);
          \token{west}{0,0.3}{a};
        \end{tikzpicture}
    \right) = \deg a,
\end{equation}
parities $0$, $0$, and $\bar{a}$ (respectively), and subject to the relations
\begin{gather} \label{tokrel}
    \begin{tikzpicture}[anchorbase]
        \draw[->] (0,0) -- (0,0.7);
        \token{east}{0,0.35}{1};
    \end{tikzpicture}
    =
    \begin{tikzpicture}[anchorbase]
        \draw[->] (0,0) -- (0,0.7);
    \end{tikzpicture}
    ,\quad \lambda\:
    \begin{tikzpicture}[anchorbase]
        \draw[->] (0,0) -- (0,0.7);
        \token{east}{0,0.35}{a};
    \end{tikzpicture}
    + \mu\:
    \begin{tikzpicture}[anchorbase]
        \draw[->] (0,0) -- (0,0.7);
        \token{west}{0,0.35}{b};
    \end{tikzpicture}
    =
    \begin{tikzpicture}[anchorbase]
        \draw[->] (0,0) -- (0,0.7);
        \token{west}{0,0.35}{\lambda a +  \mu b};
    \end{tikzpicture}
    ,\quad
    \begin{tikzpicture}[anchorbase]
        \draw[->] (0,0) -- (0,0.7);
        \token{east}{0,0.2}{b};
        \token{east}{0,0.45}{a};
    \end{tikzpicture}
    =
    \begin{tikzpicture}[anchorbase]
        \draw[->] (0,0) -- (0,0.7);
        \token{west}{0,0.35}{ab};
    \end{tikzpicture}
    ,\quad
    \begin{tikzpicture}[anchorbase]
        \draw[->] (-0.3,-0.3) -- (0.3,0.3);
        \draw[->] (0.3,-0.3) -- (-0.3,0.3);
        \token{east}{-0.15,-0.15}{a};
    \end{tikzpicture}
    =
    \begin{tikzpicture}[anchorbase]
        \draw[->] (-0.3,-0.3) -- (0.3,0.3);
        \draw[->] (0.3,-0.3) -- (-0.3,0.3);
        \token{west}{0.1,0.1}{a};
    \end{tikzpicture}
    ,\quad a,b \in A,\ \lambda,\mu \in \kk,
    \\ \label{braidrel}
    \begin{tikzpicture}[anchorbase]
        \draw[->] (0,0) -- (1,1);
        \draw[->] (1,0) -- (0,1);
        \draw[->] (0.5,0) .. controls (0,0.5) .. (0.5,1);
    \end{tikzpicture}
    =
    \begin{tikzpicture}[anchorbase]
        \draw[->] (0,0) -- (1,1);
        \draw[->] (1,0) -- (0,1);
        \draw[->] (0.5,0) .. controls (1,0.5) .. (0.5,1);
    \end{tikzpicture}
    ,\quad
    \begin{tikzpicture}[anchorbase]
        \draw[->] (0,0) \braidto (0.5,0.5) \braidto (0,1);
        \draw[->] (0.5,0) \braidto (0,0.5) \braidto (0.5,1);
    \end{tikzpicture}
    =
    \begin{tikzpicture}[anchorbase]
        \draw[->] (0,0) --(0,1);
        \draw[->] (0.5,0) -- (0.5,1);
    \end{tikzpicture},
    \\ \label{affrel}
    \begin{tikzpicture}[anchorbase]
        \draw[->] (-0.3,-0.3) -- (0.3,0.3);
        \draw[->] (0.3,-0.3) -- (-0.3,0.3);
        \singdot{-0.12,0.12};
    \end{tikzpicture}
    \ -\
    \begin{tikzpicture}[anchorbase]
        \draw[->] (-0.3,-0.3) -- (0.3,0.3);
        \draw[->] (0.3,-0.3) -- (-0.3,0.3);
        \singdot{0.15,-0.15};
    \end{tikzpicture}
    = \sum_{b \in \bB_A}
    \begin{tikzpicture}[anchorbase]
        \draw[->] (-0.2,-0.3) -- (-0.2,0.3);
        \draw[->] (0.2,-0.3) -- (0.2,0.3);
        \token{east}{-0.2,0}{b};
        \token{west}{0.2,0}{b^\vee};
    \end{tikzpicture}
    \ ,\quad
    \begin{tikzpicture}[anchorbase]
        \draw[->] (0,0) -- (0,0.7);
        \singdot{0,0.2};
        \token{east}{0,0.4}{a};
    \end{tikzpicture}
    =
    \begin{tikzpicture}[anchorbase]
        \draw[->] (0,0) -- (0,0.7);
        \singdot{0,0.4};
        \token{west}{0,0.2}{a};
    \end{tikzpicture}
    \ ,\quad a \in A.
\end{gather}

We refer to the morphisms
$
  \begin{tikzpicture}[anchorbase]
    \draw[->] (0,0) -- (0,0.4);
    \token{west}{0,0.2}{a};
  \end{tikzpicture}
$
as \emph{tokens} and the morphisms
$
  \begin{tikzpicture}[anchorbase]
    \draw[->] (0,0) -- (0,0.4);
    \singdot{0,0.2};
  \end{tikzpicture}
$
as \emph{dots}.  It is immediate from this definition that there is an isomorphism of graded algebras
\begin{equation} \label{imathaff}
  \imath_n \colon \AWA_n(A) \to \End_{\AWC}(\uparrow^{\otimes n})
\end{equation}
sending $s_i$ to a crossing of the $i$-th and $(i+1)$-st strings, $x_j$ to a dot on the $j$-th string, and $a^{(j)}$ to a token labelled $a$ on the $j$-th string.  Here we number strings by $1,2,\dotsc$ from right to left.  Using $\imath_n$, we can \emph{identify} the algebra $\AWA_n(A)$ with $\End_{\AWC}(\uparrow^{\otimes n})$.  Since $\Hom_{\AWC}(\uparrow^{\otimes m}, \uparrow^{\otimes n}) = 0$ when $m \ne n$, we have an isomorphism of graded $\kk$-modules
\begin{equation} \label{beatdown}
    (\cocenter{\imath_n})_{n \in \N} \colon \bigoplus_{n \in \N} C(\AWA_n(A)) \xrightarrow{\cong} \Tr(\AWC).
\end{equation}

We refer to the morphisms in $\AWC$ defined by the elements $\tau_{i,j} \in \AWA_n(A)_{2d}$ from \cref{tau} as \emph{teleporters}, and denote them by a line segment connecting tokens on strings $i$ and $j$:
\begin{equation}\label{brexit}
  \begin{tikzpicture}[anchorbase]
    \draw[->] (0,-0.4) --(0,0.4);
    \draw[->] (0.5,-0.4) -- (0.5,0.4);
    \teleport{0,0}{0.5,0};
  \end{tikzpicture}
  =
  \begin{tikzpicture}[anchorbase]
    \draw[->] (0,-0.4) --(0,0.4);
    \draw[->] (0.5,-0.4) -- (0.5,0.4);
    \teleport{0,0.2}{0.5,-0.2};
  \end{tikzpicture}
  =
  \begin{tikzpicture}[anchorbase]
    \draw[->] (0,-0.4) --(0,0.4);
    \draw[->] (0.5,-0.4) -- (0.5,0.4);
    \teleport{0,-0.2}{0.5,0.2};
  \end{tikzpicture}
  = \sum_{b \in \bB_A}
  \begin{tikzpicture}[anchorbase]
    \draw[->] (0,-0.4) --(0,0.4);
    \draw[->] (0.5,-0.4) -- (0.5,0.4);
    \token{east}{0,0.15}{b};
    \token{west}{0.5,-0.15}{b^\vee};
  \end{tikzpicture}
  = \sum_{b \in \bB_A}
  \begin{tikzpicture}[anchorbase]
    \draw[->] (0,-0.4) --(0,0.4);
    \draw[->] (0.5,-0.4) -- (0.5,0.4);
    \token{east}{0,-0.15}{b^\vee};
    \token{west}{0.5,0.15}{b};
  \end{tikzpicture}\ .
\end{equation}
Note that we do not insist that the tokens in a teleporter are drawn at the same horizontal level, the convention when this is not the case being that $b$ is on the higher of the tokens and $b^\vee$ is on the lower one.  We will also draw teleporters in larger diagrams.  When doing so, we add a sign of $(-1)^{y \bar{b}}$ in front of the $b$ summand in \cref{brexit}, where $y$ is the sum of the parities of all morphisms in the diagram vertically between the tokens labeled $b$ and $b^\vee$.  For example,
\[
  \begin{tikzpicture}[anchorbase]
    \draw[->] (-1,-0.4) -- (-1,0.4);
    \draw[->] (-0.5,-0.4) -- (-0.5,0.4);
    \draw[->] (0,-0.4) -- (0,0.4);
    \draw[->] (0.5,-0.4) -- (0.5,0.4);
    \token{east}{-1,0}{a};
    \token{west}{0,0.1}{c};
    \teleport{-0.5,0.2}{0.5,-0.2};
  \end{tikzpicture}
  =
  \sum_{b \in \bB_A} (-1)^{(\bar{a} + \bar{c}) \bar{b}}
  \begin{tikzpicture}[anchorbase]
    \draw[->] (-1,-0.4) -- (-1,0.4);
    \draw[->] (-0.5,-0.4) -- (-0.5,0.4);
    \draw[->] (0,-0.4) -- (0,0.4);
    \draw[->] (0.5,-0.4) -- (0.5,0.4);
    \token{east}{-1,0}{a};
    \token{west}{0,0.1}{c};
    \token{east}{-0.5,0.2}{b};
    \token{west}{0.5,-0.2}{b^\vee};
  \end{tikzpicture}
  \ .
\]
This convention ensures that one can slide the endpoints of teleporters along strands:
\[
  \begin{tikzpicture}[anchorbase]
    \draw[->] (-1,-0.4) -- (-1,0.4);
    \draw[->] (-0.5,-0.4) -- (-0.5,0.4);
    \draw[->] (0,-0.4) -- (0,0.4);
    \draw[->] (0.5,-0.4) -- (0.5,0.4);
    \token{east}{-1,0}{a};
    \token{west}{0,0.1}{c};
    \teleport{-0.5,0.2}{0.5,-0.2};
  \end{tikzpicture}
  =
  \begin{tikzpicture}[anchorbase]
    \draw[->] (-1,-0.4) -- (-1,0.4);
    \draw[->] (-0.5,-0.4) -- (-0.5,0.4);
    \draw[->] (0,-0.4) -- (0,0.4);
    \draw[->] (0.5,-0.4) -- (0.5,0.4);
    \token{east}{-1,0}{a};
    \token{east}{0,0.1}{c};
    \teleport{-0.5,-0.2}{0.5,0.2};
  \end{tikzpicture}
  =
  \begin{tikzpicture}[anchorbase]
    \draw[->] (-1,-0.4) -- (-1,0.4);
    \draw[->] (-0.5,-0.4) -- (-0.5,0.4);
    \draw[->] (0,-0.4) -- (0,0.4);
    \draw[->] (0.5,-0.4) -- (0.5,0.4);
    \token{east}{-1,0}{a};
    \token{west}{0,0.1}{c};
    \teleport{-0.5,0.2}{0.5,0.2};
  \end{tikzpicture}
  =
  \begin{tikzpicture}[anchorbase]
    \draw[->] (-1,-0.4) -- (-1,0.4);
    \draw[->] (-0.5,-0.4) -- (-0.5,0.4);
    \draw[->] (0,-0.4) -- (0,0.4);
    \draw[->] (0.5,-0.4) -- (0.5,0.4);
    \token{east}{-1,0}{a};
    \token{west}{0,0.1}{c};
    \teleport{-0.5,-0.2}{0.5,-0.2};
  \end{tikzpicture}
  \ .
\]

Using teleporters, the first relation in \cref{affrel} can now be written as
\begin{equation} \label{upslides1}
  \begin{tikzpicture}[anchorbase]
    \draw[->] (-0.3,-0.3) -- (0.3,0.3);
    \draw[->] (0.3,-0.3) -- (-0.3,0.3);
    \singdot{-0.12,0.12};
  \end{tikzpicture}
  \ -\
  \begin{tikzpicture}[anchorbase]
    \draw[->] (-0.3,-0.3) -- (0.3,0.3);
    \draw[->] (0.3,-0.3) -- (-0.3,0.3);
    \singdot{0.15,-0.15};
  \end{tikzpicture}
  \ =\
  \begin{tikzpicture}[anchorbase]
    \draw[->] (-0.2,-0.3) -- (-0.2,0.3);
    \draw[->] (0.2,-0.3) -- (0.2,0.3);
    \teleport{-0.2,0}{0.2,0};
  \end{tikzpicture}
\end{equation}
Composing with the crossing on the top and bottom, we see that we also have the relation
\begin{equation} \label{upslides2}
  \begin{tikzpicture}[anchorbase]
    \draw[->] (-0.3,-0.3) -- (0.3,0.3);
    \draw[->] (0.3,-0.3) -- (-0.3,0.3);
    \singdot{-0.12,-0.12};
  \end{tikzpicture}
  \ -\
  \begin{tikzpicture}[anchorbase]
    \draw[->] (-0.3,-0.3) -- (0.3,0.3);
    \draw[->] (0.3,-0.3) -- (-0.3,0.3);
    \singdot{0.15,0.15};
  \end{tikzpicture}
  \ =\
  \begin{tikzpicture}[anchorbase]
    \draw[->] (-0.2,-0.3) -- (-0.2,0.3);
    \draw[->] (0.2,-0.3) -- (0.2,0.3);
    \teleport{-0.2,0}{0.2,0};
  \end{tikzpicture}
  \ .
\end{equation}
It follows from \cref{travel} that tokens can ``teleport'' across teleporters (justifying the terminology) in the sense that, for $a \in A$, we have
\begin{equation} \label{tokteleport}
  \begin{tikzpicture}[anchorbase]
    \draw[->] (0,-0.5) --(0,0.5);
    \draw[->] (0.5,-0.5) -- (0.5,0.5);
    \token{west}{0.5,-0.25}{a};
    \teleport{0,0}{0.5,0};
  \end{tikzpicture}
=  \begin{tikzpicture}[anchorbase]
    \draw[->] (0,-0.5) --(0,0.5);
    \draw[->] (0.5,-0.5) -- (0.5,0.5);
    \token{east}{0,0.25}{a};
    \teleport{0,0}{0.5,0};
  \end{tikzpicture}
   \ ,
\qquad
  \begin{tikzpicture}[anchorbase]
    \draw[->] (0,-0.5) --(0,0.5);
    \draw[->] (0.5,-0.5) -- (0.5,0.5);
    \token{east}{0,-0.25}{a};
    \teleport{0,0}{0.5,0};
  \end{tikzpicture}
  =
  \begin{tikzpicture}[anchorbase]
    \draw[->] (0,-0.5) --(0,0.5);
    \draw[->] (0.5,-0.5) -- (0.5,0.5);
    \token{west}{0.5,0.25}{a};
    \teleport{0,0}{0.5,0};
  \end{tikzpicture}
  \ ,
\end{equation}
where the strings can occur anywhere in a diagram (i.e.\ they do not need to be adjacent).  The endpoints of teleporters slide through crossings and dots, and they can teleport too.  For example we have
\begin{equation} \label{laser}
  \begin{tikzpicture}[anchorbase]
    \draw[->] (-0.4,-0.4) to (0.4,0.4);
    \draw[->] (0.4,-0.4) to (-0.4,0.4);
    \draw[->] (0.8,-0.4) to (0.8,0.4);
    \teleport{-0.15,0.15}{0.8,0.15};
  \end{tikzpicture}
  \ =\
  \begin{tikzpicture}[anchorbase]
    \draw[->] (-0.4,-0.4) to (0.4,0.4);
    \draw[->] (0.4,-0.4) to (-0.4,0.4);
    \draw[->] (0.8,-0.4) to (0.8,0.4);
    \teleport{0.2,-0.2}{0.8,-0.2};
  \end{tikzpicture}
  \ ,\qquad
  \begin{tikzpicture}[anchorbase]
    \draw[->] (-0.2,-0.5) to (-0.2,0.5);
    \draw[->] (0.2,-0.5) to (0.2,0.5);
    \singdot{-0.2,0};
    \teleport{-0.2,-0.25}{0.2,0};
  \end{tikzpicture}
  \ =\
  \begin{tikzpicture}[anchorbase]
    \draw[->] (-0.2,-0.5) to (-0.2,0.5);
    \draw[->] (0.2,-0.5) to (0.2,0.5);
    \singdot{-0.2,0};
    \teleport{-0.2,0.25}{0.2,0};
  \end{tikzpicture}
  \ ,\qquad
  \begin{tikzpicture}[anchorbase]
    \draw[->] (-0.4,-0.5) to (-0.4,0.5);
    \draw[->] (0,-0.5) to (0,0.5);
    \draw[->] (0.4,-0.5) to (0.4,0.5);
    \teleport{-0.4,-0.25}{0,-0.25};
    \teleport{0,0}{0.4,0};
  \end{tikzpicture}
  \ =\
  \begin{tikzpicture}[anchorbase]
    \draw[->] (-0.4,-0.5) to (-0.4,0.5);
    \draw[->] (0,-0.5) to (0,0.5);
    \draw[->] (0.4,-0.5) to (0.4,0.5);
    \teleport{-0.4,0.25}{0.4,0.25};
    \teleport{0,0}{0.4,0};
  \end{tikzpicture}
    \ .
\end{equation}

\section{The Frobenius Heisenberg category\label{FHC}}

We now recall the definition of the Frobenius Heisenberg category.  We follow the definition of \cite[Def.~5.1]{BSW20}, which is a slight modification of the original definition in \cite{Sav19}.  When $k=-1$, this category was introduced earlier in \cite{RS17}.

\begin{defin}
  The \emph{Frobenius Heisenberg category} $\Heis_k(A)$ of central charge $k \in \Z$ is the strict graded monoidal category generated by objects $\uparrow$ and $\downarrow$ and morphisms
  \begin{gather*}
    \begin{tikzpicture}[anchorbase]
      \draw[->] (0,0) -- (0,0.6);
      \singdot{0,0.3};
    \end{tikzpicture}
    \colon \uparrow \to \uparrow
    \ ,\quad
    \begin{tikzpicture}[anchorbase]
      \draw [->](0,0) -- (0.6,0.6);
      \draw [->](0.6,0) -- (0,0.6);
    \end{tikzpicture}
    \colon \uparrow \otimes \uparrow\ \to \uparrow \otimes \uparrow
    \ ,\quad
    \begin{tikzpicture}[anchorbase]
      \draw[->] (0,0) -- (0,0.6);
      \token{west}{0,0.3}{a};
    \end{tikzpicture}
    \colon \uparrow \to \uparrow
    \ ,\ a \in A,
    \\
    \begin{tikzpicture}[anchorbase]
      \draw[->] (0,.2) -- (0,0) arc (180:360:.3) -- (.6,.2);
    \end{tikzpicture}
    \ \colon \one \to\ \downarrow\otimes \uparrow
    , \quad
    \begin{tikzpicture}[anchorbase]
      \draw[->] (0,-.2) -- (0,0) arc (180:0:.3) -- (.6,-.2);
    \end{tikzpicture}
    \ \colon \uparrow\otimes \downarrow\ \to \one
    , \quad
    \begin{tikzpicture}[anchorbase]
      \draw[<-] (0,.2) -- (0,0) arc (180:360:.3) -- (.6,.2);
    \end{tikzpicture}
    \ \colon \one \to\ \uparrow\otimes \downarrow
    , \quad
    \begin{tikzpicture}[anchorbase]
      \draw[<-] (0,0) -- (0,.2) arc (180:0:.3) -- (.6,0);
    \end{tikzpicture}
    \ \colon \downarrow\otimes \uparrow\ \to \one,
  \end{gather*}
  subject to certain relations.  The degree and parities of the dot, crossing, and token are as in \cref{updegs}, while the cups/caps are even of degrees
  \begin{equation} \label{trappist}
    \deg \left(\:
    \begin{tikzpicture}[anchorbase]
      \draw[->] (0,.2) -- (0,0) arc (180:360:.3) -- (.6,.2);
    \end{tikzpicture}
    \: \right) = k d,
    \quad
    \deg \left( \:
    \begin{tikzpicture}[anchorbase]
      \draw[->] (0,-.2) -- (0,0) arc (180:0:.3) -- (.6,-.2);
    \end{tikzpicture}
    \:\right) = - k d,
    \quad
    \deg \left(
    \:\begin{tikzpicture}[anchorbase]
      \draw[<-] (0,.2) -- (0,0) arc (180:360:.3) -- (.6,.2);
    \end{tikzpicture}
    \:\right) = - k d,
    \quad
    \deg \left(
    \:\begin{tikzpicture}[anchorbase]
      \draw[<-] (0,0) -- (0,.2) arc (180:0:.3) -- (.6,0);
    \end{tikzpicture}
    \:\right) = k d.
  \end{equation}
  The defining relations are \cref{tokrel,braidrel,affrel}, plus the following additional relations:
  \begin{gather} \label{squid}
    \begin{tikzpicture}[anchorbase]
      \draw[->] (-0.3,-0.5) to (-0.3,0) to[out=up,in=up,looseness=2] (0,0) to[out=down,in=down,looseness=2] (0.3,0) to (0.3,0.5);
    \end{tikzpicture}
    \ =\
    \begin{tikzpicture}[anchorbase]
      \draw[->] (0,-0.5) to (0,0.5);
    \end{tikzpicture}
    \ ,\qquad
    \begin{tikzpicture}[anchorbase]
      \draw[->] (-0.3,0.5) to (-0.3,0) to[out=down,in=down,looseness=2] (0,0) to[out=up,in=up,looseness=2] (0.3,0) to (0.3,-0.5);
    \end{tikzpicture}
    \ =\
    \begin{tikzpicture}[anchorbase]
      \draw[<-] (0,-0.5) to (0,0.5);
    \end{tikzpicture}
    \ ,
  \\ \label{chew}
    \cbubble{a}{r}
    \ = -\delta_{r,k-1} \tr(a)
    \quad \text{if } 0 \le r < k,
    \qquad
    \ccbubble{a}{r}
    \ = \delta_{r,-k-1} \tr(a)
    \quad \text{if } 0 \le r < -k,
    \\ \label{curly}
    \begin{tikzpicture}[anchorbase]
      \draw[->] (0,-0.5) to[out=up,in=0] (-0.3,0.2) to[out=180,in=up] (-0.45,0) to[out=down,in=180] (-0.3,-0.2) to[out=0,in=down] (0,0.5);
    \end{tikzpicture}
    = \delta_{k,0}\
    \begin{tikzpicture}[anchorbase]
      \draw[->] (0,-0.5) -- (0,0.5);
    \end{tikzpicture}
    \quad \text{if } k \le 0,
    \qquad
    \begin{tikzpicture}[anchorbase]
      \draw[->] (0,-0.5) to[out=up,in=180] (0.3,0.2) to[out=0,in=up] (0.45,0) to[out=down,in=0] (0.3,-0.2) to[out=180,in=down] (0,0.5);
    \end{tikzpicture}
    = \delta_{k,0}\
    \begin{tikzpicture}[anchorbase]
      \draw[->] (0,-0.5) -- (0,0.5);
    \end{tikzpicture}
    \quad \text{if } k \ge 0,
    \\ \label{squish}
    \begin{tikzpicture}[anchorbase]
      \draw[<-] (0,0) \braidto (0.5,0.6) \braidto (0,1.2);
      \draw[->] (0.5,0) \braidto (0,0.6) \braidto (0.5,1.2);
    \end{tikzpicture}
    \ =\
    \begin{tikzpicture}[anchorbase]
      \draw[<-] (0,0) -- (0,1.2);
      \draw[->] (0.3,0) -- (0.3,1.2);
    \end{tikzpicture}
    \ + \sum_{r,s \ge 0}
    \begin{tikzpicture}[anchorbase]
      \draw[->] (-0.2,0.6) to (-0.2,0.35) arc(180:360:0.2) to (0.2,0.6);
      \draw[<-] (-0.2,-0.6) to (-0.2,-0.35) arc(180:0:0.2) to (0.2,-0.6);
      \draw[->] (0.8,0) arc(360:0:0.2);
      \multdot{east}{0.4,0}{-r-s-2};
      \multdot{west}{0.2,0.42}{r};
      \multdot{west}{0.2,-0.42}{s};
      \teleport{0.18,0.25}{0.459,0.141};
      \teleport{0.18,-0.25}{0.459,-0.141};
    \end{tikzpicture}
    \ ,\qquad
    \begin{tikzpicture}[anchorbase]
      \draw[->] (0,0) \braidto (0.5,0.6) \braidto (0,1.2);
      \draw[<-] (0.5,0) \braidto (0,0.6) \braidto (0.5,1.2);
    \end{tikzpicture}
    \ =\
    \begin{tikzpicture}[anchorbase]
      \draw[->] (0,0) -- (0,1.2);
      \draw[<-] (0.3,0) -- (0.3,1.2);
    \end{tikzpicture}
    \ + \sum_{r,s \ge 0}
    \begin{tikzpicture}[anchorbase]
      \draw[<-] (-0.2,0.6) to (-0.2,0.35) arc(180:360:0.2) to (0.2,0.6);
      \draw[->] (-0.2,-0.6) to (-0.2,-0.35) arc(180:0:0.2) to (0.2,-0.6);
      \draw[->] (-0.8,0) arc(-180:180:0.2);
      \multdot{west}{-0.4,0}{-r-s-2};
      \multdot{east}{-0.2,0.42}{r};
      \multdot{east}{-0.2,-0.42}{s};
      \teleport{-0.18,0.25}{-0.459,0.141};
      \teleport{-0.18,-0.25}{-0.459,-0.141};
    \end{tikzpicture}
    \ .
  \end{gather}
  Note that the relations \cref{squish} involve some diagrammatic shorthands which have not yet been defined.  In particular, we are using the left and right crossings defined by
  \begin{equation} \label{fall}
    \begin{tikzpicture}[anchorbase]
      \draw[->] (0,0) -- (0.6,0.6);
      \draw[<-] (0.6,0) -- (0,0.6);
    \end{tikzpicture}
    \ :=\
    \begin{tikzpicture}[anchorbase]
      \draw[->] (0.2,-0.3) to (-0.2,0.3);
      \draw[<-] (0.6,-0.3) to[out=up,in=45,looseness=2] (0,0) to[out=225,in=down,looseness=2] (-0.6,0.3);
    \end{tikzpicture}
    \ ,\qquad
    \begin{tikzpicture}[anchorbase]
      \draw[<-] (0,0) -- (0.6,0.6);
      \draw[->] (0.6,0) -- (0,0.6);
    \end{tikzpicture}
    \ :=\
    \begin{tikzpicture}[anchorbase]
      \draw[->] (-0.2,-0.3) to (0.2,0.3);
      \draw[<-] (-0.6,-0.3) to[out=up,in=135,looseness=2] (0,0) to[out=-45,in=down,looseness=2] (0.6,0.3);
    \end{tikzpicture}
    \ .
  \end{equation}
  In addition, the sums in \cref{squish} also involve negatively-dotted bubbles.  To interpret these, one first uses the second relation in \cref{affrel} to collect the tokens which arise on expanding the definitions of the teleporters into a single token.  Then the negatively-dotted bubbles labelled by $a \in A$ are defined by
  \begin{align} \label{fakel}
    \begin{tikzpicture}[baseline={(0,-0.15)}]
      \bubleft{0,0}{a}{r-k};
    \end{tikzpicture}
      &:= \sum_{b_1,\dotsc,b_{r} \in \bB_A} \det
      \left(
    \begin{tikzpicture}[baseline={(0,-0.15)}]
      \draw[->] (0,0.3) arc(90:-270:0.3);
      \token{east}{-0.25,0.17}{b^\vee_{j-1}b_j};
      \multdot{east}{-0.25,-0.17}{i-j+k};
    \end{tikzpicture}
    \right)_{i,j=1}^{r+1},\quad\:\:\qquad
    r < k,
    \\ \label{faker}
    \begin{tikzpicture}[baseline={(0,-0.15)}]
      \bubright{0,0}{a}{r+k};
    \end{tikzpicture}
    &:= (-1)^{r} \sum_{b_1,\dotsc,b_{r} \in \bB_A} \det
    \left(
    \begin{tikzpicture}[baseline={(0,-0.15)}]
      \draw[->] (0,0.3) arc(90:450:0.3);
      \token{west}{0.25,-0.17}{b^\vee_{j-1}b_j};
      \multdot{west}{0.25,0.17}{i-j-k};
    \end{tikzpicture}
    \right)_{i,j=1}^{r+1},\quad
    r < -k,
  \end{align}
  with the convention that $b_0^\vee := a$ and $b_{r+1} := 1$.  The determinants here mean the usual Laplace expansions keeping the non-commuting variables in each monomial ordered in the same way as the columns from which they are taken (see \cite[(17)]{Sav19}), and we interpret them as $\tr(a)$ if $r+1 = 0$ or as $0$ if $r+1 < 0$. Note in particular that the sums in \cref{squish} are actually finite according to these definitions.  This concludes the definition of $\Heis_k(A)$.
\end{defin}

Up to isomorphism, $\Heis_k(A)$ depends only on the underlying graded superalgebra $A$, and not on its trace map; see \cite[Lem.~5.3]{BSW20}.  As explained in the proof of \cite[Th.~1.2]{Sav19}, the defining relations of $\Heis_k(A)$ imply that the following, which is a $(1+k\dim A)\times 1$ matrix or a $1 \times (1-k \dim A)$ matrix, respectively, is an isomorphism in the additive envelope of $\Heis_k(A)$:
\begin{equation} \label{invrel}
  \begin{aligned}
    \left[
      \begin{tikzpicture}[anchorbase]
        \draw [->](0,0) -- (0.6,0.6);
        \draw [<-](0.6,0) -- (0,0.6);
      \end{tikzpicture}
      \quad
      \begin{tikzpicture}[anchorbase]
        \draw[->] (0,0) -- (0,0.5) arc (180:0:.3) -- (0.6,0);
        \multdot{west}{0,0.4}{r};
        \token{west}{0,0.15}{b^\vee};
      \end{tikzpicture}\ ,\
      0 \le r \le k-1,\ b \in \bB_A
    \right]^T
    \ &\colon \uparrow \otimes\downarrow \to \downarrow \otimes\uparrow \oplus \one ^{\oplus k \dim A} & \text{if } k \ge 0,
    \\
    \left[
      \begin{tikzpicture}[anchorbase]
        \draw [->](0,0) -- (0.6,0.6);
        \draw [<-](0.6,0) -- (0,0.6);
      \end{tikzpicture}
      \quad
      \begin{tikzpicture}[anchorbase]
        \draw[->] (0,0.8) -- (0,0.3) arc (180:360:.3) -- (0.6,0.8);
        \multdot{east}{0.6,0.5}{r};
        \token{east}{0.6,0.25}{b^\vee};
      \end{tikzpicture}\ ,\
      0 \le r \le -k-1,\ b \in \bB_A
    \right]
    \ &\colon \uparrow \otimes\downarrow \oplus \one^{\oplus (-k \dim A)} \to \downarrow \otimes\uparrow
    & \text{if } k < 0.
  \end{aligned}
\end{equation}

\begin{rem}
    \begin{enumerate}
        \item The category $\Heis_{-1}(\kk)$ is the original Heisenberg category introduced in \cite{Kho11}.  The definition was extended to arbitrary central charge $k$ in \cite{MS18,Bru18}.
        \item The category $\Heis_0(\kk)$ is the affine oriented Brauer category of \cite{BCNR17}.  The categories $\Heis_0(A)$ can be viewed as Frobenius algebra deformations of the affine oriented Brauer category.
        \item When $A$ is a skew-zigzag algebra, the category $\Heis_{-1}(A)$ was studied in \cite[\S 6.1]{CL12}.
    \end{enumerate}
\end{rem}

We will need some other relations, which are proved in \cite[Th.~1.3]{Sav19}.  The relation
\cref{squid} means that $\downarrow$ is right dual to $\uparrow$ (in the appropriate graded sense).  In fact, we also have
\begin{equation}
    \begin{tikzpicture}[anchorbase]
        \draw[<-] (-0.3,-0.5) to (-0.3,0) to[out=up,in=up,looseness=2] (0,0) to[out=down,in=down,looseness=2] (0.3,0) to (0.3,0.5);
    \end{tikzpicture}
    \ =\
    \begin{tikzpicture}[anchorbase]
        \draw[<-] (0,-0.5) to (0,0.5);
    \end{tikzpicture}
    \ ,\qquad
    \begin{tikzpicture}[anchorbase]
        \draw[<-] (-0.3,0.5) to (-0.3,0) to[out=down,in=down,looseness=2] (0,0) to[out=up,in=up,looseness=2] (0.3,0) to (0.3,-0.5);
    \end{tikzpicture}
    \ =\
    \begin{tikzpicture}[anchorbase]
        \draw[->] (0,-0.5) to (0,0.5);
    \end{tikzpicture}
    \ ,
\end{equation}
and so $\downarrow$ is also left dual to $\uparrow$.  Thus $\Heis_k(A)$ is \emph{rigid}.  In fact, $\Heis_k(A)$ admits a strictly pivotal structure.  We define downward tokens, dots, and crossings by
\begin{gather*}
  \begin{tikzpicture}[anchorbase]
    \draw[<-] (0,-0.3) -- (0,0.3);
    \token{west}{0,0}{a};
  \end{tikzpicture}
  :=\
  \begin{tikzpicture}[anchorbase]
    \draw[->] (-0.4,0.5) -- (-0.4,-0.1) arc(180:360:0.2) -- (0,0.1) arc(180:0:0.2) -- (0.4,-0.5);
    \token{west}{0,0}{a};
  \end{tikzpicture}
  \ =\
  \begin{tikzpicture}[anchorbase]
    \draw[->] (0.4,0.5) -- (0.4,-0.1) arc(360:180:0.2) -- (0,0.1) arc(0:180:0.2) -- (-0.4,-0.5);
    \token{west}{0,0}{a};
  \end{tikzpicture}
  \ ,\qquad
  \begin{tikzpicture}[anchorbase]
      \draw[<-] (0,-0.3) -- (0,0.3);
      \singdot{0,0};
  \end{tikzpicture}
  :=
  \begin{tikzpicture}[anchorbase]
    \draw[->] (-0.4,0.5) -- (-0.4,-0.1) arc(180:360:0.2) -- (0,0.1) arc(180:0:0.2) -- (0.4,-0.5);
    \singdot{0,0};
  \end{tikzpicture}
  \ =\
  \begin{tikzpicture}[anchorbase]
    \draw[->] (0.4,0.5) -- (0.4,-0.1) arc(360:180:0.2) -- (0,0.1) arc(0:180:0.2) -- (-0.4,-0.5);
    \singdot{0,0};
  \end{tikzpicture}
  \ ,
  \\
  \begin{tikzpicture}[anchorbase]
    \draw[<-] (0,0) -- (.5,.6);
    \draw[<-] (.5,0) -- (0,.6);
  \end{tikzpicture}
  :=
  \begin{tikzpicture}[anchorbase,scale=0.5]
    \draw[->] (-1.5,1.5) .. controls (-1.5,0.5) and (-1,-1) .. (0,0) .. controls (1,1) and (1.5,-0.5) .. (1.5,-1.5);
    \draw[->] (-2,1.5) .. controls (-2,-2) and (1.5,-1.5) .. (0,0) .. controls (-1.5,1.5) and (2,2) .. (2,-1.5);
  \end{tikzpicture}
  \ =\
  \begin{tikzpicture}[anchorbase]
    \draw[<-] (0.2,-0.3) to (-0.2,0.3);
    \draw[<-] (0.6,-0.3) to[out=up,in=45,looseness=2] (0,0) to[out=225,in=down,looseness=2] (-0.6,0.3);
  \end{tikzpicture}
  \ =\
  \begin{tikzpicture}[anchorbase]
    \draw[<-] (-0.2,-0.3) to (0.2,0.3);
    \draw[<-] (-0.6,-0.3) to[out=up,in=135,looseness=2] (0,0) to[out=-45,in=down,looseness=2] (0.6,0.3);
  \end{tikzpicture}
  \ =\
  \begin{tikzpicture}[anchorbase,scale=0.5]
    \draw[->] (1.5,1.5) .. controls (1.5,0.5) and (1,-1) .. (0,0) .. controls (-1,1) and (-1.5,-0.5) .. (-1.5,-1.5);
    \draw[->] (2,1.5) .. controls (2,-2) and (-1.5,-1.5) .. (0,0) .. controls (1.5,1.5) and (-2,2) .. (-2,-1.5);
  \end{tikzpicture}
  \ .
\end{gather*}
It follows that dots, tokens, and crossings slide over cups and caps.  We will sometimes draw dots and tokens at the critical point of cups and caps since this causes no ambiguity.

There is a natural symmetry between the categories $\Heis_k(A)$ and $\Heis_{-k}(A)$ that we now describe. For an object $X$ in $\Heis_k(A)$, let $\Omega_k(X)$ be the object of $\Heis_{-k}(A)$ obtained by replacing all $\uparrow$'s with $\downarrow$'s and $\downarrow$'s with $\uparrow$'s.  Now let $f \colon X \rightarrow Y$ be a morphism in $\Heis_k(A)$ represented by some diagram.  Let $x$ be the total number of crossings and let $y$ by the total number of \emph{odd} tokens in the diagram.  Then let $\Omega_k(f) \colon \Omega_k(Y) \rightarrow \Omega_k(X)$ be the morphism in $\Heis_{-k}(A)$ defined by the reflection of the diagram $f$ in a horizontal axis multiplied by $(-1)^{x+\binom{y}{2}}$.  Then there is an isomorphism of strict
graded monoidal categories
\begin{equation} \label{Omega}
  \Omega_k \colon \Heis_k(A) \xrightarrow{\cong} \Heis_{-k}(A)^\op,
  \qquad
  X \mapsto \Omega_k(X), \quad
  f \mapsto \Omega_k(f)^\op.
\end{equation}

We also have an isomorphism
\begin{equation} \label{Phi}
  \Phi_k \colon \Heis_k(A) \xrightarrow{\cong} \left( \Heis_k(A)^\op \right)^\rev = \left( \Heis_k(A)^\rev \right)^\op
\end{equation}
given by rotating diagrams through $180\degree$, then multiplying by $(-1)^{\binom{y}{2}}$, where $y$ is the number of odd tokens.

The bubbles (both genuine and fake) satisfy the \emph{infinite grassmannian relations}: for $a,b \in A$, we have
\begin{gather} \label{weed}
  \cbubble{a}{r} = -\delta_{r,k-1} \tr(a)
  \quad \text{if } r \le k-1,
  \quad
  \ccbubble{a}{r} = \delta_{r,-k-1} \tr(a)
  \quad \text{if } r \le -k-1,
  \\ \label{infgrass}
  \sum_{r \in \Z}
  \begin{tikzpicture}[baseline={(0,-0.1)}]
    \draw[->] (0,0.5) arc(90:450:0.2);
    \draw[->] (0,-0.5) arc(270:-90:0.2);
    \teleport{0,0.1}{0,-0.1};
    \token{east}{-0.2,0.3}{a};
    \token{east}{-0.2,-0.3}{b};
    \multdot{west}{0.2,0.3}{r};
    \multdot{west}{0.2,-0.3}{n-r-2};
  \end{tikzpicture}
  \ = - \delta_{n,0} \tr(ab) 1_\one.
\end{gather}
We also have the \emph{curl relations}
\begin{equation} \label{curls}
  \begin{tikzpicture}[anchorbase]
    \draw[->] (0,-0.5) to[out=up,in=0] (-0.3,0.2) to[out=180,in=up] (-0.45,0) to[out=down,in=180] (-0.3,-0.2) to[out=0,in=down] (0,0.5);
    \multdot{east}{-0.45,0}{r};
  \end{tikzpicture}
  = \sum_{s \ge 0}
  \begin{tikzpicture}[anchorbase]
    \draw[->] (0,-0.5) -- (0,0.5);
    \draw[->] (-0.5,0.2) arc(90:450:0.2);
    \multdot{east}{-0.7,0}{r-s-1};
    \multdot{west}{0,0.25}{s};
    \teleport{-0.3,0}{0,0};
  \end{tikzpicture}
  \ ,\qquad
  \begin{tikzpicture}[anchorbase]
    \draw[->] (0,-0.5) to[out=up,in=180] (0.3,0.2) to[out=0,in=up] (0.45,0) to[out=down,in=0] (0.3,-0.2) to[out=180,in=down] (0,0.5);
    \multdot{west}{0.45,0}{r};
  \end{tikzpicture}
  = - \sum_{s \ge 0}
  \begin{tikzpicture}[anchorbase]
    \draw[->] (0,-0.5) -- (0,0.5);
    \draw[->] (0.5,0.2) arc(90:-270:0.2);
    \multdot{west}{0.7,0}{r-s-1};
    \multdot{east}{0,0.25}{s};
    \teleport{0.3,0}{0,0};
  \end{tikzpicture}
  \ ,
\end{equation}
the \emph{bubble slide relations}
\begin{equation} \label{bubslide}
  \begin{tikzpicture}[anchorbase]
    \draw[->] (0,-0.5) to (0,0.5);
    \bubleft{0.6,0}{a}{r};
  \end{tikzpicture}
  =
  \begin{tikzpicture}[anchorbase]
    \draw[->] (0,-0.5) to (0,0.5);
    \bubleft{-0.6,0}{a}{r};
  \end{tikzpicture}
  \ - \sum_{s,t \ge 0}
  \begin{tikzpicture}[anchorbase]
    \draw[->] (0,-0.6) to (0,0.5);
    \draw[->] (-0.3,0) arc(0:360:0.2);
    \teleport{-0.5,0.2}{0,0.2};
    \teleport{-0.5,-0.2}{0,-0.2};
    \token{west}{0,0}{a};
    \multdot{west}{0,-0.4}{s+t};
    \multdot{east}{-0.7,0}{r-s-t-2};
  \end{tikzpicture}
  \ ,
  \begin{tikzpicture}[anchorbase]
    \draw[->] (0,-0.5) to (0,0.5);
    \bubright{-0.6,0}{a}{r};
  \end{tikzpicture}
  =
  \begin{tikzpicture}[anchorbase]
    \draw[->] (0,-0.5) to (0,0.5);
    \bubright{0.6,0}{a}{r};
  \end{tikzpicture}
  \ - \sum_{s,t \ge 0}
  \begin{tikzpicture}[anchorbase]
    \draw[->] (0,-0.6) to (0,0.5);
    \draw[->] (0.3,0) arc(180:-180:0.2);
    \teleport{0.5,0.2}{0,0.2};
    \teleport{0.5,-0.2}{0,-0.2};
    \token{east}{0,0}{a};
    \multdot{east}{0,-0.4}{s+t};
    \multdot{west}{0.7,0}{r-s-t-2};
  \end{tikzpicture}
  \ ,
\end{equation}
and the \emph{alternating braid relation}
\begin{equation} \label{altbraid}
  \begin{tikzpicture}[anchorbase]
    \draw[->] (-0.4,-0.5) to (0.4,0.5);
    \draw[->] (0.4,-0.5) to (-0.4,0.5);
    \draw[<-] (0,-0.5) \braidto (-0.4,0) \braidto (0,0.5);
  \end{tikzpicture}
  -
  \begin{tikzpicture}[anchorbase]
    \draw[->] (-0.4,-0.5) to (0.4,0.5);
    \draw[->] (0.4,-0.5) to (-0.4,0.5);
    \draw[<-] (0,-0.5) \braidto (0.4,0) \braidto (0,0.5);
  \end{tikzpicture}
  = \sum_{r,s,t \ge 0}
  \begin{tikzpicture}[anchorbase]
    \draw[<-] (-0.2,0.8) to (-0.2,0.6) arc(180:360:0.2) to (0.2,0.8);
    \draw[->] (-0.2,-0.8) to (-0.2,-0.6) arc(180:0:0.2) to (0.2,-0.8);
    \draw[->-=0.17] (0,-0.2) arc(-90:270:0.2);
    \draw[->] (0.6,-0.8) to (0.6,0.8);
    \multdot{east}{-0.2,0}{-r-s-t-3};
    \multdot{east}{-0.2,0.6}{r};
    \multdot{east}{-0.2,-0.6}{s};
    \multdot{west}{0.6,0.5}{t};
    \teleport{0,0.4}{0,0.2};
    \teleport{0,-0.4}{0,-0.2};
    \teleport{0.2,0}{0.6,0};
  \end{tikzpicture}
  + \sum_{r,s,t \ge 0}
  \begin{tikzpicture}[anchorbase]
    \draw[<-] (0.2,0.8) to (0.2,0.6) arc(360:180:0.2) to (-0.2,0.8);
    \draw[->] (0.2,-0.8) to (0.2,-0.6) arc(0:180:0.2) to (-0.2,-0.8);
    \draw[->-=0.17] (0,-0.2) arc(-90:-450:0.2);
    \draw[->] (-0.6,-0.8) to (-0.6,0.8);
    \multdot{west}{0.2,0}{-r-s-t-3};
    \multdot{west}{0.2,0.6}{r};
    \multdot{west}{0.2,-0.6}{s};
    \multdot{east}{-0.6,0.5}{t};
    \teleport{0,0.2}{0,0.4};
    \teleport{0,-0.2}{0,-0.4};
    \teleport{-0.2,0}{-0.6,0};
  \end{tikzpicture}
  \ .
\end{equation}

The center of $\Heis_k(A)$, consisting of closed diagrams, is a Frobenius analogue of the ring of symmetric functions, as we now explain.  Let $\Sym(A)$ denote the graded symmetric superalgebra generated by the graded super vector space $C(A)[x]$, where $x$ here is an even indeterminate of degree $2d$.  For $n \in \Z$ and $a \in A$, let $h_n(a) \in \Sym(A)$ denote
\begin{equation}
    h_n(a) :=
    \begin{cases}
        0 & \text{if $n < 0$}, \\
        \tr(a) & \text{if $n=0$}, \\
        \cocenter{a} x^{n-1} & \text{if $n > 0$}.
    \end{cases}
\end{equation}
This defines an even homogeneous linear map $h_n \colon A \rightarrow \Sym(A)$ of degree $2d(n-1)$.  Since $h_n(a)$ depends only on the image of $a \in C(A)$, we can define $h_n(\cocenter{a}) = h_n(a)$.  If $\bB_{C(A)}$ is a homogeneous basis of $C(A)$, then $\Sym(A)$ is the graded supercommutative superalgebra generated freely by the elements $h_r(a)$, $r \ge 1$, $a \in \bB_{C(A)}$, with $\deg(h_r(a)) = 2d(r-1) + \deg(a)$ and $\overline{h_r(a)} = \bar{a}$.

\begin{lem}[{\cite[Lem.~7.1]{BSW20}}] \label{controversy}
    For each $n \in \Z$, there is a unique even homogeneous linear map $e_n \colon A \rightarrow \Sym(A)$ of degree $2(n-1)d$ such that
    \[
        \sum_{c \in \bB_A} e(ac;-u) h(c^\vee b; u) = \tr(ab),
        \qquad
        \text{for all }a,b \in A,
    \]
    where we are using the generating functions
    \[
        e(a;u) := \sum_{n \geq 0} e_n(a) u^{-n}, \quad
        h(a;u) := \sum_{n \geq 0} h_n(a) u^{-n}
        \in \Sym(A) \llbracket u^{-1} \rrbracket.
    \]
\end{lem}

For $r \in \Z$, we define a map
\begin{equation} \label{powerdef}
    p_r \colon A \to \Sym(A),\quad
    p_r(a) :=
    \begin{cases}
        0 & \text{if } r \le 0, \\
        \sum_{s=1}^r \sum_{b \in \bB_A} (-1)^{s-1} s h_{r-s}(b) e_s(b^\vee a) & \text{if } r > 0.
    \end{cases}
\end{equation}

\begin{cor} \label{coconut}
    We have an isomorphism of graded commutative algebras
    \begin{equation} \label{beta}
        \begin{gathered}
            \beta_k \colon \Sym(A) \to \End_{\Heis_k(A)}(\one), \\
            e_r(a) \mapsto (-1)^{r-1} \ccbubble{a}{r-k-1},\
            h_r(a) \mapsto \cbubble{a}{r+k-1},\
            p_r(a) \mapsto \sum_{\substack{s,t \ge 0 \\ s+t=r}} s
            \begin{tikzpicture}[centerbase]
                \draw[->] (0.4,0.2) arc(90:450:0.2);
                \draw[->] (-0.4,0.2) arc(90:-270:0.2);
                \token{west}{0.6,0}{a};
                \teleport{-0.2,0}{0.2,0};
                \multdot{north}{0.4,-0.2}{s-k-1};
                \multdot{north}{-0.4,-0.2}{t+k-1};
            \end{tikzpicture}
            \ ,\ r \in \Z,\ a \in A.
        \end{gathered}
    \end{equation}
\end{cor}

\begin{proof}
    It follows from \cref{controversy,infgrass} that we have a well-defined homomorphism of graded algebras with image of $e_r(a)$ and $h_r(a)$ as in \cref{beta}.  The fact that $\beta_k$ is an isomorphism follows from the basis theorem \cite[Th.~7.2]{BSW20} for $\Heis_k(A)$.  The expression for the image of $p_r(a)$ follows from \cref{powerdef}.
\end{proof}

\begin{eg}
    We can identify $\Sym(\kk)$ with the ring of symmetric functions with coefficients in $\kk$ by identifying $e_n(1)$ with the $n$-th elementary symmetric function.  Then $h_n(1)$ is the $n$-th homogeneous symmetric function (this follows from \cite[(2.6)]{Mac95}) and $p_n(1)$ is the $n$-th power sum symmetric function (see \cite[p.~33]{Mac95}).
\end{eg}

Recalling the symmetries \cref{Omega,Phi}, we see that
\begin{gather} \label{Omsym}
    \beta_{-k}^{-1} \circ \Omega_k \circ \beta_k \colon
    e_r(a) \mapsto (-1)^{r-1} h_r(a),\quad
    h_r(a) \mapsto (-1)^{r-1} e_r(a),\quad
    p_r(a) \mapsto -p_r(a),
    \\ \label{Phsym}
    \beta_k^{-1} \circ \Phi_k \circ \beta_k = \id_{\Sym(A)}.
\end{gather}

It follows from the basis theorem \cite[Th.~7.2]{BSW20} for $\Heis_k(A)$ that the natural functor
\begin{equation} \label{imath}
    \imath \colon \AWC \to \Heis_k(A),
\end{equation}
taking $\uparrow$ to $\uparrow$ and mapping each string diagram in $\AWC$ to the same diagram in $\Heis_k(A)$, is a faithful functor of $\kk$-linear graded monoidal supercategories.  Thus, for $n \ge 1$, we have, from \cref{imathaff}, a homomorphism of graded superalgebras
\begin{equation} \label{imathn}
    \imath_n \colon \AWA_n(A) \to \End_{\Heis_k(A)}(\uparrow^{\otimes n}),
\end{equation}
which we also denote by $\imath_n$.

Similarly, viewing $A^\op$ as a graded Frobenius superalgebra with trace map $\tr$ being the same underlying linear map as for $A$, we have a faithful functor of $\kk$-linear graded monoidal supercategories
\begin{equation} \label{jmath}
  \jmath \colon \AWC[A^\op] \to \Heis_k(A),
\end{equation}
taking $\uparrow$ to $\downarrow$ and mapping each string diagram in $\AWC$ to the same diagram in $\Heis_k(A)$, multiplied by $(-1)^x$, where $x$ is the number of crossings in the diagram. Equivalently, for each $n \ge 1$ we have a homomorphism of graded algebras
\begin{equation} \label{jmathn}
  \jmath_n \colon \AWA_n(A^\op) \to \End_{\Heis_k(A)}(\downarrow^{\otimes n})
\end{equation}
sending $- s_i$ to the crossing of the $i$-th and $(i+1)$-st strings, $x_j$ to a dot on the $j$-th string, and $a^{(j)}$ to a token labelled $a$ on the $j$-th string.

\section{Trace of the Frobenius Heisenberg category}

In this section we prove our main result: a linear isomorphism between the trace of the Frobenius Heisenberg category and a central reduction of the Frobenius W-algebra.

For $n \ge 1$, $r \in \N$, and $a \in A$, define
\begin{align} \label{opal+}
    \hcL_{n,r}(a) &:= \imath_n \left( a^{(1)} \left(\frac{x_1 + x_2 + \dotsb + x_n}{n} \right)^r \rho_n \right), & \cL_{n,r} &:= \cocenter{\hcL_{n,r}} \in \Tr(\Heis_k(A)),
    \\ \label{opal-}
    \hcL_{-n,r}(a) &:= (-1)^{n-1} \jmath_n \left( a^{(1)} \left(\frac{x_1 + x_2 + \dotsb + x_n}{n} \right)^r \rho_n \right), & \cL_{-n,r} &:= \cocenter{\hcL_{-n,r}} \in \Tr(\Heis_k(A)),
\end{align}
where $\rho_n$ is the cycle $(1\, 2\, \cdots\, n) \in S_n$ and $a^{(1)}$ is defined as in \cref{aidef}.  (In fact, it follows from \cref{tokrel,affrel} that we could replace $a^{(1)}$ by $a^{(i)}$ for any $1 \le i \le n$.)  For $r \in \N$, $a \in C(A)$, we recursively define $\hcL_{0,r}(a)$ by
\begin{equation} \label{studded}
    \hcL_{0,r}(a) := \frac{-1}{r+1} \beta_k(p_{r+1}(a)) - \frac{1}{r+1} \sum_{i=1}^{\left\lfloor \frac{r}{2} \right\rfloor} \binom{r+1}{2i+1} \hcL_{0,r-2i} \left( \frac{a \kappa^i}{4^i} \right),
\end{equation}
and set $\cL_{0,r}(a) = \cocenter{\hcL_{0,r}(a)}$.  In particular, we have
\begin{align} \label{L00}
    \hcL_{0,0}(a)
    &= \ccbubble{a}{-k} = \cbubble{a}{k},
    \\ \label{L01}
    \hcL_{0,1}(a)
    &= \ccbubble{a}{1-k} - \tfrac{1}{2}\
    \begin{tikzpicture}[centerbase]
        \draw[->] (-0.4,0.2) arc(90:450:0.2);
        \draw[->] (0.4,0.2) arc(90:450:0.2);
        \teleport{-0.2,0}{0.2,0};
        \token{west}{0.6,0}{a};
        \multdot{north}{-0.4,-0.2}{-k};
        \multdot{north}{0.4,-0.2}{-k};
    \end{tikzpicture},
    \\ \label{L02}
    \hcL_{0,2}(a)
    &= \ccbubble{a}{2-k} -
    \begin{tikzpicture}[centerbase]
        \draw[->] (-0.4,0.2) arc(90:450:0.2);
        \draw[->] (0.4,0.2) arc(90:450:0.2);
        \teleport{-0.2,0}{0.2,0};
        \token{west}{0.6,0}{a};
        \multdot{north}{-0.4,-0.2}{-k};
        \multdot{north}{0.4,-0.2}{1-k};
    \end{tikzpicture}
    + \tfrac{1}{3}\
    \begin{tikzpicture}[centerbase]
        \draw[->] (0,0.2) arc(90:450:0.2);
        \draw[->] (-0.6,0.2) arc(90:450:0.2);
        \draw[->] (0.6,0.2) arc(90:450:0.2);
        \teleport{-0.4,0}{-0.2,0};
        \teleport{0.4,0}{0.2,0};
        \token{west}{0.8,0}{a};
        \multdot{north}{-0.6,-0.2}{-k};
        \multdot{north}{0,-0.2}{-k};
        \multdot{north}{0.6,-0.2}{-k};
    \end{tikzpicture}
    - \tfrac{1}{12} \ccbubble{a \kappa}{-k}.
\end{align}
Note that, for $n \in \Z$, $r \in \N$, and $a \in A$, the element $\hcL_{n,r}(a)$ depends only on $n$, $r$, and the image $\cocenter{a}$ of $a$ in $C(A)$.  So we can define $\hcL_{n,r}(\cocenter{a}) := \cL_{n,r}(a)$ and $\cL_{n,r}(\cocenter{a}) = \cL_{n,r}(a)$.

The superalgebra $\Tr(\Heis_k(A))$ inherits the degree grading from $\Heis_k(A)$.  In also has a \emph{rank} grading given by the total number of up strands minus the total number of down strands, and an \emph{order} filtration given by number of dots.  In particular, $\rank \cL_{m,r}(a) = m$ and $\cL_{m,r}(a)$ lives in order $\le r$.  (Compare to \cref{rankorder}.)

Recall the order $\gtrdot$ on $\Z \times \N \times \bB_{C(A)}$ from \cref{order}.  The main result of the current section is the following.

\begin{theo} \label{linisom}
    As a $\kk$-module, the algebra $\Tr(\Heis_k(A))$ has basis
    \begin{multline*}
        \{\cL_{n_1,r_1}(a_1) \cL_{n_2,r_2}(a_2) \dotsm \cL_{n_t,r_t}(a_t) :
        t \in \N,\ (n_i,r_i,a_i) \in \Z \times \N \times \bB_{C(A)},\\
        (n_1,r_1,a_1) \gtrdot (n_2,r_2,a_2) \gtrdot \dotsb \gtrdot (n_t,r_t,a_t)\}.
    \end{multline*}
    Therefore, the map
    \begin{equation} \label{zeta}
        \begin{aligned}
            \zeta_k \colon W(A)/(C-k) &\to \Tr(\Heis_k(A)),\\
            L_{n_1,r_1}(a_1) L_{n_2,r_2}(a_2) \dotsm L_{n_t,r_t}(a_t)
            &\mapsto \cL_{n_1,r_1}(a_1) \cL_{n_2,r_2}(a_2) \dotsm \cL_{n_t,r_t}(a_t),
        \end{aligned}
    \end{equation}
    $t \in \N$, $(n_1,r_1,a_1) \gtrdot (n_2,r_2,a_2) \gtrdot \dotsb \gtrdot (n_t,r_t,a_t)$,
    is an isomorphism of $\kk$-modules.  This map preserves the degree grading, rank grading, and order filtration.
\end{theo}

The proof of \cref{linisom} will occupy the remainder of this section.  Our first step is to give a triangular decomposition of $\Tr(\Heis_k(A))$.

Since the maps $\imath_n$ and $\jmath_n$ defined in \cref{imath,jmath} are homomorphisms of graded algebras, they induce homomorphisms of graded $\kk$-modules
\begin{align*}
    \cocenter{\imath_n} &\colon C \left( \AWA_n(A) \right) \to \Tr(\Heis_k(A)),\quad
    \cocenter{\imath_n}(\cocenter{y}) = \cocenter{\imath_n(y)},\ y \in \AWA_n(A), \\
    \cocenter{\jmath_n} &\colon C \left( \AWA_n(A^\op) \right) \to \Tr(\Heis_k(A)),\quad
    \cocenter{\jmath_n}(\cocenter{y}) = \cocenter{\jmath_n(y)},\ y \in \AWA_n(A^\op).
\end{align*}
Define
\begin{align*}
    \cocenter{\imath}
    &:= \left( \cocenter{\imath_n} \right)_{n \in \N} \colon \bigoplus_{n \in \N} C \left( \AWA_n(A) \right) \to \Tr(\Heis_k(A)),
    \\
    \cocenter{\jmath}
    &:= \left( \cocenter{\jmath_n} \right)_{n \in \N} \colon \bigoplus_{n \in \N} C \left( \AWA_n(A^\op) \right) \to \Tr(\Heis_k(A)).
\end{align*}
In addition, since $\Sym(A)$ is supercommutative, hence canonically isomorphic to its cocenter, we have a homomorphism of graded $\kk$-modules
\[
    \cocenter{\beta_k} \colon \Sym(A) \to \Tr(\Heis_k(A)),\quad
    \cocenter{\beta_k}(f) = \cocenter{\beta_k(f)},\quad f \in \Sym(A).
\]

\begin{prop} \label{yurt}
    We have isomorphisms of graded $\kk$-modules
    \begin{align*}
        \cocenter{\jmath} \otimes \cocenter{\imath} \otimes \cocenter{\beta_k} \colon \bigoplus_{m,n \in \N} C \left( \AWA_m(A^\op) \right) \otimes C \left( \AWA_n(A) \right) \otimes \Sym(A) &\xrightarrow{\cong} \Tr(\Heis_k(A)),\quad \text{if } k \ge 0, \\
        \cocenter{\imath} \otimes \cocenter{\jmath} \otimes \cocenter{\beta_k} \colon \bigoplus_{m,n \in \N} C \left( \AWA_m(A) \right) \otimes C \left( \AWA_n(A^\op) \right) \otimes \Sym(A) &\xrightarrow{\cong} \Tr(\Heis_k(A)), \quad \text{if } k \le 0.
    \end{align*}
\end{prop}

\begin{proof}
    Suppose $k \ge 0$, the case $k \le 0$ being analogous.  It follows from \cref{invrel} that every object of $\Heis_k(A)$ is isomorphic to a direct sum of objects of the form $\downarrow^{\otimes m} \otimes \uparrow^{\otimes n}$, $m,n \in \N$.  Let $\cC$ be the full subcategory of $\Heis_k(A)$ whose objects are $\downarrow^{\otimes m} \otimes \uparrow^{\otimes n}$, $m,n \in \N$.  Then, by \cite[Lem.~2.1]{BHLW17}, we have $\Tr(\Heis_k(A)) \cong \Tr(\cC)$.

    Now, for $m \in \Z$, let $\cC^{(m)}$ be the full subcategory of $\Heis_k(A)$ whose objects are $\downarrow^{\otimes (m+n)} \otimes \uparrow^{\otimes n}$ for $n \ge \max(0,-m)$.  Then there are no morphisms between objects in $\cC^{(m)}$ and $\cC^{(n)}$ for $m \ne n$.  In other words, $\cC = \bigoplus_{m \in \Z} \cC^{(m)}$.  Hence $\Tr(\cC) = \bigoplus_{m \in \Z} \Tr(\cC^{(m)})$.

    As in \cref{traces}, we can view $\cC^{(m)}$ as a locally unital graded superalgebra $R$ with system of idempotents $\{1_n:= 1_{X_n} : n \ge \max(0,-m)\}$, where $X_n := \downarrow^{\otimes (m+n)} \otimes \uparrow^{\otimes n}$.  For each $n_1,n_2 \ge \max(0,-m)$, an \emph{$(X_{n_1},X_{n_2})$-matching} is a bijection between the sets
    \[
        \{\uparrow \text{ in } X_{n_1} \} \sqcup \{\downarrow \text{ in } X_{n_2}\}
        \quad \text{and} \quad
        \{\downarrow \text{ in } X_{n_1} \} \sqcup \{\uparrow \text{ in } X_{n_2}\}.
    \]
    A \emph{reduced lift} of an $(X_{n_1},X_{n_2})$-matching is a string diagram representing a morphism $X_{n_1} \to X_{n_2}$ such that
    \begin{itemize}
        \item the endpoints of each string are points which correspond under the given matching;
        \item there are no floating bubbles and no dots or tokens on any string;
        \item there are no self-intersections of strings and no two strings cross each other more than once.
    \end{itemize}
    Fix a set $D(n_2,n_1)$ consisting of a choice of reduced lift for each $(X_{n_1},X_{n_2})$-matching.  Then choose a basis $\bB_\one$ of $\End_{\Heis_k(A)}(\one)$. Now, for $n,n_1,n_2 \ge \max(0,-m)$, define the following:
    \begin{itemize}
        \item If $n_1 > n_2$, let $D_{n_2,n_1}$ denote the set of all morphisms that can be obtained from the elements of $D(n_2,n_1)$ \emph{containing no cups} by adding to each string \emph{involved in a cap} a nonnegative number of dots near the terminus of the string and one element of $\bB_A$ to the string.
        \item If $n_1 < n_2$, let $D_{n_2,n_1}$ denote the set of all morphisms that can be obtained from the elements of $D(n_2,n_1)$ \emph{containing no caps} by adding to each string \emph{involved in a cup} a nonnegative number of dots near the terminus of the string and one element of $\bB_A$ to the string.
        \item Let $D_{n,n} = \{1_{X_n}\}$.
        \item Let $D_n$ denote the set of all morphisms that can be obtained from the elements of $D(n,n)$ \emph{containing no cups or caps} by adding to each string a nonnegative number of dots near the terminus of the string and one element of $\bB_A$ to the string, and then placing an element of $\bB_\one$ to the right of all the strings.
    \end{itemize}
    It follows from the basis theorem \cite[Th.~7.2]{BSW20} for $\Heis_k(A)$ that the sets $D_{n_2,n_1}$, $D_n$ satisfy the conditions \eqref{B1} and \eqref{B2} from \cref{coextro}, where
    \[
        R_n = \jmath \left( \AWA_{m+n}(A^\op) \right) \otimes \imath \left( \AWA_n(A) \right) \otimes \End_{\Heis_k(A)}(\one).
    \]
    In fact, in \cite[Th.~7.2]{BSW20}, the basis elements of $\Heis_k(A)$ have dots near the termini of the strands connecting the top and bottom of the diagram, whereas here we place the dots on such strands in the middle of the diagram, above all the caps and below all the cups.  But the proof of  \cite[Th.~7.2]{BSW20} goes through with this modification, or we can use the fact that dots slide through crossings modulo diagrams with fewer dots.  Now the proposition follows from \cref{snowshoe}.
\end{proof}

In light of \cref{yurt}, in order to compute $\Tr(\Heis_k(A))$ as a graded $\kk$-supermodule, it suffices to compute $C(\AWA_n(A))$.  By \cref{beatdown}, we may instead compute $\Tr(\AWC)$.

For $n \ge 1$, $r \in \N$, and $a \in A$, define, via the map \cref{imathaff},
\[
    \hcL'_{n,r}(a) := \imath_n \left( a^{(1)} x_1^r \rho_n \right)
    =
    \begin{tikzpicture}[anchorbase]
        \draw[->] (-1,-0.3) \braidto (-0.1,0.3) to (-0.1,0.7);
        \draw[->] (-0.8,-0.3) \braidto (-1,0.3) to (-1,0.7);
        \draw[->] (-0.3,-0.3) \braidto (-0.5,0.3) to (-0.5,0.7);
        \draw[->] (-0.1,-0.3) \braidto (-0.3,0.3) to (-0.3,0.7);
        \node at (-0.71,0.4) {$\cdots$};
        \multdot{west}{-0.1,0.3}{r};
        \token{west}{-0.1,0.5}{a};
    \end{tikzpicture}
    \in \End_{\AWC}(\uparrow^{\otimes n}),\quad
    \cL'_{n,r}(a) := \cocenter{\hcL'_{n,r}(a)} \in \Tr(\AWC).
\]
As for $\cL_{n,r}(a)$, the element $\cL'_{n,r}(a)$ depends only on the image $\cocenter{a}$ of $a$ in $C(A)$.  So we may define $\cL'_{n,r}(\cocenter{a}) := \cL'_{n,r}(a)$.

\begin{prop} \label{bridge1}
    The trace $\Tr(\AWC)$ is spanned over $\kk$ by the elements
    \begin{multline} \label{honey1}
        \cL'_{n_1,r_1}(a_1) \cL'_{n_2,r_2}(a_2) \dotsb \cL'_{n_t,r_t}(a_t),\quad t \in \N,\ (n_i,r_i,a_i) \in \Z_{>0} \times \N \times \bB_{C(A)}, \\
        (n_1,r_1,a_1) \gtrdot (n_2,r_2,a_2) \gtrdot \dotsb \gtrdot (n_t,r_t,a_t).
    \end{multline}
\end{prop}

\begin{proof}
    We will suppress the map $\imath_n$ and identify $\AWA_n(A)$ with $\End_{\AWC}(\uparrow^{\otimes n})$.  Since \linebreak $\Hom_{\AWC}(\uparrow^{\otimes m}, \uparrow^{\otimes n}) = 0$ for $m \ne n$, it suffices to fix $n \in \N$, and show that every element of $C(\AWA_n(A))$ is a linear combination of elements of the form \cref{honey1}.

    We proceed by induction on the filtration \cref{tart}.  First consider $\AWA_n(A)_0 = A^{\otimes n} \rtimes S_n$.  Let $\ba \in A^{\otimes n}$ and $\pi \in S_n$.  Conjugating by the appropriate permutation, we have that $\cocenter{\ba \pi} = \cocenter{\ba_1 \rho_{n_1}} \dotsm \cocenter{\ba_t \rho_{n_t}}$ for some $\ba_1, \dotsc, \ba_t \in A^{\otimes n}$ and $n_1 \ge \dotsb n_t$.  Then, using the fourth relation in \cref{tokrel}, together with the trace relation, we can move all tokens to the rightmost string in each cycle.  This completes the argument for $\AWA_n(A)_0$.

    Now suppose that $l \ge 1$ and that $C \left( \AWA_n(A)_{\le l-1} \right)$ is spanned by elements of the form \cref{honey1}.  Suppose $f \in P_n(A)$ has polynomial degree $r$ and $\pi \in S_n$.  Again, conjugating by an appropriate permutation, we can assume $\pi$ is a product of cycles of weakly decreasing length, modulo terms in $\AWA_n(A)_{l-1}$.  (Here we use the fact that dots slide through crossings, modulo terms with fewer dots.)  Then we move all tokens to the rightmost strand in each cycle.  Finally, we slide all dots to the rightmost strand, again using the fact that dots slide through crossing modulo terms with fewer dots.

    Finally, if $a$ is odd, then
    \[
        \hcL'_{n,r}(a) \hcL'_{n,r}(a) = \frac{1}{2} [\hcL'_{n,r}(a), \hcL'_{n,r}(a)]
        \in \left[ \AWA_n(A), \AWA_n(A) \right]. \qedhere
    \]
\end{proof}

\begin{cor} \label{bridge2}
    The trace $\Tr(\AWC)$ is spanned over $\kk$ by the elements
    \begin{multline} \label{honey2}
        \cL_{n_1,r_1}(a_1) \cL_{n_2,r_2}(a_2) \dotsb \cL_{n_t,r_t}(a_t),\quad t \in \N,\ (n_i,r_i,a_i) \in \Z_{>0} \times \N \times \bB_{C(A)}, \\
        (n_1,r_1,a_1) \gtrdot (n_2,r_2,a_2) \gtrdot \dotsb \gtrdot (n_t,r_t,a_t).
    \end{multline}
\end{cor}

\begin{proof}
    As in the proof of \cref{bridge1}, it suffices to show that every element of $C(\AWA_n(A))$ is a linear combination of elements of the form \cref{honey2}.  Choose an order on the set
    \[
        \left\{ \big( (n_1,r_1,a_1), \dotsc, (n_t,r_t,a_t) \big) : t \in \N,\ (n_i,r_i,a_i) \in \Z_{>0} \times \N \times \bB_{C(A)},\ n_1 + \dotsb + n_t = n \right\}
    \]
    compatible with the partial order given by the total number $\sum_i r_i$ of dots.  This order induces an order on the corresponding elements \cref{honey1,honey2}.  Since dots on cycles can be slid to the rightmost strand modulo diagrams with fewer dots, we see that the transition matrix from the elements \cref{honey1} to \cref{honey2} is unitriangular.  Hence the corollary follows from \cref{bridge1}.
\end{proof}

For $n \ge 1$, define the nested right cups and left caps
\[
    \teacup_n :=
    \begin{tikzpicture}[anchorbase]
        \draw[<-] (-0.2,0.1) -- (-0.2,0) arc(180:360:0.2) -- (0.2,0.1);
        \draw[<-] (-0.8,0.1) -- (-0.8,0) arc(180:360:0.8) -- (0.8,0.1);
        \node at (0,-0.4) {$\vdots$};
    \end{tikzpicture}
    \colon \one \to \uparrow^{\otimes n} \otimes \downarrow^{\otimes n},
    \quad
    \teacap_n :=
    \begin{tikzpicture}[anchorbase]
        \draw[->] (-0.2,-0.1) -- (-0.2,0) arc(180:0:0.2) -- (0.2,-0.1);
        \draw[->] (-0.8,-0.1) -- (-0.8,0) arc(180:0:0.8) -- (0.8,-0.1);
        \node at (0,0.6) {$\vdots$};
    \end{tikzpicture}
    \colon \uparrow^{\otimes n} \otimes \downarrow^{\otimes n} \to \one.
\]
We define $\teacup_0 = \teacap_0 = 1_\one$.  Then, for $n \in \N$, consider the $\kk$-linear map
\begin{equation} \label{maple}
    \AWA_n(A) \to \End_{\Heis_k(A)}(\one),\quad
    z \mapsto \teacap_n \circ (\imath(z) \otimes 1_{\downarrow^{\otimes n}}) \circ \teacup_n.
\end{equation}
Diagrammatically, this map closes a diagram off to the right.  This map factors through the cocenter $C(\AWA_n(A))$, and so we have the induced map
\[
    \theta_k \colon C \left( \AWA_n(A) \right) \to \Sym(A),\quad
    \langle z \rangle \mapsto \beta_k^{-1} \left( \teacap_n \circ (\imath(z) \otimes 1_{\downarrow^{\otimes n}}) \circ \teacup_n \right).
\]
of degree $-2kdn$.

Recall that $\Sym(A)$ is a supercommutative superalgebra freely generated by $h_r(a)$, $r \ge 1$, $a \in \bB_{C(A)}$, where the generator $h_r(a)$ has degree $2d(r-1) + \deg(a)$ and parity $\bar{a}$.  Thus $\Sym(A)$ has a basis given by the monomials
\begin{equation} \label{lego}
    h_{r_1}(a_1) h_{r_2}(a_2) \dotsm h_{r_t}(a_t),\quad
    t \in \N,\ r_1 \ge \dotsb \ge r_t,\ a_1,\dotsc,a_t \in \bB_{C(A)},\ a_i \ne a_{i+1} \text{ when } r_i=r_{i+1}.
\end{equation}

Consider another grading on $\Sym(A)$, which we call \emph{bubble number}, given by $\bub(h_r(a)) = r-1$.  Thus, amongst monomials the elements \cref{lego} of a given degree, the elements with the smallest bubble number are those with the smallest value of $\sum_{i=1}^t \deg(a_i)$.  We now fix an order $\le$ on the set of monomials \cref{lego} compatible with the partial order given by degree and, amongst monomials of the same degree, compatible with the partial order given by bubble number.  When we speak of \emph{higher order terms} below, we mean with respect to the order $\le$.

\begin{prop} \label{treble}
    Assume $k \le 0$, and suppose, $t \in \N$, $a_1,\dotsc,a_t \in \bB_{C(A)}$, $n_1,\dotsc,n_t \ge 1$ and $f_i \in \kk[x_1,\dotsc,x_{n_i}]$ is a monomial of polynomial degree $r_i$ for $1 \le i \le t$.  Then
    \[
        \theta_k \left( (f_1 a_1^{(1)} \rho_{n_1}) \otimes \dotsb \otimes (f_t a_t^{(1)} \rho_{n_t}) \right)
        = \pm h_{r_1-n_1k+1}(a_1) h_{r_2-n_2k+1}(a_2) \dotsm h_{r_t-n_tk+1}(a_t) + \text{higher order terms}.
    \]
\end{prop}

\begin{proof}
    We prove the result by induction on $n := n_1 + \dotsb + n_t$.  The case $n=0$ (i.e.\ $t=0$) is trivial, so suppose $n \ge 1$.  Let $\theta_k' = \beta_k \circ \theta_k$ and set $z = (f_1 a_1^{(1)} \rho_{n_1}) \otimes \dotsb \otimes (f_t a_t^{(1)} \rho_{n_t})$.  Suppose $n_t \ge 2$ and define $r \in \N$ and $f_t' \in \kk[x_1,\dotsc,x_{n_t-1}]$ such that $f_t = f_t' \otimes a_t^{(1)} x_1^r$.  Then, using square brackets to denote the closure of a diagram to the right as in \cref{maple}, we have
    \[
        \theta_k'(z)
        =
        \left[\
            \begin{tikzpicture}[anchorbase]
                \draw (0,-0.3) \braidto (0.8,0.3);
                \draw (0.2,-0.3) \braidto (0,0.3);
                \draw (0.4,-0.3) \braidto (0.2,0.3);
                \draw (0.6,-0.3) \braidto (0.4,0.3);
                \draw (0.8,-0.3) \braidto (0.6,0.3);
                \draw (-0.1,0.3) rectangle (0.9,0.8) node[midway] {$\scriptstyle{f_1 a_1}$};
                \draw[->] (0,0.8) -- (0,1);
                \draw[->] (0.2,0.8) -- (0.2,1);
                \draw[->] (0.4,0.8) -- (0.4,1);
                \draw[->] (0.6,0.8) -- (0.6,1);
                \draw[->] (0.8,0.8) -- (0.8,1);
            \end{tikzpicture}
            \cdots
            \begin{tikzpicture}[anchorbase]
                \draw (0,-0.3) \braidto (0.8,0.3);
                \draw (0.2,-0.3) \braidto (0,0.3);
                \draw (0.4,-0.3) \braidto (0.2,0.3);
                \draw (0.6,-0.3) \braidto (0.4,0.3);
                \draw (0.8,-0.3) \braidto (0.6,0.3);
                \draw (-0.1,0.3) rectangle (0.9,0.8) node[midway] {$\scriptstyle{f_t a_t}$};
                \draw[->] (0,0.8) -- (0,1);
                \draw[->] (0.2,0.8) -- (0.2,1);
                \draw[->] (0.4,0.8) -- (0.4,1);
                \draw[->] (0.6,0.8) -- (0.6,1);
                \draw[->] (0.8,0.8) -- (0.8,1);
            \end{tikzpicture}
        \ \right]
        =
        \left[\
            \begin{tikzpicture}[anchorbase]
                \draw (0,-0.3) \braidto (0.8,0.3);
                \draw (0.2,-0.3) \braidto (0,0.3);
                \draw (0.4,-0.3) \braidto (0.2,0.3);
                \draw (0.6,-0.3) \braidto (0.4,0.3);
                \draw (0.8,-0.3) \braidto (0.6,0.3);
                \draw (-0.1,0.3) rectangle (0.9,0.8) node[midway] {$\scriptstyle{f_1 a_1}$};
                \draw[->] (0,0.8) -- (0,1);
                \draw[->] (0.2,0.8) -- (0.2,1);
                \draw[->] (0.4,0.8) -- (0.4,1);
                \draw[->] (0.6,0.8) -- (0.6,1);
                \draw[->] (0.8,0.8) -- (0.8,1);
            \end{tikzpicture}
            \cdots
            \begin{tikzpicture}[anchorbase]
                \draw (0,-0.6) \braidto (0.9,0.3) -- (0.9,0.5) arc (180:0:0.2) -- (1.3,-0.3) arc (360:180:0.2) \braidto (0.6,0.3);
                \draw (0.2,-0.6) \braidto (0,0.3);
                \draw (0.4,-0.6) \braidto (0.2,0.3);
                \draw (0.6,-0.6) \braidto (0.4,0.3);
                \draw (-0.1,0.3) rectangle (0.7,0.8) node[midway] {$\scriptstyle{f_t'}$};
                \draw[->] (0,0.8) -- (0,1);
                \draw[->] (0.2,0.8) -- (0.2,1);
                \draw[->] (0.4,0.8) -- (0.4,1);
                \draw[->] (0.6,0.8) -- (0.6,1);
                \multdot{west}{0.9,0.3}{r};
                \token{west}{0.9,0.5}{a_t};
            \end{tikzpicture}
        \ \right]
    \]
    where we have suppressed the map $\imath_n$ to simplify notation and, within each rectangle, $x_i$ corresponds to a dot on the $i$-th strand (numbering from right to left) entering the rectangle,and $a_i$ corresponds to a token labelled $a_i$ placed on the rightmost strand.  Using the curl relation \cref{curls}, and then the bubble slide relations \cref{bubslide} to slide the resulting bubbles out to the right, we see that
    \[
        \theta_k'(z)
        =
        \left[\
            \begin{tikzpicture}[anchorbase]
                \draw (0,-0.3) \braidto (0.8,0.3);
                \draw (0.2,-0.3) \braidto (0,0.3);
                \draw (0.4,-0.3) \braidto (0.2,0.3);
                \draw (0.6,-0.3) \braidto (0.4,0.3);
                \draw (0.8,-0.3) \braidto (0.6,0.3);
                \draw (-0.1,0.3) rectangle (0.9,0.8) node[midway] {$\scriptstyle{f_1 a_1}$};
                \draw[->] (0,0.8) -- (0,1);
                \draw[->] (0.2,0.8) -- (0.2,1);
                \draw[->] (0.4,0.8) -- (0.4,1);
                \draw[->] (0.6,0.8) -- (0.6,1);
                \draw[->] (0.8,0.8) -- (0.8,1);
            \end{tikzpicture}
            \cdots
            \begin{tikzpicture}[anchorbase]
                \draw (0,-0.6) \braidto (0.6,0.3);
                \draw (0.2,-0.6) \braidto (0,0.3);
                \draw (0.4,-0.6) \braidto (0.2,0.3);
                \draw (0.6,-0.6) \braidto (0.4,0.3);
                \draw (-0.2,0.3) rectangle (0.8,0.8) node[midway] {$\scriptstyle{f_t' x_1^{r-k} a_t}$};
                \draw[->] (0,0.8) -- (0,1);
                \draw[->] (0.2,0.8) -- (0.2,1);
                \draw[->] (0.4,0.8) -- (0.4,1);
                \draw[->] (0.6,0.8) -- (0.6,1);
            \end{tikzpicture}
        \ \right]
        + \textup{higher order terms}.
    \]
    Repeating this process, we have
    \begin{equation} \label{curb}
        \theta_k'(z)
        =
        \left[\
            \begin{tikzpicture}[anchorbase]
                \draw (0,-0.3) \braidto (0.8,0.3);
                \draw (0.2,-0.3) \braidto (0,0.3);
                \draw (0.4,-0.3) \braidto (0.2,0.3);
                \draw (0.6,-0.3) \braidto (0.4,0.3);
                \draw (0.8,-0.3) \braidto (0.6,0.3);
                \draw (-0.1,0.3) rectangle (0.9,0.8) node[midway] {$\scriptstyle{f_1 a_1}$};
                \draw[->] (0,0.8) -- (0,1);
                \draw[->] (0.2,0.8) -- (0.2,1);
                \draw[->] (0.4,0.8) -- (0.4,1);
                \draw[->] (0.6,0.8) -- (0.6,1);
                \draw[->] (0.8,0.8) -- (0.8,1);
            \end{tikzpicture}
            \cdots
             \begin{tikzpicture}[anchorbase]
                \draw (0,-0.3) \braidto (0.8,0.3);
                \draw (0.2,-0.3) \braidto (0,0.3);
                \draw (0.4,-0.3) \braidto (0.2,0.3);
                \draw (0.6,-0.3) \braidto (0.4,0.3);
                \draw (0.8,-0.3) \braidto (0.6,0.3);
                \draw (-0.1,0.3) rectangle (0.9,0.8) node[midway] {$\scriptstyle{f_{t-1} a_{t-1}}$};
                \draw[->] (0,0.8) -- (0,1);
                \draw[->] (0.2,0.8) -- (0.2,1);
                \draw[->] (0.4,0.8) -- (0.4,1);
                \draw[->] (0.6,0.8) -- (0.6,1);
                \draw[->] (0.8,0.8) -- (0.8,1);
                \draw[->] (1.4,0.3) arc(90:-270:0.2);
                \multdot{west}{1.6,0.1}{r_t+(1-n_t)k};
                \token{east}{1.2,0.1}{a_t};
            \end{tikzpicture}
        \ \right]
        + \textup{higher order terms}.
    \end{equation}
    In addition, if $n_t=1$, we immediately have \cref{curb}.  We then use the bubble slide relations \cref{bubslide} and the induction hypothesis to see that
    \[
        \theta_k'(z)
        =
        \left[\
            \begin{tikzpicture}[anchorbase]
                \draw (0,-0.3) \braidto (0.8,0.3);
                \draw (0.2,-0.3) \braidto (0,0.3);
                \draw (0.4,-0.3) \braidto (0.2,0.3);
                \draw (0.6,-0.3) \braidto (0.4,0.3);
                \draw (0.8,-0.3) \braidto (0.6,0.3);
                \draw (-0.1,0.3) rectangle (0.9,0.8) node[midway] {$\scriptstyle{f_1 a_1}$};
                \draw[->] (0,0.8) -- (0,1);
                \draw[->] (0.2,0.8) -- (0.2,1);
                \draw[->] (0.4,0.8) -- (0.4,1);
                \draw[->] (0.6,0.8) -- (0.6,1);
                \draw[->] (0.8,0.8) -- (0.8,1);
            \end{tikzpicture}
            \cdots
             \begin{tikzpicture}[anchorbase]
                \draw (0,-0.3) \braidto (0.8,0.3);
                \draw (0.2,-0.3) \braidto (0,0.3);
                \draw (0.4,-0.3) \braidto (0.2,0.3);
                \draw (0.6,-0.3) \braidto (0.4,0.3);
                \draw (0.8,-0.3) \braidto (0.6,0.3);
                \draw (-0.1,0.3) rectangle (0.9,0.8) node[midway] {$\scriptstyle{f_{t-1} a_{t-1}}$};
                \draw[->] (0,0.8) -- (0,1);
                \draw[->] (0.2,0.8) -- (0.2,1);
                \draw[->] (0.4,0.8) -- (0.4,1);
                \draw[->] (0.6,0.8) -- (0.6,1);
                \draw[->] (0.8,0.8) -- (0.8,1);
            \end{tikzpicture}
        \ \right]
        \otimes
        \begin{tikzpicture}[anchorbase]
            \draw[->] (1.4,0.3) arc(90:-270:0.2);
            \multdot{west}{1.6,0.1}{r_t+(1-n_t)k};
            \token{east}{1.2,0.1}{a_t};
        \end{tikzpicture}
        + \textup{higher order terms}.
    \]
    Since
    \[
        \beta_k^{-1}
        \left(
            \begin{tikzpicture}[anchorbase]
                \draw[->] (1.4,0.3) arc(90:-270:0.2);
                \multdot{west}{1.6,0.1}{r_t+(1-n_t)k};
                \token{east}{1.2,0.1}{a_t};
            \end{tikzpicture}
        \right)
        = h_{r_t-n_tk+1}(a_t),
    \]
    the result now follows by induction.
\end{proof}

\begin{cor} \label{canopy}
    The elements \cref{honey2} are a basis for $\Tr(\AWC)$ as a $\kk$-module.
\end{cor}

\begin{proof}
    In light of \cref{bridge2}, it remains to show that the elements \cref{honey2} are linearly independent.  Consider a finite collection of these elements.  Choose $k \in \Z$ such that $-k$ is larger than all the $r_i$ appearing in this finite collection.  Then \cref{treble} implies that this set of elements is linearly independent.
\end{proof}

\begin{lem} \label{rake}
    For $r \in \N$ and $a \in C(A)$, we have that
    \[
        \hcL_{0,r}(a) - \ccbubble{a}{r-k}
    \]
    is contained in the subalgebra of $\End_{\Heis_k(A)}(\one)$ generated by $\ccbubble{b}{s-k}$, $0 \le s < r$, $b \in C(A)$.
\end{lem}

\begin{proof}
    Let $R_r$ be the subalgebra of $\End_{\Heis_k(A)}(\one)$ generated by $\ccbubble{b}{s-k}$, $0 \le s < r$, $b \in A$.  First note that
    \[
        \beta_k(p_{r+1}(a))
        \overset{\cref{weed}}{=}
        -(r+1) \ccbubble{a}{r-k} + \sum_{\substack{s,t \ge 1 \\ s+t=r+1}} s
        \begin{tikzpicture}[centerbase]
            \draw[->] (0.4,0.2) arc(90:450:0.2);
            \draw[->] (-0.4,0.2) arc(90:-270:0.2);
            \token{west}{0.6,0}{a};
            \teleport{-0.2,0}{0.2,0};
            \multdot{north}{0.4,-0.2}{s-k-1};
            \multdot{north}{-0.4,-0.2}{t+k-1};
        \end{tikzpicture}
        \overset{\cref{infgrass}}{\in} -(r+1) \ccbubble{a}{r-k} + R_r.
    \]
    The lemma then follows from \cref{studded} by induction on $r$.
\end{proof}

\begin{cor} \label{jungle}
    The elements
    \[
        \hcL_{0,r_1}(a_1) \hcL_{0,r_2}(a_2) \dotsm \hcL_{0,r_t}(a_t),\quad
        t \in \N,\ (r_i,a_i) \in \N \times \bB_{C(A)},\ (0,r_1,a_1) \gtrdot \dotsb \gtrdot (0,r_t,a_t),
    \]
    form a basis for $\End_{\Heis_k(A)}(\one)$.
\end{cor}

\begin{proof}
    This follows from \cref{rake,coconut}, together with the fact that $\Sym(A)$ is a polynomial algebra in the $e_r(a)$, $r \in \N$, $a \in \bB_{C(A)}$.
\end{proof}

\begin{proof}[Proof of \cref{linisom}]
    \Cref{linisom} now follows from \cref{yurt,canopy,jungle}.
\end{proof}

\begin{conj} \label{hope}
  The map \cref{zeta} is an isomorphism of graded superalgebras.
\end{conj}
We discuss some evidence for \cref{hope} in \cref{sec:conj}.

We can now see that the symmetries \cref{Omega,Phi} of the Frobenius Heisenberg category correspond to the involutions \cref{omega-alg,phi-alg}.  The maps \cref{omega-alg,phi-alg} induce isomorphisms of graded algebras
\[
  \omega_k \colon W(A)/(C-k) \xrightarrow{\cong} W(A)/(C+k)
  \quad \text{and} \quad
  \varphi_k \colon W(A)/(C-k) \xrightarrow{\cong} \left( W(A)/(C-k) \right)^\op.
\]

\begin{prop}
  We have
  \[
    \zeta_{-k} \circ \omega_k = \cocenter{\Omega_k} \circ \zeta_k
    \quad \text{and} \quad
    \zeta_k \circ \varphi_k = \cocenter{\Phi_k} \circ \zeta_k.
  \]
\end{prop}

\begin{proof}
  It follows from \cref{opal+,opal-} that $\cocenter{\Omega_k(\cL_{n,r}(a))} = (-1)^{n-1} \cL_{-n,r}(a)$ and $\cocenter{\Phi_k(\cL_{n,r}(a))} = \cL_{-n,r}(a)$ for $n \ne 0$.  This matches \cref{noodles}.  The case $n=0$ follows from \cref{Omsym,Phsym,studded}.
\end{proof}

\begin{rem} \label{Solleveld}
  When $A = \kk$, $\AWA_n(\kk)$ is the degenerate affine Hecke algebra, whose cocenter was first computed by Solleveld in \cite{Sol10}.  Solleveld's result was a key ingredient in the arguments of \cite{CLLS18}.  The methods of the current paper give a completely different approach to computing this cocenter, and make the arguments here independent of \cite{Sol10}.  This new approach utilizes the action of the trace of $\Heis_k(A)$ on its center, together with the basis theorem \cite[Th.~7.2]{BSW20} for $\Heis_k(A)$.  Note that to compute the cocenter of $\AWA_n(A)$, whose definition is independent of $k$, we perform computations in $\Heis_k(A)$ for large values of $|k|$.  This illustrates the value of a thorough knowledge of the Frobenius Heisenberg category at \emph{arbitrary} central charge.
\end{rem}

\section{Comments on the conjecture\label{sec:conj}}

In this section we discuss \cref{hope}, including giving some evidence for its veracity.

\subsection{Special cases\label{special}}

When $A = \kk$ and $k=\pm 1$, \cref{hope} essentially reduces to \cite[Th.~1]{CLLS18}.  The authors of \cite{CLLS18} compute the trace of Khovanov's original Heisenberg category, which is the quotient of $\Heis_{-1}(\kk)$ by the additional relation
\begin{equation} \label{pop}
    \begin{tikzpicture}[centerbase]
        \draw[->] (0,0.2) arc(90:450:0.2);
        \singdot{0.2,0};
    \end{tikzpicture}
    = 0.
\end{equation}
By \cref{L00}, the bubble on the right-hand side of \cref{pop} is $\hcL_{0,0}(1)$.  For this reason, the isomorphism in \cite[Th.~1]{CLLS18} is between the trace of Khovanov's Heisenberg category and $\cW_{1+\infty}/(C-1,w_{0,0})$ (see the discussion above \cref{rocks} for notation); the quotienting by $w_{0,0}$ corresponding to the relation \cref{pop}.  There are also some sign differences between the treatment in \cite{CLLS18} and that of the current paper.  These arise from the fact that, in \cite{CLLS18}, they consider the central charge $k=-1$ Heisenberg category, but compare to the $C=1$ quotient of $\cW_{1+\infty}$.

The obstacle to employing the methods of \cite{CLLS18} to compute the trace of the general Frobenius Heisenberg category $\Heis_k(A)$ is that the proof of \cite[Th.~1]{CLLS18} does not involve directly showing that the map from the trace of the Heisenberg category to $\cW_{1+\infty}/(C-1,w_{0,0})$ is an algebra isomorphism.  Rather, it uses the fact that $\cW_{1+\infty}$ has a rather small set of generators, and computes the action of these generators (and their counterparts in the trace of the Heisenberg category) on a faithful Fock space representation.  When one moves to the setting of general central charge $k$ or to a more general Frobenius algebra $A$, one needs a much larger set to generate the corresponding algebra $\rW(A)$.  The computation of the required relations involving these generators is much more complicated.

Nevertheless, when $k = \pm 1$ and the operation $\diamond$ has a unit element (e.g.\ when $A$ is the group algebra of a finite group as in \cref{groupalg}), one can use the methods of \cite{CLLS18} to prove \cref{hope}.  More precisely, one can show that $\rW(A)$ is generated by $L_{-1,0}(a)$, $L_{1,0}(a)$, $L_{0,2}(a)$, $a \in C(A)$.  Then the diagrammatic computations in \cref{compute} are sufficient to mimic the arguments of \cite{CLLS18}.

\subsection{The single cycle phenomenon}

We can view the Lie algebra $\fW(A)$ naturally as the Lie subalgebra of $W(A)$ (with the usual commutator as Lie bracket) spanned by the $L_{m,r}(a)$, $m \in \Z$, $r \in \N$, $a \in C(A)$, and the central element $C$.  Under the isomorphism $\zeta_k$ of \cref{linisom}, $\fW(A)$ then corresponds to the span of the $\cL_{m,r}(a)$.  In other words, the image of $\fW(A)$ under the map $\zeta_k$ is spanned by the classes of \emph{single cycles} (the $\cL_{m,r}(a)$, $m \ne 0$) and the $\beta_k(p_r(a))$, $r \in \N_+$ (see \cref{studded}).  Thus, \cref{hope} would imply that, when one computes the commutator $[\cL_{n,r}(a), \cL_{m,s}(b)]$, for $m+n \ne 0$, one can write it as the image in the trace of polynomial in dots and tokens on the single cycle of length $|n+m|$, where the strands are oriented up if $n+m>0$ and down if $n+m<0$.

Now suppose that the Frobenius superalgebra $A$ is nontrivially positively graded.  Then $\kappa = 0$ and, by \cref{dragon}, we have $[L_{m,r}(a), L_{n,s}(b)] = (rn-sm) L_{m+n,r+s-1}(a \diamond b)$ when $m+n \ne 0$. If we somehow knew that the commutator $[\cL_{m,r}(a), \cL_{n,s}(b)]$ could be written as a single cycle, it would not be hard to see that $[\cL_{m,r}(a), \cL_{n,s}(b)] = (rn-sm) \cL_{m+n,r+s-1}(a \diamond b)$.  The coefficient of $rn-sm$ arises from moving the $r$ dots in $\cL_{m,r}(a)$ past the $n$ strands in $\cL_{n,s}(b)$ and the $s$ dots in $\cL_{n,s}(b)$ past the $m$ strands in $\cL_{m,r}(a)$, using \cref{upslides1,upslides2}.  Then, by degree considerations, we know that the commutator cannot involve the element $\cL_{m+n,t}(c)$ for any $t < r+s-1$, $c \in C(A)$.  Thus, we can rearrange the dots on the $(m+n)$-cycle at will, knowing that all error terms (the terms coming from the teleporter in \cref{upslides1,upslides2}) must cancel.  In this way we obtain $(rn-sm) \cL_{m+n,r+s-1}(a \diamond b)$.

\subsection{Some diagrammatic commutation relations\label{compute}}

As evidence towards \cref{hope}, we compute here some diagrammatic commutation relations.  As noted in \cref{special}, these are sufficient to prove \cref{hope} in the case that $k=\pm 1$ and $\diamond$ has a unit element.  To keep the computations as uncluttered as possible, we will omit the angled brackets indicating the image of an element in the trace.

The first computation is a diagrammatic analogue of \cref{comm4}.

\begin{lem}
    For $n \in \Z$, $r \in \N$ and $a,b \in A$, we have
    \begin{equation} \label{snow}
        [\cL_{0,2}(a), \cL_{n,r}(b)]
        = 2 n \cL_{n,r+1}(a \diamond b).
    \end{equation}
\end{lem}

\begin{proof}
    The result is clear if $n=0$, since then $\hcL_{0,2}(a), \hcL_{n,r}(b) \in \End_{\Heis_k(A)}(\one)$, which is supercommutative.  Now suppose $n >0$.  A straightforward computation using the bubble slide relations \cref{bubslide} shows that
    \[
        \hcL_{0,2}(a)\
        \begin{tikzpicture}[centerbase]
            \draw[->] (0,-0.3) to (0,0.3);
            \token{west}{0,0}{b};
        \end{tikzpicture}
        = (-1)^{\bar{a} \bar{b}}\
        \begin{tikzpicture}[centerbase]
            \draw[->] (0,-0.3) to (0,0.3);
            \token{west}{0,0}{b};
        \end{tikzpicture}
        \hcL_{0,2}(a)
        + 2\
        \begin{tikzpicture}[centerbase]
            \draw[->] (0,-0.3) to (0,0.3);
            \token{west}{0,0.1}{a \diamond b};
            \singdot{0,-0.1};
        \end{tikzpicture}.
    \]
    Then \cref{snow} follows by using this relation to commute $\hcL_{0,2}(a)$ past each strand appearing in $\hcL_{n,r}(b)$.  The case $n < 0$ is analogous.
\end{proof}

The next result should be compared to \cref{comm5,studded}.

\begin{prop} \label{fin}
    For $r,s \ge 0$, in $\Tr(\Heis_k(A))$ we have
    \[
        [\cL_{1,r}(a), \cL_{-1,s}(b)]
        =
        \begin{tikzpicture}[anchorbase]
            \draw[->] (-0.2,-0.5) to (-0.2,0.5);
            \draw[<-] (0.2,-0.5) to (0.2,0.5);
            \multdot{east}{-0.2,-0.2}{r};
            \token{east}{-0.2,0.2}{a};
            \multdot{west}{0.2,-0.2}{s};
            \token{west}{0.2,0.2}{b};
        \end{tikzpicture}
        - (-1)^{\bar{a} \bar{b}}
        \begin{tikzpicture}[anchorbase]
            \draw[<-] (-0.2,-0.5) to (-0.2,0.5);
            \draw[->] (0.2,-0.5) to (0.2,0.5);
            \multdot{west}{0.2,-0.2}{r};
            \multdot{east}{-0.2,-0.2}{s};
            \token{east}{-0.2,0.2}{b};
            \token{west}{0.2,0.2}{a};
        \end{tikzpicture}
        = \beta_k \big( p_{r+s}(a \diamond b) \big) + \delta_{r+s,0} k \tr(a \diamond b).
    \]
\end{prop}

\begin{proof}
    In $\Tr(\Heis_k(A))$, we compute
    \begin{align*}
        \begin{tikzpicture}[anchorbase]
            \draw[->] (-0.2,-0.5) to (-0.2,0.5);
            \draw[<-] (0.2,-0.5) to (0.2,0.5);
            \multdot{east}{-0.2,-0.2}{r};
            \token{east}{-0.2,0.2}{a};
            \multdot{west}{0.2,-0.2}{s};
            \token{west}{0.2,0.2}{b};
        \end{tikzpicture}
        &\overset{\mathclap{\cref{squish}}}{=}\
        \begin{tikzpicture}[anchorbase]
            \draw[->] (-0.2,-0.6) to (-0.2,-0.4) \braidto (0.2,0) \braidto (-0.2,0.4) to (-0.2,0.6);
            \draw[<-] (0.2,-0.6) to (0.2,-0.4) \braidto (-0.2,0) \braidto (0.2,0.4) to (0.2,0.6);
            \multdot{east}{-0.2,-0.4}{r};
            \multdot{west}{0.2,-0.4}{s};
            \token{east}{-0.2,0.4}{a};
            \token{west}{0.2,0.4}{b};
        \end{tikzpicture}
        - \sum_{t,u \ge 0}
        \begin{tikzpicture}[anchorbase]
            \draw[->] (-0.4,0.3) arc(90:0:0.2) to (-0.2,-0.1) arc(0:-180:0.2) to (-0.6,0.1) arc(180:90:0.2);
            \draw[->] (0.4,0.3) arc(90:180:0.2) to (0.2,-0.1) arc(180:360:0.2) to (0.6,0.1) arc(0:90:0.2);
            \teleport{-0.2,0.1}{0.2,0.1};
            \teleport{-0.2,-0.1}{0.2,-0.1};
            \token{east}{-0.6,0.1}{a};
            \multdot{east}{-0.6,-0.1}{r+s+t+u};
            \token{west}{0.6,0.1}{b};
            \multdot{west}{0.6,-0.1}{-t-u-2};
        \end{tikzpicture}
        \overset{\cref{tokteleport}}{\underset{\cref{laser}}{=}}
        \begin{tikzpicture}[anchorbase]
            \draw[->] (-0.2,-0.5) to (-0.2,-0.3) \braidto (0.2,0.1) \braidto (-0.2,0.5);
            \draw[<-] (0.2,-0.5) to (0.2,-0.3) \braidto (-0.2,0.1) \braidto (0.2,0.5);
            \multdot{east}{-0.2,-0.3}{r};
            \multdot{west}{0.2,-0.3}{s};
        \end{tikzpicture}
        - \sum_{t \ge r+1} (t-r)
        \begin{tikzpicture}[anchorbase]
            \draw[->] (-0.4,0.2) arc(90:-270:0.2);
            \draw[->] (0.4,0.2) arc(90:450:0.2);
            \teleport{-0.2,0}{0.2,0};
            \multdot{north}{-0.4,-0.2}{s+t-1};
            \multdot{north}{0.4,-0.2}{r-t-1};
            \token{west}{0.6,0}{a \diamond b};
         \end{tikzpicture}
         \\
         &\overset{\mathclap{\cref{upslides1}}}{=}\
         \begin{tikzpicture}[anchorbase]
            \draw[->] (-0.2,-0.6) to (-0.2,-0.4) \braidto (0.2,0) \braidto (-0.2,0.4) to (-0.2,0.6);
            \draw[<-] (0.2,-0.6) to (0.2,-0.4) \braidto (-0.2,0) \braidto (0.2,0.4) to (0.2,0.6);
            \multdot{west}{0.2,0}{r};
            \multdot{west}{0.2,-0.4}{s};
            \token{east}{-0.2,0.4}{a};
            \token{west}{0.2,0.4}{b};
        \end{tikzpicture}
        - \sum_{u=0}^{r-1}
        \begin{tikzpicture}[anchorbase]
            \draw[->] (0,0.5) arc(90:0:0.2) to[out=down,in=up] (-0.2,-0.3) arc(180:360:0.2) to[out=up,in=down] (-0.2,0.3) arc(180:90:0.2);
            \multdot{west}{0.2,0.3}{s+u};
            \multdot{west}{0.2,-0.3}{r-1-u};
            \teleport{0.15,0.15}{0.15,-0.15};
            \token{east}{-0.2,0.3}{a};
            \token{east}{-0.2,-0.3}{b};
        \end{tikzpicture}
        - \sum_{t \ge r+1} (t-r)
        \begin{tikzpicture}[anchorbase]
            \draw[->] (-0.4,0.2) arc(90:-270:0.2);
            \draw[->] (0.4,0.2) arc(90:450:0.2);
            \teleport{-0.2,0}{0.2,0};
            \multdot{north}{-0.4,-0.2}{s+t-1};
            \multdot{north}{0.4,-0.2}{r-t-1};
            \token{west}{0.6,0}{a \diamond b};
         \end{tikzpicture}
         \\
         &\overset{\mathclap{\cref{curls}}}{=}\
         \begin{tikzpicture}[anchorbase]
            \draw[->] (-0.2,-0.6) to (-0.2,-0.4) \braidto (0.2,0) \braidto (-0.2,0.4) to (-0.2,0.6);
            \draw[<-] (0.2,-0.6) to (0.2,-0.4) \braidto (-0.2,0) \braidto (0.2,0.4) to (0.2,0.6);
            \multdot{west}{0.2,0}{r};
            \multdot{west}{0.2,-0.4}{s};
            \token{east}{-0.2,0.4}{a};
            \token{west}{0.2,0.4}{b};
        \end{tikzpicture}
        - \sum_{u=0}^{r-1} \sum_{t \ge 0}
        \begin{tikzpicture}[anchorbase]
            \draw[->] (-0.4,0.3) arc(90:0:0.2) to (-0.2,-0.1) arc(0:-180:0.2) to (-0.6,0.1) arc(180:90:0.2);
            \draw[->] (0.4,0.3) arc(90:180:0.2) to (0.2,-0.1) arc(180:360:0.2) to (0.6,0.1) arc(0:90:0.2);
            \teleport{-0.2,0.1}{0.2,0.1};
            \teleport{-0.2,-0.1}{0.2,-0.1};
            \multdot{north}{-0.4,-0.3}{s+t+u};
            \multdot{west}{0.6,-0.1}{r-t-u-2};
            \token{east}{-0.6,0.1}{a};
            \token{west}{0.6,0.1}{b};
         \end{tikzpicture}
        - \sum_{t \ge r+1} (t-r)
        \begin{tikzpicture}[anchorbase]
            \draw[->] (-0.4,0.2) arc(90:-270:0.2);
            \draw[->] (0.4,0.2) arc(90:450:0.2);
            \teleport{-0.2,0}{0.2,0};
            \multdot{north}{-0.4,-0.2}{s+t-1};
            \multdot{north}{0.4,-0.2}{r-t-1};
            \token{west}{0.6,0}{a \diamond b};
         \end{tikzpicture}
         \\
         &\overset{\mathclap{\cref{tokteleport}}}{\underset{\mathclap{\cref{laser}}}{=}}\
         \begin{tikzpicture}[anchorbase]
            \draw[->] (-0.2,-0.6) to (-0.2,-0.4) \braidto (0.2,0) \braidto (-0.2,0.4) to (-0.2,0.6);
            \draw[<-] (0.2,-0.6) to (0.2,-0.4) \braidto (-0.2,0) \braidto (0.2,0.4) to (0.2,0.6);
            \multdot{west}{0.2,0}{r};
            \multdot{west}{0.2,-0.4}{s};
            \token{east}{-0.2,0.4}{a};
            \token{west}{0.2,0.4}{b};
        \end{tikzpicture}
        - \sum_{u=0}^{r-1} \sum_{t \ge u+1}
        \begin{tikzpicture}[anchorbase]
            \draw[->] (-0.4,0.2) arc(90:-270:0.2);
            \draw[->] (0.4,0.2) arc(90:450:0.2);
            \teleport{-0.2,0}{0.2,0};
            \multdot{north}{-0.4,-0.2}{s+t-1};
            \multdot{north}{0.4,-0.2}{r-t-1};
            \token{west}{0.6,0}{a \diamond b};
        \end{tikzpicture}
        - \sum_{t \ge r+1} (t-r)
        \begin{tikzpicture}[anchorbase]
            \draw[->] (-0.4,0.2) arc(90:-270:0.2);
            \draw[->] (0.4,0.2) arc(90:450:0.2);
            \teleport{-0.2,0}{0.2,0};
            \multdot{north}{-0.4,-0.2}{s+t-1};
            \multdot{north}{0.4,-0.2}{r-t-1};
            \token{west}{0.6,0}{a \diamond b};
         \end{tikzpicture}
         \\
         &=
         \begin{tikzpicture}[anchorbase]
            \draw[->] (-0.2,-0.6) to (-0.2,-0.4) \braidto (0.2,0) \braidto (-0.2,0.4) to (-0.2,0.6);
            \draw[<-] (0.2,-0.6) to (0.2,-0.4) \braidto (-0.2,0) \braidto (0.2,0.4) to (0.2,0.6);
            \multdot{west}{0.2,0}{r};
            \multdot{west}{0.2,-0.4}{s};
            \token{east}{-0.2,0.4}{a};
            \token{west}{0.2,0.4}{b};
        \end{tikzpicture}
        - \sum_{t \ge 1} t
        \begin{tikzpicture}[anchorbase]
            \draw[->] (-0.4,0.2) arc(90:-270:0.2);
            \draw[->] (0.4,0.2) arc(90:450:0.2);
            \teleport{-0.2,0}{0.2,0};
            \multdot{north}{-0.4,-0.2}{s+t-1};
            \multdot{north}{0.4,-0.2}{r-t-1};
            \token{west}{0.6,0}{a \diamond b};
        \end{tikzpicture}
        \\
        &\overset{\mathclap{\cref{upslides1}}}{=}\
        \begin{tikzpicture}[anchorbase]
            \draw[->] (-0.2,-0.6) to (-0.2,-0.4) \braidto (0.2,0) \braidto (-0.2,0.4) to (-0.2,0.6);
            \draw[<-] (0.2,-0.6) to (0.2,-0.4) \braidto (-0.2,0) \braidto (0.2,0.4) to (0.2,0.6);
            \multdot{west}{0.2,0}{r};
            \multdot{east}{-0.2,0}{s};
            \token{east}{-0.2,0.4}{a};
            \token{west}{0.2,0.4}{b};
        \end{tikzpicture}
        - \sum_{u=0}^{s-1}
        \begin{tikzpicture}[anchorbase]
            \draw[->] (0,0.5) arc(90:0:0.2) to[out=down,in=up] (-0.2,-0.3) arc(180:360:0.2) to[out=up,in=down] (-0.2,0.3) arc(180:90:0.2);
            \multdot{west}{0.2,0.3}{s-1-u};
            \multdot{west}{0.2,-0.3}{r+u};
            \teleport{0.15,0.15}{0.15,-0.15};
            \token{east}{-0.2,0.3}{a};
            \token{east}{-0.2,-0.3}{b};
        \end{tikzpicture}
        - \sum_{t \ge 1} t
        \begin{tikzpicture}[anchorbase]
            \draw[->] (-0.4,0.2) arc(90:-270:0.2);
            \draw[->] (0.4,0.2) arc(90:450:0.2);
            \teleport{-0.2,0}{0.2,0};
            \multdot{north}{-0.4,-0.2}{s+t-1};
            \multdot{north}{0.4,-0.2}{r-t-1};
            \token{west}{0.6,0}{a \diamond b};
        \end{tikzpicture}
        \\
        &\overset{\mathclap{\cref{curls}}}{=}\
        \begin{tikzpicture}[anchorbase]
            \draw[->] (-0.2,-0.6) to (-0.2,-0.4) \braidto (0.2,0) \braidto (-0.2,0.4) to (-0.2,0.6);
            \draw[<-] (0.2,-0.6) to (0.2,-0.4) \braidto (-0.2,0) \braidto (0.2,0.4) to (0.2,0.6);
            \multdot{west}{0.2,0}{r};
            \multdot{east}{-0.2,0}{s};
            \token{east}{-0.2,0.4}{a};
            \token{west}{0.2,0.4}{b};
        \end{tikzpicture}
        + \sum_{u=0}^{s-1} \sum_{t \ge 0}
        \begin{tikzpicture}[anchorbase]
            \draw[->] (-0.4,0.3) arc(90:0:0.2) to (-0.2,-0.1) arc(0:-180:0.2) to (-0.6,0.1) arc(180:90:0.2);
            \draw[->] (0.4,0.3) arc(90:180:0.2) to (0.2,-0.1) arc(180:360:0.2) to (0.6,0.1) arc(0:90:0.2);
            \teleport{-0.2,0.1}{0.2,0.1};
            \teleport{-0.2,-0.1}{0.2,-0.1};
            \multdot{north}{-0.4,-0.3}{s-t-u-2};
            \multdot{west}{0.6,-0.1}{r+t+u};
            \token{east}{-0.6,0.1}{a};
            \token{west}{0.6,0.1}{b};
        \end{tikzpicture}
        - \sum_{t \ge 1} t
        \begin{tikzpicture}[anchorbase]
            \draw[->] (-0.4,0.2) arc(90:-270:0.2);
            \draw[->] (0.4,0.2) arc(90:450:0.2);
            \teleport{-0.2,0}{0.2,0};
            \multdot{north}{-0.4,-0.2}{s+t-1};
            \multdot{north}{0.4,-0.2}{r-t-1};
            \token{west}{0.6,0}{a \diamond b};
        \end{tikzpicture}
        \\
        &\overset{\mathclap{\cref{tokteleport}}}{\underset{\mathclap{\cref{laser}}}{=}}\
        \begin{tikzpicture}[anchorbase]
            \draw[->] (-0.2,-0.6) to (-0.2,-0.4) \braidto (0.2,0) \braidto (-0.2,0.4) to (-0.2,0.6);
            \draw[<-] (0.2,-0.6) to (0.2,-0.4) \braidto (-0.2,0) \braidto (0.2,0.4) to (0.2,0.6);
            \multdot{west}{0.2,0}{r};
            \multdot{east}{-0.2,0}{s};
            \token{east}{-0.2,0.4}{a};
            \token{west}{0.2,0.4}{b};
        \end{tikzpicture}
        + \sum_{u=0}^{s-1} \sum_{t \ge u+1}
        \begin{tikzpicture}[anchorbase]
            \draw[->] (-0.4,0.2) arc(90:-270:0.2);
            \draw[->] (0.4,0.2) arc(90:450:0.2);
            \teleport{-0.2,0}{0.2,0};
            \multdot{north}{-0.4,-0.2}{s-t-1};
            \multdot{north}{0.4,-0.2}{r+t-1};
            \token{west}{0.6,0}{a \diamond b};
        \end{tikzpicture}
        - \sum_{t \ge 1} t
        \begin{tikzpicture}[anchorbase]
            \draw[->] (-0.4,0.2) arc(90:-270:0.2);
            \draw[->] (0.4,0.2) arc(90:450:0.2);
            \teleport{-0.2,0}{0.2,0};
            \multdot{north}{-0.4,-0.2}{s+t-1};
            \multdot{north}{0.4,-0.2}{r-t-1};
            \token{west}{0.6,0}{a \diamond b};
        \end{tikzpicture}
        \\
        &\overset{\mathclap{\cref{squish}}}{=}\ (-1)^{\bar{a} \bar{b}}
        \begin{tikzpicture}[anchorbase]
            \draw[<-] (-0.2,-0.5) to (-0.2,0.5);
            \draw[->] (0.2,-0.5) to (0.2,0.5);
            \multdot{west}{0.2,-0.2}{r};
            \multdot{east}{-0.2,-0.2}{s};
            \token{east}{-0.2,0.2}{b};
            \token{west}{0.2,0.2}{a};
        \end{tikzpicture}
        + \sum_{t,u \ge 0}
        \begin{tikzpicture}[anchorbase]
            \draw[->] (-0.4,0.3) arc(90:0:0.2) to (-0.2,-0.1) arc(0:-180:0.2) to (-0.6,0.1) arc(180:90:0.2);
            \draw[->] (0.4,0.3) arc(90:180:0.2) to (0.2,-0.1) arc(180:360:0.2) to (0.6,0.1) arc(0:90:0.2);
            \teleport{-0.2,0.1}{0.2,0.1};
            \teleport{-0.2,-0.1}{0.2,-0.1};
            \multdot{north}{-0.4,-0.3}{-t-u-2};
            \multdot{west}{0.6,-0.1}{r+s+t+u};
            \token{east}{-0.6,0.1}{a};
            \token{west}{0.6,0.1}{b};
        \end{tikzpicture}
        + \sum_{u=0}^{s-1} \sum_{t \ge u+1}
        \begin{tikzpicture}[anchorbase]
            \draw[->] (-0.4,0.2) arc(90:-270:0.2);
            \draw[->] (0.4,0.2) arc(90:450:0.2);
            \teleport{-0.2,0}{0.2,0};
            \multdot{north}{-0.4,-0.2}{s-t-1};
            \multdot{north}{0.4,-0.2}{r+t-1};
            \token{west}{0.6,0}{a \diamond b};
        \end{tikzpicture}
        - \sum_{t \ge 1} t
        \begin{tikzpicture}[anchorbase]
            \draw[->] (-0.4,0.2) arc(90:-270:0.2);
            \draw[->] (0.4,0.2) arc(90:450:0.2);
            \teleport{-0.2,0}{0.2,0};
            \multdot{north}{-0.4,-0.2}{s+t-1};
            \multdot{north}{0.4,-0.2}{r-t-1};
            \token{west}{0.6,0}{a \diamond b};
        \end{tikzpicture}
        \\
        &= (-1)^{\bar{a} \bar{b}}
        \begin{tikzpicture}[anchorbase]
            \draw[<-] (-0.2,-0.5) to (-0.2,0.5);
            \draw[->] (0.2,-0.5) to (0.2,0.5);
            \multdot{west}{0.2,-0.2}{r};
            \multdot{east}{-0.2,-0.2}{s};
            \token{east}{-0.2,0.2}{b};
            \token{west}{0.2,0.2}{a};
        \end{tikzpicture}
        + \sum_{t \ge s+1} (t-s)
        \begin{tikzpicture}[anchorbase]
            \draw[->] (-0.4,0.2) arc(90:-270:0.2);
            \draw[->] (0.4,0.2) arc(90:450:0.2);
            \teleport{-0.2,0}{0.2,0};
            \multdot{north}{-0.4,-0.2}{s-t-1};
            \multdot{north}{0.4,-0.2}{r+t-1};
            \token{west}{0.6,0}{a \diamond b};
        \end{tikzpicture}
        + \sum_{u=0}^{s-1} \sum_{t \ge u+1}
        \begin{tikzpicture}[anchorbase]
            \draw[->] (-0.4,0.2) arc(90:-270:0.2);
            \draw[->] (0.4,0.2) arc(90:450:0.2);
            \teleport{-0.2,0}{0.2,0};
            \multdot{north}{-0.4,-0.2}{s-t-1};
            \multdot{north}{0.4,-0.2}{r+t-1};
            \token{west}{0.6,0}{a \diamond b};
        \end{tikzpicture}
        - \sum_{t \ge 1} t
        \begin{tikzpicture}[anchorbase]
            \draw[->] (-0.4,0.2) arc(90:-270:0.2);
            \draw[->] (0.4,0.2) arc(90:450:0.2);
            \teleport{-0.2,0}{0.2,0};
            \multdot{north}{-0.4,-0.2}{s+t-1};
            \multdot{north}{0.4,-0.2}{r-t-1};
            \token{west}{0.6,0}{a \diamond b};
        \end{tikzpicture}
        \\
        &= (-1)^{\bar{a} \bar{b}}
        \begin{tikzpicture}[anchorbase]
            \draw[<-] (-0.2,-0.5) to (-0.2,0.5);
            \draw[->] (0.2,-0.5) to (0.2,0.5);
            \multdot{west}{0.2,-0.2}{r};
            \multdot{east}{-0.2,-0.2}{s};
            \token{east}{-0.2,0.2}{b};
            \token{west}{0.2,0.2}{a};
        \end{tikzpicture}
        + \sum_{t \ge 1} t
        \begin{tikzpicture}[anchorbase]
            \draw[->] (-0.4,0.2) arc(90:-270:0.2);
            \draw[->] (0.4,0.2) arc(90:450:0.2);
            \teleport{-0.2,0}{0.2,0};
            \multdot{north}{-0.4,-0.2}{s-t-1};
            \multdot{north}{0.4,-0.2}{r+t-1};
            \token{west}{0.6,0}{a \diamond b};
        \end{tikzpicture}
        - \sum_{t \ge 1} t
        \begin{tikzpicture}[anchorbase]
            \draw[->] (-0.4,0.2) arc(90:-270:0.2);
            \draw[->] (0.4,0.2) arc(90:450:0.2);
            \teleport{-0.2,0}{0.2,0};
            \multdot{north}{-0.4,-0.2}{s+t-1};
            \multdot{north}{0.4,-0.2}{r-t-1};
            \token{west}{0.6,0}{a \diamond b};
        \end{tikzpicture}
        \\
        &= (-1)^{\bar{a} \bar{b}}
        \begin{tikzpicture}[anchorbase]
            \draw[<-] (-0.2,-0.5) to (-0.2,0.5);
            \draw[->] (0.2,-0.5) to (0.2,0.5);
            \multdot{west}{0.2,-0.2}{r};
            \multdot{east}{-0.2,-0.2}{s};
            \token{east}{-0.2,0.2}{b};
            \token{west}{0.2,0.2}{a};
        \end{tikzpicture}
        + \sum_{t \in \Z} t
        \begin{tikzpicture}[anchorbase]
            \draw[->] (-0.4,0.2) arc(90:-270:0.2);
            \draw[->] (0.4,0.2) arc(90:450:0.2);
            \teleport{-0.2,0}{0.2,0};
            \multdot{north}{-0.4,-0.2}{s-t-1};
            \multdot{north}{0.4,-0.2}{r+t-1};
            \token{west}{0.6,0}{a \diamond b};
        \end{tikzpicture}
        \overset{\cref{weed}}{=} (-1)^{\bar{a} \bar{b}}
        \begin{tikzpicture}[anchorbase]
            \draw[<-] (-0.2,-0.5) to (-0.2,0.5);
            \draw[->] (0.2,-0.5) to (0.2,0.5);
            \multdot{west}{0.2,-0.2}{r};
            \multdot{east}{-0.2,-0.2}{s};
            \token{east}{-0.2,0.2}{b};
            \token{west}{0.2,0.2}{a};
        \end{tikzpicture}
        + \sum_{t = -r-k}^{s-k} t
        \begin{tikzpicture}[anchorbase]
            \draw[->] (-0.4,0.2) arc(90:-270:0.2);
            \draw[->] (0.4,0.2) arc(90:450:0.2);
            \teleport{-0.2,0}{0.2,0};
            \multdot{north}{-0.4,-0.2}{s-t-1};
            \multdot{north}{0.4,-0.2}{r+t-1};
            \token{west}{0.6,0}{a \diamond b};
        \end{tikzpicture}
        \\
        &\overset{\mathclap{\cref{tokteleport}}}{=}\ (-1)^{\bar{a} \bar{b}}
        \begin{tikzpicture}[anchorbase]
            \draw[<-] (-0.2,-0.5) to (-0.2,0.5);
            \draw[->] (0.2,-0.5) to (0.2,0.5);
            \multdot{west}{0.2,-0.2}{r};
            \multdot{east}{-0.2,-0.2}{s};
            \token{east}{-0.2,0.2}{b};
            \token{west}{0.2,0.2}{a};
        \end{tikzpicture}
        + \sum_{t=0}^{r+s} (t-r-k)
        \begin{tikzpicture}[anchorbase]
            \draw[->] (-0.4,0.2) arc(90:-270:0.2);
            \draw[->] (0.4,0.2) arc(90:450:0.2);
            \teleport{-0.2,0}{0.2,0};
            \multdot{north}{-0.4,-0.2}{k-1+r+s-t};
            \multdot{west}{0.6,0}{-k-1+t};
            \token{east}{-0.6,0}{a \diamond b};
        \end{tikzpicture}
        \\
        &\overset{\mathclap{\cref{infgrass}}}{=}\ (-1)^{\bar{a} \bar{b}}
        \begin{tikzpicture}[anchorbase]
            \draw[<-] (-0.2,-0.5) to (-0.2,0.5);
            \draw[->] (0.2,-0.5) to (0.2,0.5);
            \multdot{west}{0.2,-0.2}{r};
            \multdot{east}{-0.2,-0.2}{s};
            \token{east}{-0.2,0.2}{b};
            \token{west}{0.2,0.2}{a};
        \end{tikzpicture}
        + \delta_{r+s,0} k \tr(a \diamond b)
        + \sum_{t=1}^{r+s} t
        \begin{tikzpicture}[anchorbase]
            \draw[->] (-0.4,0.2) arc(90:-270:0.2);
            \draw[->] (0.4,0.2) arc(90:450:0.2);
            \teleport{-0.2,0}{0.2,0};
            \multdot{north}{-0.4,-0.2}{k-1+r+s-t};
            \multdot{west}{0.6,0}{-k-1+t};
            \token{east}{-0.6,0}{a \diamond b};
        \end{tikzpicture}
        \ .
    \end{align*}
    The result then follows from \cref{powerdef}.
\end{proof}

The following result should be compared to \cref{comm3}.

\begin{lem}
    For $m,n \in \N$ and $a,b \in A$, we have
    \begin{equation} \label{VirHeis}
        [\cL_{m,1}(a), \cL_{n,0}(b)] = n \cL_{m+n,0}(a \diamond b).
    \end{equation}
\end{lem}

\begin{proof}
    When $n=0$, \cref{VirHeis} follows immediately from the fact that $\hcL_{0,0}(b)$ is strictly central by \cref{bubslide}.  The bubble slide relations \cref{bubslide} also imply that
    \[
        \hcL_{0,1}(a)\
        \begin{tikzpicture}[centerbase]
            \draw[->] (0,-0.3) to (0,0.3);
            \token{west}{0,0}{b};
        \end{tikzpicture}
        = (-1)^{\bar{a} \bar{b}}\
        \begin{tikzpicture}[centerbase]
            \draw[->] (0,-0.3) to (0,0.3);
            \token{west}{0,0}{b};
        \end{tikzpicture}
        \hcL_{0,1}(a)
        +\
        \begin{tikzpicture}[centerbase]
            \draw[->] (0,-0.3) to (0,0.3);
            \token{west}{0,0}{a \diamond b};
        \end{tikzpicture}.
    \]
    Then \cref{VirHeis} follows for $m=0$, $n \ge 1$, by using the above relation to commute $\hcL_{0,1}(a)$ past each strand in $\hcL_{n,0}(b)$.

    It remains to consider the case $m,n \ge 1$.  In order to make the diagrams easier to read, we draw them in the case $m=4$, $n=3$; the general case is analogous.  In $\Tr(\Heis_k(A))$, we have
    \begin{equation} \label{VHdiag}
        \begin{tikzpicture}[anchorbase]
            \draw[->] (-0.7,-0.3) \braidto (-0.1,0.3) to (-0.1,0.5);
            \draw[->] (-0.5,-0.3) \braidto (-0.7,0.3) to (-0.7,0.5);
            \draw[->] (-0.3,-0.3) \braidto (-0.5,0.3) to (-0.5,0.5);
            \draw[->] (-0.1,-0.3) \braidto (-0.3,0.3) to (-0.3,0.5);
            \singdot{-0.5,0.3};
            \token{west}{-0.1,0.3}{a};
            \draw[->] (0.3,-0.3) \braidto (0.7,0.3) to (0.7,0.5);
            \draw[->] (0.5,-0.3) \braidto (0.3,0.3) to (0.3,0.5);
            \draw[->] (0.7,-0.3) \braidto (0.5,0.3) to (0.5,0.5);
            \token{west}{0.7,0.3}{b};
        \end{tikzpicture}
        =
        \begin{tikzpicture}[anchorbase]
            \draw[->] (0.1,-1.2) \braidto (-0.7,-0.3) \braidto (-0.1,0.3) \braidto (0.7,1.2);
            \draw[->] (0.3,-1.2) \braidto (-0.5,-0.3) \braidto (-0.7,0.3) \braidto (0.1,1.2);
            \draw[->] (0.5,-1.2) \braidto (-0.3,-0.3) \braidto (-0.5,0.3) \braidto (0.3,1.2);
            \draw[->] (0.7,-1.2) \braidto (-0.1,-0.3) \braidto (-0.3,0.3) \braidto (0.5,1.2);
            \singdot{-0.5,0.3};
            \token{west}{-0.1,0.3}{a};
            \draw[->] (-0.7,-1.2) \braidto (0.3,-0.3) \braidto (0.7,0.3) \braidto (-0.3,1.2);
            \draw[->] (-0.5,-1.2) \braidto (0.5,-0.3) \braidto (0.3,0.3) \braidto (-0.7,1.2);
            \draw[->] (-0.3,-1.2) \braidto (0.7,-0.3) \braidto (0.5,0.3) \braidto (-0.5,1.2);
            \token{west}{0.7,0.3}{b};
        \end{tikzpicture}
        \ ,
    \end{equation}
    where the dot is on an arbitrary strand in the $m$-cycle.  We now slide this dot towards the top of the diagram.  As it passes through each of the $n$ crossings, the relation \cref{upslides2} gives an extra term which is an $(m+n)$-cycle with no dots and a token $a \diamond b$.  Relation \cref{VirHeis} follows.
\end{proof}

Finally, note that it follows from \cref{jupiter} that $[L_{m,r}(a), L_{n,s}(b)]=0$ whenever $a \diamond b = 0$.  The diagrammatic analogue of this is the following.

\begin{prop}
    If $m,n \in \Z$, $r,s \in \N$, $a,b \in A$, and $a \diamond b = 0$, then
    \[
        [\cL_{m,r}(a), \cL_{n,s}(b)] = 0.
    \]
\end{prop}

\begin{proof}
    To compute this commutation relation, we begin as in \cref{VHdiag}, except that our diagrams can carry more dots, and our strands can be oriented in opposite directions.  Each time we slide a dot through a crossing, we pick up terms that are a single cycle with a token labeled $a \diamond b$, which is zero.  Thus, dots slide through strands, giving the desired relation.
\end{proof}

\bibliographystyle{alphaurl}
\bibliography{FHtrace}

\newcommand{\etalchar}[1]{$^{#1}$}
\begin{thebibliography}{BHLW17}

\bibitem[BCNR17]{BCNR17}
J.~Brundan, J.~Comes, D.~Nash, and A.~Reynolds.
\newblock A basis theorem for the affine oriented {B}rauer category and its
  cyclotomic quotients.
\newblock {\em Quantum Topol.}, 8(1):75--112, 2017.
\newblock \href {http://arxiv.org/abs/1404.6574} {\path{arXiv:1404.6574}},
  \href {https://doi.org/10.4171/QT/87} {\path{doi:10.4171/QT/87}}.

\bibitem[BE17]{BE17}
J.~Brundan and A.~P. Ellis.
\newblock Monoidal supercategories.
\newblock {\em Comm. Math. Phys.}, 351(3):1045--1089, 2017.
\newblock \href {http://arxiv.org/abs/1603.05928} {\path{arXiv:1603.05928}},
  \href {https://doi.org/10.1007/s00220-017-2850-9}
  {\path{doi:10.1007/s00220-017-2850-9}}.

\bibitem[BGHL14]{BGHL14}
A.~Beliakova, Z.~Guliyev, K.~Habiro, and A.~D. Lauda.
\newblock Trace as an alternative decategorification functor.
\newblock {\em Acta Math. Vietnam.}, 39(4):425--480, 2014.
\newblock \href {http://arxiv.org/abs/1409.1198} {\path{arXiv:1409.1198}},
  \href {https://doi.org/10.1007/s40306-014-0092-x}
  {\path{doi:10.1007/s40306-014-0092-x}}.

\bibitem[BHLW17]{BHLW17}
A.~Beliakova, K.~Habiro, A.~D. Lauda, and B.~Webster.
\newblock Current algebras and categorified quantum groups.
\newblock {\em J. Lond. Math. Soc. (2)}, 95(1):248--276, 2017.
\newblock \href {http://arxiv.org/abs/1412.1417} {\path{arXiv:1412.1417}},
  \href {https://doi.org/10.1112/jlms.12001} {\path{doi:10.1112/jlms.12001}}.

\bibitem[Bru18]{Bru18}
J.~Brundan.
\newblock On the definition of {H}eisenberg category.
\newblock {\em Algebr. Comb.}, 1(4):523--544, 2018.
\newblock \href {http://arxiv.org/abs/1709.06589} {\path{arXiv:1709.06589}},
  \href {https://doi.org/10.5802/alco.26} {\path{doi:10.5802/alco.26}}.

\bibitem[BSW18]{BSW-K0}
J.~Brundan, A.~Savage, and B.~Webster.
\newblock The degenerate {H}eisenberg category and its {G}rothendieck ring.
\newblock 2018.
\newblock \href {http://arxiv.org/abs/1812.03255} {\path{arXiv:1812.03255}}.

\bibitem[BSW20a]{BSW20}
J.~Brundan, A.~Savage, and B.~Webster.
\newblock Foundations of {F}robenius {H}eisenberg categories.
\newblock 2020.
\newblock \href {http://arxiv.org/abs/2007.01642} {\path{arXiv:2007.01642}}.

\bibitem[BSW20b]{BSW-qheis}
J.~Brundan, A.~Savage, and B.~Webster.
\newblock On the definition of quantum {H}eisenberg category.
\newblock {\em Algebra Number Theory}, 14(2):275--321, 2020.
\newblock \href {http://arxiv.org/abs/1812.04779} {\path{arXiv:1812.04779}},
  \href {https://doi.org/10.2140/ant.2020.14.275}
  {\path{doi:10.2140/ant.2020.14.275}}.

\bibitem[BSW20c]{BSW-qFrobHeis}
J.~Brundan, A.~Savage, and B.~Webster.
\newblock Quantum {F}robenius {H}eisenberg categorification.
\newblock 2020.
\newblock \href {http://arxiv.org/abs/2009.06690} {\path{arXiv:2009.06690}}.

\bibitem[CL12]{CL12}
S.~Cautis and A.~Licata.
\newblock Heisenberg categorification and {H}ilbert schemes.
\newblock {\em Duke Math. J.}, 161(13):2469--2547, 2012.
\newblock \href {http://arxiv.org/abs/1009.5147} {\path{arXiv:1009.5147}},
  \href {https://doi.org/10.1215/00127094-1812726}
  {\path{doi:10.1215/00127094-1812726}}.

\bibitem[CLL{\etalchar{+}}18]{CLLSS18}
S.~Cautis, A.~D. Lauda, A.~M. Licata, P.~Samuelson, and J.~Sussan.
\newblock The elliptic {H}all algebra and the deformed {K}hovanov {H}eisenberg
  category.
\newblock {\em Selecta Math. (N.S.)}, 24(5):4041--4103, 2018.
\newblock \href {http://arxiv.org/abs/1609.03506} {\path{arXiv:1609.03506}},
  \href {https://doi.org/10.1007/s00029-018-0429-8}
  {\path{doi:10.1007/s00029-018-0429-8}}.

\bibitem[CLLS18]{CLLS18}
S.~Cautis, A.~D. Lauda, A.~M. Licata, and J.~Sussan.
\newblock W-algebras from {H}eisenberg categories.
\newblock {\em J. Inst. Math. Jussieu}, 17(5):981--1017, 2018.
\newblock \href {http://arxiv.org/abs/1501.00589} {\path{arXiv:1501.00589}},
  \href {https://doi.org/10.1017/S1474748016000189}
  {\path{doi:10.1017/S1474748016000189}}.

\bibitem[Cou16]{Cou16}
C.~Couture.
\newblock Skew-zigzag algebras.
\newblock {\em SIGMA Symmetry Integrability Geom. Methods Appl.}, 12:Paper No.
  062, 19, 2016.
\newblock \href {http://arxiv.org/abs/1509.08405} {\path{arXiv:1509.08405}},
  \href {https://doi.org/10.3842/SIGMA.2016.062}
  {\path{doi:10.3842/SIGMA.2016.062}}.

\bibitem[HK01]{HK01}
R.~S. Huerfano and M.~Khovanov.
\newblock A category for the adjoint representation.
\newblock {\em J. Algebra}, 246(2):514--542, 2001.
\newblock \href {http://arxiv.org/abs/math/0002060}
  {\path{arXiv:math/0002060}}, \href {https://doi.org/10.1006/jabr.2001.8962}
  {\path{doi:10.1006/jabr.2001.8962}}.

\bibitem[Kho14]{Kho11}
M.~Khovanov.
\newblock Heisenberg algebra and a graphical calculus.
\newblock {\em Fund. Math.}, 225(1):169--210, 2014.
\newblock \href {http://arxiv.org/abs/1009.3295} {\path{arXiv:1009.3295}},
  \href {https://doi.org/10.4064/fm225-1-8} {\path{doi:10.4064/fm225-1-8}}.

\bibitem[KR93]{KR93}
V.~Kac and A.~Radul.
\newblock Quasifinite highest weight modules over the {L}ie algebra of
  differential operators on the circle.
\newblock {\em Comm. Math. Phys.}, 157(3):429--457, 1993.
\newblock URL: \url{http://projecteuclid.org/euclid.cmp/1104254017}.

\bibitem[KWY98]{KWY98}
V.~G. Kac, W.~Wang, and C.~H. Yan.
\newblock Quasifinite representations of classical {L}ie subalgebras of
  {$\mathcal{W}_{1+\infty}$}.
\newblock {\em Adv. Math.}, 139(1):56--140, 1998.
\newblock \href {http://arxiv.org/abs/math/9801136}
  {\path{arXiv:math/9801136}}, \href {https://doi.org/10.1006/aima.1998.1753}
  {\path{doi:10.1006/aima.1998.1753}}.

\bibitem[LRS18]{LRS18}
A.~Licata, D.~Rosso, and A.~Savage.
\newblock A graphical calculus for the {J}ack inner product on symmetric
  functions.
\newblock {\em J. Combin. Theory Ser. A}, 155:503--543, 2018.
\newblock \href {http://arxiv.org/abs/1610.01862} {\path{arXiv:1610.01862}},
  \href {https://doi.org/10.1016/j.jcta.2017.11.020}
  {\path{doi:10.1016/j.jcta.2017.11.020}}.

\bibitem[LS13]{LS13}
A.~Licata and A.~Savage.
\newblock Hecke algebras, finite general linear groups, and {H}eisenberg
  categorification.
\newblock {\em Quantum Topol.}, 4(2):125--185, 2013.
\newblock \href {http://arxiv.org/abs/1101.0420} {\path{arXiv:1101.0420}},
  \href {https://doi.org/10.4171/QT/37} {\path{doi:10.4171/QT/37}}.

\bibitem[Mac95]{Mac95}
I.~G. Macdonald.
\newblock {\em Symmetric functions and {H}all polynomials}.
\newblock Oxford Mathematical Monographs. The Clarendon Press, Oxford
  University Press, New York, second edition, 1995.
\newblock With contributions by A. Zelevinsky, Oxford Science Publications.

\bibitem[MS18]{MS18}
M.~Mackaay and A.~Savage.
\newblock Degenerate cyclotomic {H}ecke algebras and higher level {H}eisenberg
  categorification.
\newblock {\em J. Algebra}, 505:150--193, 2018.
\newblock \href {http://arxiv.org/abs/1705.03066} {\path{arXiv:1705.03066}},
  \href {https://doi.org/10.1016/j.jalgebra.2018.03.004}
  {\path{doi:10.1016/j.jalgebra.2018.03.004}}.

\bibitem[OR17]{OR17}
C.~O. O\u{g}uz and M.~Reeks.
\newblock Trace of the twisted {H}eisenberg category.
\newblock {\em Comm. Math. Phys.}, 356(3):1117--1154, 2017.
\newblock \href {http://arxiv.org/abs/1702.08108} {\path{arXiv:1702.08108}},
  \href {https://doi.org/10.1007/s00220-017-2992-9}
  {\path{doi:10.1007/s00220-017-2992-9}}.

\bibitem[RS17]{RS17}
D.~Rosso and A.~Savage.
\newblock A general approach to {H}eisenberg categorification via wreath
  product algebras.
\newblock {\em Math. Z.}, 286(1-2):603--655, 2017.
\newblock \href {http://arxiv.org/abs/1507.06298} {\path{arXiv:1507.06298}},
  \href {https://doi.org/10.1007/s00209-016-1776-9}
  {\path{doi:10.1007/s00209-016-1776-9}}.

\bibitem[RS20]{RS19}
D.~Rosso and A.~Savage.
\newblock Quantum affine wreath algebras.
\newblock {\em Doc. Math.}, 25:425--456, 2020.
\newblock \href {http://arxiv.org/abs/1902.00143} {\path{arXiv:1902.00143}},
  \href {https://doi.org/10.25537/dm.2020v25.425-456}
  {\path{doi:10.25537/dm.2020v25.425-456}}.

\bibitem[Sav]{Sav-exp}
A.~Savage.
\newblock String diagrams and categorification.
\newblock In {\em Interactions of Quantum Affine Algebras with Cluster
  Algebras, Current Algebras and Categorification}, Progr. Math.
\newblock To appear.
\newblock \href {http://arxiv.org/abs/1806.06873} {\path{arXiv:1806.06873}}.

\bibitem[Sav19]{Sav19}
A.~Savage.
\newblock Frobenius {H}eisenberg categorification.
\newblock {\em Algebr. Comb.}, 2(5):937--967, 2019.
\newblock \href {http://arxiv.org/abs/1802.01626} {\path{arXiv:1802.01626}},
  \href {https://doi.org/10.5802/alco.73} {\path{doi:10.5802/alco.73}}.

\bibitem[Sav20]{Sav20}
A.~Savage.
\newblock Affine wreath product algebras.
\newblock {\em Int. Math. Res. Not. IMRN}, (10):2977--3041, 2020.
\newblock \href {http://arxiv.org/abs/1709.02998} {\path{arXiv:1709.02998}},
  \href {https://doi.org/10.1093/imrn/rny092} {\path{doi:10.1093/imrn/rny092}}.

\bibitem[Sol10]{Sol10}
M.~Solleveld.
\newblock Homology of graded {H}ecke algebras.
\newblock {\em J. Algebra}, 323(6):1622--1648, 2010.
\newblock \href {https://doi.org/10.1016/j.jalgebra.2010.01.011}
  {\path{doi:10.1016/j.jalgebra.2010.01.011}}.

\bibitem[SV13]{SV13}
O.~Schiffmann and E.~Vasserot.
\newblock Cherednik algebras, {W}-algebras and the equivariant cohomology of
  the moduli space of instantons on {$\bold{A}^2$}.
\newblock {\em Publ. Math. Inst. Hautes \'{E}tudes Sci.}, 118:213--342, 2013.
\newblock \href {http://arxiv.org/abs/1202.2756} {\path{arXiv:1202.2756}},
  \href {https://doi.org/10.1007/s10240-013-0052-3}
  {\path{doi:10.1007/s10240-013-0052-3}}.

\bibitem[TV17]{TV17}
V.~Turaev and A.~Virelizier.
\newblock {\em Monoidal categories and topological field theory}, volume 322 of
  {\em Progress in Mathematics}.
\newblock Birkh\"{a}user/Springer, Cham, 2017.
\newblock \href {https://doi.org/10.1007/978-3-319-49834-8}
  {\path{doi:10.1007/978-3-319-49834-8}}.

\bibitem[Wan04]{Wan04}
W.~Wang.
\newblock Vertex algebras and the class algebras of wreath products.
\newblock {\em Proc. London Math. Soc. (3)}, 88(2):381--404, 2004.
\newblock \href {http://arxiv.org/abs/math/0203004}
  {\path{arXiv:math/0203004}}, \href
  {https://doi.org/10.1112/S0024611503014382}
  {\path{doi:10.1112/S0024611503014382}}.

\end{thebibliography}

\end{document}